\documentclass{article}
\usepackage[english]{babel}
\usepackage[latin1]{inputenc}
\usepackage{t1enc}
\usepackage{graphicx}
\usepackage{amssymb}
\usepackage{amsmath}
\usepackage{amsthm}
\usepackage{enumerate}
\usepackage{color,xcolor}
\newtheorem{theorem}{Theorem}[section]
\newtheorem{lemma}[theorem]{Lemma}
\newtheorem{corollary}[theorem]{Corollary}
\newtheorem{proposition}[theorem]{Proposition}
\theoremstyle{definition}
\newtheorem{remark}[theorem]{Remark}
\newtheorem{example}[theorem]{Example}
\newtheorem{examples}[theorem]{Examples}
\newtheorem{definition}[theorem]{Definition}
\numberwithin{equation}{section}

\usepackage{citeref}

\allowdisplaybreaks[1]

\numberwithin{equation}{section}

\newcommand{\nd}{\red{\ensuremath d}} %%% dimension variable in this file: d
\renewcommand{\nd}{{\ensuremath d}} %%% dimension variable in this file: d

\definecolor{verde}{rgb}{0.32,0.82,0.38}

\newcommand{\B}{\ensuremath{B^s_{p,q}}}

\newcommand{\dint}{\ensuremath{\mathrm{d}}}
\newcommand{\bit}{\begin{itemize}}
\newcommand{\eit}{\end{itemize}}
\newcommand{\beq}{\begin{equation}}
\newcommand{\eeq}{\end{equation}}
\newcommand{\supp}{\mathrm{supp}\,}

\newcommand{\id}{\mathrm{id}}
\newcommand{\hra}{\hookrightarrow}
\def\ls{\lesssim}
\newcommand{\Mf}{{\mathcal M}_{\varphi,p}}

\newcommand{\MfB}{{\mathcal N}^{s}_{\varphi,p,q}}
\newcommand{\MfF}{{\mathcal E}^{s}_{\varphi,p,q}}
\newcommand{\MFa}{{\mathcal E}^{s_1}_{u_1,p_1,q_1}}
\newcommand{\MFb}{{\mathcal E}^{s_2}_{u_2,p_2,q_2}}
\newcommand{\MAa}{{\mathcal A}^{s_1}_{u_1,p_1,q_1}}
\newcommand{\MAb}{{\mathcal A}^{s_2}_{u_2,p_2,q_2}}

\newcommand{\nat}{\ensuremath{\mathbb{N}}}
\newcommand{\no}{\ensuremath{\nat_0}}
\newcommand{\rr}{\ensuremath{{\mathbb R}}}
\newcommand{\real}{\ensuremath{{\mathbb R}}}
\newcommand{\rd}{\ensuremath{{\mathbb R}^\nd}}
\newcommand{\cc}{\ensuremath{{\mathbb C}}}
\newcommand{\zz}{\ensuremath{{\mathbb Z}}}
\newcommand{\zd}{\ensuremath{\mathbb{Z}^\nd}}

\newcommand{\M}{{\mathcal M}_{\varphi,p}}
\newcommand{\Ma}{{\mathcal M}_{\varphi_1,p_1}}
\newcommand{\Mb}{{\mathcal M}_{\varphi_2,p_2}}
\newcommand{\Mu}{{\mathcal M}_{u,p}}
\newcommand{\Me}{{\mathcal M}_{u_1,p_1}}
\newcommand{\Mz}{{\mathcal M}_{u_2,p_2}}
\newcommand{\MA}{{\mathcal A}^{s}_{\varphi,p,q}}
\newcommand{\MB}{{\mathcal N}^{s}_{\varphi,p,q}}
\newcommand{\MF}{{\mathcal E}^{s}_{\varphi,p,q}}
\newcommand{\MAu}{{\mathcal A}^{s}_{u,p,q}}
\newcommand{\MBu}{{\mathcal N}^{s}_{u,p,q}}
\newcommand{\MFu}{{\mathcal E}^{s}_{u,p,q}}
\newcommand{\MBa}{{\mathcal N}^{s_1}_{\varphi_1,p_1,q_1}}
\newcommand{\MBb}{{\mathcal N}^{s_2}_{\varphi_2,p_2,q_2}}

\newcommand{\Gp}{{\mathcal G}_p}
\newcommand{\Gpx}[1]{{\mathcal G}_{#1}}
\newcommand{\Bphi}{B^{s,\varphi}_{p,q}}
\newcommand{\Bphia}{B^{s_1,\varphi_1}_{p_1,q_1}}
\newcommand{\Bphib}{B^{s_2,\varphi_2}_{p_2,q_2}}
\newcommand{\Fphi}{F^{s,\varphi}_{p,q}}
\newcommand{\Fphia}{F^{s_1,\varphi_1}_{p_1,q_1}}
\newcommand{\Fphib}{F^{s_2,\varphi_2}_{p_2,q_2}}
\newcommand{\Aphi}{A^{s,\varphi}_{p,q}}

\newcommand{\n}{{n}^{s}_{\varphi,p,q}}
\newcommand{\na}{{n}^{s_1}_{\varphi_1,p_1,q_1}}
\newcommand{\nb}{{n}^{s_2}_{\varphi_2,p_2,q_2}}

\newcommand{\Bt}{B^{s,\tau}_{p,q}}
\newcommand{\Ft}{F^{s,\tau}_{p,q}}
\newcommand{\At}{A^{s,\tau}_{p,q}}
\newcommand{\Ata}{A^{s_1,\tau_1}_{p_1,q_1}}
\newcommand{\Atb}{A^{s_2,\tau_2}_{p_2,q_2}}
\newcommand{\Bta}{B^{s_1,\tau_1}_{p_1,q_1}}
\newcommand{\Btb}{B^{s_2,\tau_2}_{p_2,q_2}}
\newcommand{\bp}{b^{s,\varphi}_{p,q}}
\newcommand{\bpa}{b^{s_1,\varphi_1}_{p_1,q_1}}
\newcommand{\bpb}{b^{s_2,\varphi_2}_{p_2,q_2}}

\newcommand{\bmo}{\ensuremath \mathrm{bmo}}

\newcommand{\whole}[1]{\ensuremath \lfloor #1 \rfloor}
\newcommand{\up}[1]{\ensuremath  \lceil #1 \rceil}

\newcommand{\magenta}[1]{{\color{magenta}#1}}
\newcommand{\ignore}[1]{}
\newcommand{\ds}{\displaystyle}
\newcommand{\open}[1]{\smallskip\noindent\fbox{\parbox{\textwidth}{\color{blue}\bfseries\begin{center}
				#1 \end{center}}}\\ \smallskip}

\begin{document}

\title{Compact embeddings of generalised Morrey smoothness spaces on bounded domains}

\date{\today}
\author{Dorothee D. Haroske\footnotemark[1], Susana D. Moura\footnotemark[2]  and  Leszek Skrzypczak\footnotemark[1]}

\maketitle

%\ignore{
\footnotetext[1]{{Both authors were partially supported by the German Research Foundation (DFG), Grant No. Ha 2794/10-1.}}
\footnotetext[2]{{The author was partially supported by the Centre for Mathematics of the University of Coimbra under the Portuguese Foundation for Science and Technology (FCT), Grants UID/00324/2025 and UID/PRR/00324/2025.}}
%\footnotetext{ }
%}

\centerline{This article is dedicated to Prof. Hans Triebel on the Occasion of his 90th Birthday.}
\smallskip~

\begin{abstract} 
We study embeddings within different scales of  generalised  smoothness Morrey spaces defined on bounded smooth domains, i.e.,   in $\mathcal{N}^s_{\varphi,p,q}(\Omega)$, $\mathcal{E}^s_{\varphi,p,q}(\Omega)$, $B^{s,\varphi}_{p,q}(\Omega)$ and $F^{s,\varphi}_{p,q}(\Omega)$ spaces. 
	We prove  sufficient conditions for  continuity and compactness of the embeddings. In some cases the conditions are also necessary. We generalise and even improve some earlier results known for the classical smoothness Morrey spaces. Our approach is based on wavelet characterisation of the function spaces.    
\end{abstract}

\medskip

{\bfseries 2020 MSC}: \smallskip 46E35\\


\medskip\medskip

\ignore{
\open{some general questions:
  \begin{itemize}
  \item
    Please check all blue parts, in particular the introduction including its very end where I tried to express our gratitude both to Helena and Hans Triebel. Since all the papers should have the same dedication (as above), any personal remarks should be part of the text, preferably at the end of the introduction or at the very end (before the references). 
  \item
    Sometimes the number $q^\ast$ appears, which is reasonable, of course, for notation like $\ell_{q^\ast}$. But do we want to write in statements (theorems, etc.) rather `if $q_1\leq q_2$' instead of 'if $q^\ast=\infty$'? At the moment we have both, I would prefer (for the statements) the direct (first) form (and have changed it to $q_1\leq q_2$), but I am also fine with the other one. But I would unify it.
  \item
    We sometimes have $\lim_{t\to 0} \varphi(t)$, sometimes $\lim_{t\to 0^+} \varphi(t)$, both is fine with me, but we may unify. What do you prefer?
  \end{itemize}
}
}

\section{Introduction}

  It is well-known, that smoothness function spaces built upon Morrey spaces $\mathcal{M}_{u,p}(\rd)$, $0<p\leq u<\infty$, enjoy some renaissance in the last decades. This refreshed and deepened interest is surely due to the prominent paper by Kozono and Yamazaki \cite{KY} in which smoothness spaces of Besov-Morrey type $\MBu(\rd)$, $s\in\real$, $0<q\leq\infty$,  were used to study Navier-Stokes equations. This observation was followed by Mazzucato's findings \cite{Maz03}, related again to PDEs.  The Triebel-Lizorkin-Morrey spaces
 $\MFu(\rd)$ were later introduced by  Tang and Xu \cite{TX}.

  Another line of research followed the construction of Besov-type spaces $\Bt(\rd)$ and Triebel-Lizorkin-type spaces $\Ft(\rd)$,  $0<p<\infty$, $0<q\leq\infty$, $s\in\real$, $\tau\geq 0$, sharing the same idea of a local-global approach, but with different genesis. These spaces were introduced  and studied systematically in \cite{YSY10} with some forerunners in \cite{ElBaraka1,ElBaraka2, ElBaraka3}.
  Meanwhile there exists a variety of very nice results and a detailed systematic approach to such spaces which can be found, for instance, in the monographs \cite{YSY10,FHS-MS-1,FHS-MS-2,Tri13,Tri14} and the extended survey papers \cite{s011,s011a}. The applications of the results are widespread, cf. \cite{FP,Tri17,YFS}.
  
  In this paper we return to the functional analytic point of view and study generalised Morrey smoothness spaces built, partly, upon generalised Morrey spaces $\M(\rd)$, $0<p<\infty$, $\varphi:(0,\infty)\rightarrow [0,\infty)$. The generalised version of Morrey spaces $\mathcal{M}_{u,p}(\rd)$, $0 < p \le u < \infty$, {was} introduced by  Mizuhara \cite{mi} and Nakai \cite{Nak94}. However, though out of the scope of the present paper, an essential motivation to study such generalisations comes again from the study of PDEs, cf. \cite{FHS,KMR,WNTZ,ZJSZ,KNS}, see also \cite{Saw18} for further information and historical remarks. In a similar approach as for the scales of smoothness spaces of Besov-Morrey spaces $\MBu$ and Triebel-Lizorkin-Morrey spaces $\MFu$ mentioned above, this results in spaces of type $\MB$ and $\MF$. These spaces were introduced and studied by Nakamura, Noi and Sawano  \cite{NNS16}, cf. also \cite{AGNS}. Let us mention, that such generalised Morrey smoothness spaces cover, in particular, also local smoothness Morrey spaces  considered by Triebel \cite{Tri13,Tri14}. Note that in case of $\varphi(t)\sim t^{d/u}$, $0<p\leq u<\infty$, $t>0$, we regain the previous situation, that is, $\MBu(\rd)=\MB(\rd)$ and $\MFu(\rd)=\MF(\rd)$ (in the sense of equivalent norms).  In \cite{HL23,HLMS24} we followed the second approach and generalised spaces $\Bt$ and $\Ft$ to their counterparts $\Bphi$ and $\Fphi$. Here as well, the special setting $\varphi(t)\sim t^{d(\frac1p-\tau)}$, $0\leq \tau<\frac1p$, $t>0$, leads to the former spaces, that is, $\Bt(\rd)=\Bphi(\rd)$ and $\Ft(\rd)=\Fphi(\rd)$   (in the sense of equivalent norms).
    There and in our recent paper \cite{hms22} we obtained embeddings results and prepared further essential tools (extending \cite{NNS16}) like wavelet characterisations.

    In such situations, in the absence of weights or other additional assumptions like symmetries etc., no compactness can be expected from the embeddings within scales of spaces on $\rd$ (as we already know from the special cases). However, if we assume the underlying domain to be bounded and sufficiently smooth, then we can prove compactness.  There is a paper by Izuki and Noi dealing with 
    spaces on domains in \cite{IN19}, but it seems that spaces on domains in the setting of generalised Morrey smoothness scales, have not yet been studied in full generality. For the special case $\varphi(t)\sim t^{d/u}$, $t>0$, we discussed in \cite{HSS-morrey} in some detail already different approaches, various intrinsic characterisations, definitions by restrictions etc. We do not want to repeat this here, but restrict ourselves to the definition by restriction. Our main focus lies on the precise description of the interplay between smoothness parameters and regularity, characterised mainly by $s\in\real$, $0<p<\infty$, and the function $\varphi$ (and to a smaller extent by the fine index $q$ with $0<q\leq\infty$). Looking at our results for the special cases dealt with in \cite{GHS-21,GHS-23,HSS-morrey,hs13,hs14,hs20,hs24} we mainly face the challenge how to `translate' or generalise the technically quite involved conditions for, say, compactness as presented in Corollary~\ref{comp-class} below. How can the characterisations like \eqref{bd3acomp} and \eqref{tau-comp-u2} be expressed in terms of the functions $\varphi_i$, $i=1,2$? The answer is positive and quite interesting in our opinion. The clue in our argument is the introduction of some critical smoothness indices $\sigma$, $\overline{\sigma}$ and $\sigma_\infty$ in Definition~3.16 below. Roughly speaking, these numbers depend on the behaviour of the parameters $p_i\in (0,\infty)$ and the functions $\varphi_i$, $i=1,2$, for a given number $s_1\in\real$. In case of $\varphi_i(t)\sim t^{d(\frac{1}{p_i}-\tau_i)}$, $0\leq \tau_i<1/p_i$, $i=1,2$, and, say, $p_1\leq p_2$, this leads to $\sigma=\sigma(s_1)=s_1-\frac{d}{p_1}+\frac{d}{p_2}+d(\tau_1-\frac{p_1}{p_2}\tau_1)$, and, for $\tau_2 p_2\geq \tau_1 p_1$, to $\overline{\sigma} = \overline{\sigma}(s_1) = s_1-d(\frac{1}{p_1}-\tau_1-\frac{1}{p_2}+\tau_2)$,  which is closely connected to the condition for $s_2$ in \eqref{tau-comp-u2}. And, indeed, we can prove in Theorem~\ref{cor-5.1} below that the embedding $\Bphia(\Omega)\hookrightarrow \Bphib(\Omega)$ is compact if $s_2<\sigma(s_1)$ in case of $p_1<p_2$ and $\varphi_2^{p_2}(t) \leq c\ \varphi_1^{p_1}(t)$, $0<t\leq 1$, while for $\varphi_2^{p_2}(t) \geq c\ \varphi_1^{p_1}(t)$, $0<t\leq 1$, and arbitrary $p_1, p_2$, we need $s_2<\overline{\sigma}(s_1)$ to ensure compactness. For the continuity of that embedding we obtain in case of $p_1\geq p_2$ a complete characterisation in Theorem~\ref{P-Bp-cont} below.
    Returning to the setting of generalised Besov-Morrey spaces on domains, $\MB(\Omega)$, we have a final outcome in Theorem~\ref{th-cont-BM} as follows. If we denote by $\varrho=\min(1, p_1/p_2)$, and $\alpha_j=\sup_{0\leq \nu\leq j}\varphi_2(2^{-\nu}) \varphi_1(2^{-\nu})^{-\varrho}$, $j\in\no$, then
    \[
\MBa(\Omega)\hookrightarrow \MBb(\Omega),    \]
if, and only if,
\[
\{2^{j(s_2-s_1)} \alpha_j \varphi_1(2^{-j})^{\varrho-1}\}_{j\in\no}\in\ell_{q^\ast},
\]
where $q^\ast = \infty$ if $q_1\leq q_2$, and $\frac{1}{q^\ast}= \frac{1}{q_2}-\frac{1}{q_1}$ if $q_1>q_2$. The above embedding is compact, if, and only if, it is continuous in case of $q_1> q_2$, while in case of $q_1\leq q_2$ we need to assume that $2^{j(s_2-s_1)} \alpha_j \varphi_1(2^{-j})^{\varrho-1} \to 0$ for $j\to\infty$. Surprisingly, this new result in the generalised setting even improves our earlier result in \cite{hs13}
to some extent. We can prove a lot more continuity and compactness criteria for embeddings of generalised Morrey smoothness spaces on domains which are, as far as we know, all new in the general setting. The main tool we use is, from the technical point of view, the transfer of the function space question to the equivalent one on related sequence spaces. This is done via a wavelet decomposition of the smoothness spaces defined on $\rd$. For that reason we concentrate in the present paper on spaces on domains defined by restriction only and assume the domain, for convenience, to be sufficiently smooth. Surely it would be quite interesting (but also rather challenging) to consider wavelet characterisations on the domain -- as was done in detail in \cite{Tri08} for Besov and Triebel-Lizorkin spaces. But this might be done later elsewhere.

Following the above briefly sketched path to deal with related sequence spaces, we obtain a number of final or almost final results describing the continuity or compactness of their embeddings. This is not only a necessary preparation for our subsequent study of generalised Morrey smoothness function spaces, but of independent interest. \\

The paper is organised as follows. In Section~\ref{sect-2} we briefly collect general notation and some basic facts for the generalised Morrey smoothness spaces on $\rd$. 
In Section~\ref{sect-seq} we prepare the ground for the intended characterisations of the compactness or continuity of embeddings of the above function spaces on domains. In Theorems~\ref{th-cont} and \ref{th:comp} we investigate the embedding $\na(Q)\hookrightarrow \nb(Q)$, where $Q\subset\rd$ is an arbitrary dyadic cube, sufficiently large to contain $\Omega$, and $\n(Q)$  are the sequence spaces associated to $\MB(\Omega)$. We collect a number of examples and special settings to illustrate the outcome. To deal with the sequence spaces $\bp(Q)$, which denote the appropriate discretisations for the spaces $\Bphi(\Omega)$, we need to define the indices $\sigma$, $\sigma_\infty$ and $\overline{\sigma}$ mentioned above. This can be found in Definition~\ref{indices}, complemented by some short discussion and examples for those numbers. The essential compactness result is presented in Proposition~\ref{Lemma-LS_1+2}.

In Section~\ref{sect-MB} we present our findings concerning the embedding $\MBa(\Omega)\hookrightarrow \MBb(\Omega)$, mainly in Theorem~\ref{th-cont-BM}. Again we exemplify this general result in certain typical situations for further use. This covers, say, embeddings with (usual) Besov spaces $\B(\Omega)$ as target or source spaces, respectively, or even spaces of type $L_r(\Omega)$.

Finally, in the concluding Section~\ref{sect-MB}, we consider the remaining generalised Morrey smoothness spaces $\Bphi(\Omega)$, $\Fphi(\Omega)$ and $\MF(\Omega)$, including thus also the space $\bmo(\Omega)$, and, finally $\M(\Omega)$.\\

At this point we would like to add special thanks, first to our former colleague Helena F. Gon\c{c}alves who was very much involved in the first discussions about the subject of this paper. We owe her a lot of inspiration and very much regret, that for some personal reasons she decided to leave academics some time ago.\\

Moreover, all the three authors are extremely grateful for the long-time mentorship we enjoyed by Professor Hans Triebel. Each one of us benefited from the close scientific contact with him very much, though at different length and extent (he was, for instance, the PhD supervisor of the first two authors). We always  take a lot of inspiration working with him and trying to understand his special scientific insights, see the remarkable interview \cite{CaH-Tri} taken some years ago. In addition, we always enjoy his special sense of humour and witty remarks. We think, that in addition to his unquestionable merits as an outstanding scientist he is also a great academic teacher. And a very kind and modest person, too. It is always our pleasure to meet him and discuss some new ideas with him. Happy birthday, Hans Triebel!

%\open{missing, add special thanks to Helena in the end}

{%\color{verde}
\section{Preliminaries}\label{sect-2}

First we fix some notation. By $\nat$ we denote the \emph{set of natural numbers},
by $\no$ the set $\nat \cup \{0\}$,  and by $\zd$ the \emph{set of all lattice points
in $\rd$ having integer components}. Let $\no^\nd$, where $\nd\in\nat$, be the set of all multi-indices, $\alpha = (\alpha_1, \ldots,\alpha_\nd)$ with $\alpha_j\in\no$ and $|\alpha| := \sum_{j=1}^\nd \alpha_j$. If $x=(x_1,\ldots,x_\nd)\in\rd$ and $\alpha = (\alpha_1, \ldots,\alpha_\nd)\in\no^\nd$, then we put $x^\alpha := x_1^{\alpha_1} \cdots x_\nd^{\alpha_\nd}$. 
For $a\in\rr$, let   $\whole{a}:=\max\{k\in\zz: k\leq a\}$, $\up{a} = \min\{k\in \zz:\; k\ge a \}$, and $a_{+}:=\max(a,0)$. 
Given any $u\in (0,\infty]$, it will be denoted by $u'$ the number, possibly $\infty$, defined by the expression  $\frac{1}{u'}=(1-\frac{1}{u})_+$; in particular, when $1\leq u\leq \infty$, $u'$ is the same as the conjugate exponent defined through $\frac{1}{u}+ \frac{1}{u'}=1$.
All unimportant positive constants will be denoted by $C$, % occasionally with subscripts.
occasionally the same letter $C$ is used to denote different constants  in the same chain of inequalities.
 By the notation $A \ls B$, we mean that there exists a positive constant $c$ such that
 $A \le c \,B$, whereas  the symbol $A \sim B$ stands for $A \ls B \ls A$.
 We denote by
 $|\cdot|$ the Lebesgue measure when applied to measurable subsets of $\rd$.
For each cube $Q\subset \rd $ we denote  its side length by $ \ell(Q)$, and, for $a\in (0,\infty)$, we denote by $aQ$ the cube concentric with $Q$ having the side length $a\ell(Q)$. For $x\in\rd$ and $r \in (0, \infty)$  we denote by $Q(x, r)$ the compact cube centred at $x$ with side length $r$, whose sides are parallel to the axes of coordinates. We write simply $Q(r)=Q(0,r)$ when $x=0$.
By $\mathcal{Q}$ we denote the collection of all \emph{dyadic cubes} in $\rd$, namely, $\mathcal{Q}:= \{Q_{j,k}:= 2^{-j}([0,1)^\nd+k):\ j\in\zz,\ k\in\zd\}$.  For all $Q\in\mathcal{Q}$, let $j_Q:=-\log_2\ell(Q)$, and let $j_Q\vee 0:=\max(j_Q,\,0)$.
Given two (quasi-)Banach spaces $X$ and $Y$, we write $X\hookrightarrow Y$
if $X\subset Y$ and the natural embedding of $X$ into $Y$ is continuous.

\begin{definition} \label{def-gen-Morrey}
	Let $0<p<\infty$ and $\varphi:(0,\infty)\rightarrow [0,\infty)$ be a function which does not satisfy $\varphi\equiv 0$. Then $\M(\rd)$ is the set of all 
	locally $p$-integrable functions $f\in L_p^{\mathrm{loc}}(\rd)$ for which 
	\begin{equation*} %\label{Morrey-norm_0}
		{\|f\mid \M(\rd)\|_\star:=\sup_{x\in\rd, r>0} \varphi(r) %\bigl(\ell(Q(x,r))\bigr)
			\biggl(\frac{1}{|Q(x,r)|}\int_{Q(x,r)} |f(y)|^p \dint y\biggr)^{\frac{1}{p}} \, <\, \infty\, . }
		%\|f\mid \M(\rd)\|:=\sup_{Q\in\mathcal{Q}}\varphi %\bigl(\ell(Q)\bigr)\biggl(\frac{1}{|Q|}\int_Q |f(y)|^p \dint %y\biggr)^{\frac{1}{p}} \, <\, \infty\, .
	\end{equation*}
\end{definition}

\begin{remark} \label{rmk1}
	The above definition goes back to \cite{Nak94}.
	When $\varphi(t)=t^{\frac{\nd}{u}}$ for $t>0$ and $0<p\leq u<\infty$, then $\M(\rd)$ coincides with ${\mathcal M}_{u,p}(\rd)$, which in turn recovers the Lebesgue space $L_p(\rd)$ when $u=p$. 
	%%
	%In the definition of $\|\cdot\mid \M(\rd)\|$ balls or cubes with %sides parallel to the axes of coordinates can be taken. This  change %leads to  equivalent quasi-norms. 
	Note that for $\varphi_0\equiv 1$ (which would correspond to $u=\infty$) we obtain %as in that case
	\begin{equation}\label{M1p}
		{\mathcal M}_{\varphi_0,p}(\rd) = L_\infty(\rd),\quad 0<p<\infty,\quad \varphi_0\equiv 1,
	\end{equation}
	due to Lebesgue's differentiation theorem.  
	
	When $\varphi(t)=t^{-\sigma} \chi_{(0,1)}(t)$ where  $-\frac{\nd}{p}\leq \sigma <0$, then $\M(\rd)$ coincides with the local Morrey spaces $\mathcal{L}^{\sigma}_p(\rd)$ introduced by Triebel in \cite{Tri11}, cf. also \cite[Section~1.3.4]{Tri13}. If $ \sigma=-\frac{\nd}{p}$, then the space is a uniform Lebesgue space $\mathcal{L}_p(\rd)$.
\end{remark}

\bigskip

A natural assumption for $\Mf(\rd)$ is that $\varphi$ belongs to the class $\Gp$, where $\Gp$ is the set of all non-decreasing functions $\varphi:(0,\,\infty)\rightarrow(0,\,\infty)$ such that 
\begin{align}\label{Gp-def}
	t^{-\frac{d}{p}}\varphi(t)\geq s^{-\frac{d}{p}}\varphi(s),
\end{align}
for all $0<t\leq s<\infty$.

\begin{remark}
	A justification for the use of the class $\Gp$ comes from \cite[Lemma~2.2]{NNS16}, see also \cite{Nak00}.  More precisely, for $p$ and $\varphi$ as in Definition~\ref{def-gen-Morrey}, it holds $\mathcal{M}_{\varphi,p}(\rd)\neq \{0\}$ if, and only if, $\displaystyle\sup_{t>0} \varphi(t)  \min (t^{-\frac{\nd}{p}},1) < \infty$. 
	Moreover, if $\displaystyle\sup_{t>0} \varphi(t)  \min (t^{-\frac{d}{p}},1) < \infty$, then there exists $\varphi^*\in\Gp$ such that $\mathcal{M}_{\varphi,p}(\rd)  = {\mathcal M}_{\varphi^*,p}(\rd) $ in the sense of equivalent (quasi-)norms. We refer to \cite{FHS-MS-2} for the proofs. 
	
	One can easily check that $\mathcal{G}_{p_1}\subset \mathcal{G}_{p_2}$ if $0<p_2\le p_1<\infty$.  We refer the reader to \cite[Section~12.1.2]{FHS-MS-2} for more details about the class $\Gp$. 
	
	Note that $\varphi\in\Gp$ enjoys a doubling property, i.e., $\varphi(r)\leq \varphi(2r)\leq 2^\frac{d}{p}\varphi(r)$, $0<r<\infty$, {therefore} we can define an equivalent quasi-norm in $\Mf(\rd)$ by taking the supremum over the collection $\mathcal{Q}$ of all dyadic cubes, namely,
	\begin{align}\label{Mf-norm}
		\|f \mid \Mf(\rd)\| :=\sup_{P\in \mathcal{Q}} \varphi(\ell(P))
		\left(\frac{1}{|P|}\int_{P} |f(y)|^p \dint y \right)^{\frac{1}{p}}.
	\end{align}
	In the sequel, we always assume that $\varphi\in\Gp$ and consider the quasi-norm \eqref{Mf-norm}.   This choice is natural and covers all interesting cases. {Moreover, we shall usually assume in the sequel that $\varphi(1)=1$.}
\end{remark}

Here we illustrate some examples, partly from \cite{FHS-MS-2}.

\begin{examples}\label{ex-phi}
	\begin{enumerate}[\bfseries\upshape  (i)]
				\item Let $u\in\real$, $0<p<\infty$ and let $\varphi(t)=t^u$ for $t>0$. Then $\varphi$ belongs to $\Gp$ if, and only if, $0\leq u\leq \frac{d}{p}$.
		\item Let $0<u,\,v<\infty$. Then
		\begin{equation}\label{ex-u-v}
			\varphi_{u,v}(t)=\begin{cases}
				t^{\frac{\nd}{u}},\quad{\text{if}}\quad t\leq 1\\
				t^{\frac{\nd}{v}},\quad{\text{if}}\quad t> 1
			\end{cases}
		\end{equation}
		belongs to $\Gp$ with $p=\min(u,v)$. 
		In particular, the function $\varphi(t)=t^{\frac{\nd}{u}}$ belongs to $\Gp$ whenever $0<p\leq u<\infty$ by taking $u=v$. Moreover, the function $\varphi(t)=\max(1, t^{\nd/v})$ belongs to $\Gpx{v}$ (corresponding to $u=\infty$), while $\varphi(t)=\min(t^{\nd/u}, 1)$ belongs to $\Gpx{u}$ (corresponding to $v=\infty$). 
		\item Let $0<p<\infty$, $a\leq 0$ and {let $L$  be a sufficiently large constant. Then} $\varphi(t)=t^{\frac{d}{p}}(\log(L+t))^a$ for $t>0$ 
		belongs to $\Gp$.
		\item Let $0<p<\infty$,  {let $u$ be a sufficiently small positive  constant and  }
		$\varphi(t)=t^u(\log(e+t))^{-1}$ for $t>0$. Then $\varphi\notin\Gp$.
		\item Let
		\begin{equation}\label{ex-log}
			\varphi(t)=\begin{cases} \frac{1}{\log 2} \log(1+t), & 0<t<1, \\ t, & t\geq 1.\end{cases}
		\end{equation}
		Then $\varphi \in \mathcal{G}_{\nd}$.
		\item	Let $a \ge e$. Then
	\begin{equation}\label{exp1}
		\varphi(t)= \begin{cases}
			(\ln a)\,  (\ln t^{-1})^{-1} & \text{if} \; 0<t <a^{-1},\\
			1  & \text{if} \; t \ge a^{-1},
		\end{cases}
	\end{equation}
	belongs to  $\mathcal{G}_r$ with $r=d\ln a$. 
	\item Let $0<p \le d$. Then
	\begin{equation}\label{ex-psi}
	\psi(t)=\begin{cases}
	(e\,t)^\frac{d}{p}\ln(t^{-1})&\; \text{if} \quad 0<t<e^{-1},\\
	1& \;\text{if}\quad t\ge e^{-1},
\end{cases}
\end{equation}
belongs to $\mathcal{G}_p$. 
	\end{enumerate}
\end{examples}

Other examples can be found e.g. in \cite[Example~3.15]{Saw18}.
\medskip

Now we introduce the generalised Morrey smoothness spaces.
%\bigskip
Let $\mathcal{S}(\rd)$ be the set of all Schwartz functions on $\rd$, endowed
with the usual topology,
and denote by $\mathcal{S}'(\rd)$ its topological dual, namely,
the space of all bounded linear functionals on $\mathcal{S}(\rd)$
endowed with the { weak$^\ast$ topology}.
For all $f\in \mathcal{S}(\rd)$ or $f\in\mathcal{S}'(\rd)$, we
use $\mathcal{F} f $ to denote its Fourier transform, and $\mathcal{F}^{-1}f$ for its inverse.
Now let us define the generalised Besov-Morrey spaces introduced in \cite{NNS16}.

Let $\eta_0,\eta\in \mathcal{S}(\rd)$ be nonnegative compactly supported functions satisfying
\begin{equation*}
\eta_0(x)>0 \quad \text{if}\quad x \in Q(2),
\end{equation*}
\begin{equation*}
0\notin \supp \eta\quad \text{and} \quad \eta(x)>0 \quad \text{if}\quad x \in Q(2) \setminus Q(1).
\end{equation*}
For $j\in\nat$, let $\eta_j(x):=\eta(2^{-j}x)$,  $x\in\rd$.

\begin{definition} \label{def-spaces} 
Let $0<p<\infty$, $0<q\leq \infty$, $s\in \rr$, and $\varphi\in \Gp$. 
  \begin{enumerate}[\bfseries\upshape  (i)]
 \item The  generalised Besov-Morrey   space
  $\MB(\rd)$ is defined to be the set of all  $f\in\mathcal{S}'(\rd)$ such that
\begin{align*}
\big\|f\mid \MB(\rd)\big\|:=
\bigg(\sum_{j=0}^{\infty}2^{jsq}\big\| \mathcal{F}^{-1}(\eta_j  \mathcal{F} f)\mid
\M(\rd)\big\|^q \bigg)^{1/q} < \infty,
\end{align*}
with the usual modification if $q=\infty$.
\item Assume that there exists $C,\,\varepsilon>0$ such that 
        \begin{align}\label{intc}
            \frac{t^\varepsilon}{\varphi(t)}\leq \frac{C r^\varepsilon}{\varphi(r)}
            \quad\text{holds for}\quad t\geq r,
        \end{align}
        when $0<q<\infty$. 
The generalised Triebel-Lizorkin-Morrey space $\MF(\rd)$ is defined to be the set of all $f\in\mathcal{S}'(\rd)$ such that
         \begin{align*}
            \|f\mid \MF(\rd)\|:=\left\|\left(\sum_{j=0}^\infty 2^{j s q} |\mathcal{F}^{-1}(\eta_j  \mathcal{F} f)(\cdot)|^q\right)^\frac{1}{q}\mid\M(\rd)\right\|<\infty
        \end{align*}
        with the usual modification for $q=\infty$.
    \item The space $\mathcal{A}_{\varphi,p,q}^s(\rd)$ denotes either $\MB(\rd)$ or $\MF(\rd)$, assuming  that $\varphi$ satisfies \eqref{intc} when $q<\infty$ and $\mathcal{A}_{\varphi,p,q}^s(\rd)=\MF(\rd)$.
 \item The generalised Besov-type space ${B}_{p,q}^{s,{\varphi}}(\rd)$ is defined to be the set of all $f\in\mathcal{S}'(\rd)$ such that
 \begin{align*}
            \|f\mid {B}_{p,q}^{s,{\varphi}}(\rd)\|:=\sup_{P\in\mathcal{Q}} {\frac{\varphi(\ell(P))}{|P|^\frac{1}{p}}}
            \left[\sum_{j=j_P\vee 0}^\infty 2^{j s q}\left(\int_P
            | \mathcal{F}^{-1}(\eta_j  \mathcal{F} f)(x)|^p \dint x\right)^\frac{q}{p}\right]^\frac{1}{q}<\infty
\end{align*}
 with the usual modification for $q=\infty$.
\item
  Assume that $\varphi$ satisfies \eqref{intc} when $q<\infty$. 
        The generalised Triebel-Lizorkin-type space ${F}_{p,q}^{s,{\varphi}}(\rd)$ is defined to be the set of all $f\in\mathcal{S}'(\rd)$ such that 
        \begin{align*}
            \|f\mid {F}_{p,q}^{s,{\varphi}}(\rd)\|:=\sup_{P\in\mathcal{Q}} {\frac{\varphi(\ell(P))}{|P|^\frac{1}{p}}}
            \left[\int_P\left(\sum_{j=j_P\vee 0}^\infty 2^{j s q}
            |\mathcal{F}^{-1}(\eta_j  \mathcal{F} f)(x)|^q\right)^\frac{p}{q}\dint x\right]^\frac{1}{p}<\infty
        \end{align*}
        with the usual modification for $q=\infty$.
        \item The space ${A}_{p,q}^{s,\varphi}(\rd)$ denotes either $\Bphi(\rd)$ or $\Fphi(\rd)${, assuming} that $\varphi$ satisfies \eqref{intc} when $q<\infty$ and $\Aphi(\rd)=\Fphi(\rd)$.
        \end{enumerate}
\end{definition}
%%%%%%%%%%%%%%%%%%%%%%%%%%%%%%%%%%%

\begin{remark}\label{rem-coinc}
 The generalised Besov-Morrey  spaces and the generalised Triebel-Lizorkin-Morrey spaces have been introduced in \cite{NNS16} {while} the generalised Besov-type spaces  and the generalised Triebel-Lizorkin-type spaces have been introduced in \cite{HL23}.
  When $\varphi(t)=t^{\frac{\nd}{p}}$ for $t>0$, then %both spaces coincide with 
  the  classical Besov spaces $B^s_{p,q}(\rd)$ and the classical Triebel-Lizorkin spaces $F^s_{p,q}(\rd)$, 
for any $0<p<\infty$, $0<q\leq \infty$, and $s\in \rr$, {are recovered from both scales,} since 
$$
B^s_{p,q}(\rd)={\mathcal N}^s_{{\varphi},p,q}(\rd)= {B}_{p,q}^{s,{\varphi}}(\rd) \quad \text{and} \quad  F^s_{p,q}(\rd)={\mathcal E}^s_{{\varphi},p,q}(\rd)= {F}_{p,q}^{s,{\varphi}}(\rd). 
$$
{See, for example, \cite{T83,T92,T06,T20} for more details about these spaces.}

 When $\varphi(t)=t^{\frac{\nd}{u}}$ for $t>0$ and $0<p\leq u<\infty$, then 
$$
\MB(\rd)={\mathcal N}^s_{u,p,q}(\rd)  { \quad \text{and} \quad  \MF(\rd)={\mathcal E}^s_{u,p,q}(\rd)  }
$$
are, respectively, the usual Besov-Morrey spaces and the usual Triebel-Lizorkin-Morrey spaces which are studied in \cite{YSY10} or in the survey papers by Sickel \cite{s011,s011a}.

% When $\varphi(t)=t^\frac{d}{p}$ for $t>0$, then $B_{p,q}^{s,\varphi}(\rd)$ and $F_{p,q}^{s,\varphi}(\rd)$ coincide with the classical Besov space $B_{p,q}^s(\rd)$ and {the} Triebel-Lizorkin space $F_{p,q}^s(\rd)$, respectively;  see, for example, \cite{T83,T92,T06,T20} for more details about these spaces.

When  $\varphi(t)=\min(t^\frac{\nd}{u},1)$, then we recover  the local Besov-Morrey spaces introduced by Triebel, 
\[ \MB(\rd) = B^s_q(\mathcal{L}^\sigma_p, \rd), \quad\sigma= -\frac{\nd}{u}, \quad p\le u, \] 
cf. \cite[Section~1.3.4]{Tri13}.

The spaces ${B}_{p,q}^{s,{\varphi}}(\rd)$ and ${F}_{p,q}^{s,{\varphi}}(\rd)$ are equivalent, respectively, to the spaces ${B}_{p,q}^{s,{\tilde{\varphi}}}(\rd)$ and ${F}_{p,q}^{s,{\tilde{\varphi}}}(\rd)$ introduced in \cite{HL23}, if one replaces $\varphi(\ell(P))^{-1}|P|^\frac{1}{p}$ by a function $\tilde{\varphi}(\ell(P))$ in the above definition, where $\tilde{\varphi}\in\Gp$ and satisfies $t^{\varepsilon-\frac{d}{p}}\tilde{\varphi}(t)\leq C r^{\varepsilon-\frac{d}{p}}\tilde{\varphi}(r)$ for $t\geq r$ when $q<\infty$ in case of the $F$-spaces. 

If {$\varphi(t)=t^{d(\frac{1}{p}-\tau)}$ for $t>0$ and $\tau\in[ 0,\frac{1}{p})$, then $B_{p,q}^{s,\varphi}(\rd)$ and $F_{p,q}^{s,\varphi}(\rd)$} coincide with, respectively,  the Besov-type space $B_{p,q}^{s,\tau}(\rd)$ and the Triebel-Lizorkin-type space $F_{p,q}^{s,\tau}(\rd)$ introduced in \cite{YSY10}. 

%Of course, in both scales we can recover the  classical Besov spaces $B^s_{p,q}(\rd)$ and Triebel-Lizorkin spaces $F^s_{p,q}(\rd)$ for any $0<p<\infty$, $0<q\leq \infty$, and $s\in \rr$, since 
%$$
%A^s_{p,q}(\rd)={\mathcal A}^s_{p,p,q}(\rd) = A_{p,q}^{s,0}(\rd) .
%$$
\end{remark}

Let us briefly recall some embeddings that we shall use below. We always assume that $s\in\rr$, $0<p<\infty$, $0<q\leq\infty$, $\varphi\in\Gp$, and  $\varphi$ satisfies \eqref{intc} when $q<\infty$ and $\Aphi(\rd)=\Fphi(\rd)$ or $\MA(\rd)=\MF(\rd)$, respectively. Then,
    \begin{align*}
        \mathcal{A}_{\varphi,p,q_1}^s(\rd)\hookrightarrow\mathcal{A}_{\varphi,p,q_2}^s(\rd),\qquad q_1\leq q_2,
    \end{align*}
and 
    \begin{align*}%\label{asea}
        \mathcal{A}_{\varphi,p,q_1}^{s+\varepsilon}(\rd)\hookrightarrow\mathcal{A}_{\varphi,p,q_2}^s(\rd),\qquad \varepsilon>0,
    \end{align*}
    see also \cite[Proposition~3.3]{NNS16} for more details. In both cases we have obvious counterparts within the scale $\Aphi(\rd)$.

As stated in  \cite[Theorem~4.1]{HLMS24}, {one has}
%Let $s\in\rr$, $0<p<\infty$, $0<q\leq\infty$, $\varphi\in\Gp$. If $\varphi$ satisfies
\eqref{intc} when $q<\infty$, then 
\begin{equation} \label{E=F}
\Fphi(\rd)=\MF(\rd)
\end{equation}
with equivalent norms. Moreover, 
\begin{equation} \label{emb-B-N}
 \MfB(\rd) \hookrightarrow B_{p,q}^{s, \varphi}(\rd) \quad \text{with} \quad \mathcal{N}^s_{\varphi, p, \infty}(\rd) = B_{p,\infty}^{s, \varphi}(\rd).
\end{equation}
Note that, when $q<\infty$ and 
        \begin{equation}\label{ls1}
        \lim_{t\rightarrow 0} \varphi(t)t^{-\frac{d}{p}}=\infty \quad \text{or} \quad  \lim_{t\rightarrow +\infty} \varphi(t)t^{-\frac{d}{p}}=0,
        \end{equation}
        then $\MB(\rd)$ is a proper subspace of $B_{p,q}^{s, \varphi}(\rd)$.
        
Similar to the classical case one has
        \begin{equation}\label{elem-gen}
B^{s,\varphi}_{p,\min(p,q)}(\rd)\, \hookrightarrow \, \Fphi(\rd)\, \hookrightarrow \, B^{s,\varphi}_{p,\max(p,q)}(\rd),
\end{equation}
while for the spaces $\MA(\rd)$ we could prove in \cite[Theorem~4.3]{HLMS24} that
    \begin{align}\label{Nphi-Ephi'}
      \mathcal{N}_{\varphi,p,\min(p,q)}^s(\rd)\hookrightarrow\MF(\rd).
    \end{align}
Furthermore, if
    \begin{align}\label{een-1'}
      \lim_{t\rightarrow 0} t^{-\frac{d}{p}}\varphi(t)<\infty\quad\text{and}\quad
      \lim_{t\rightarrow +\infty}t^{-\frac{d}{p}}\varphi(t)>0,
      \end{align}
then
      \begin{align}\label{een''}
    \MfF(\rd)\hookrightarrow\mathcal{N}_{\varphi,p,q_0}^s(\rd)
\end{align}
if, and only if, $q_0\geq  \max(p,q)$. 
Otherwise, if \eqref{een-1'} is not satisfied, then \eqref{een''} holds if, and only if, $q_0=\infty$. We refer to \cite{saw08} for preceding results.\\

\ignore{Moreover, let $0<p<u<\infty$, then according to \cite[Proposition~1.3]{saw08}, we have a sharp embedding
    \begin{align}\label{se-au}
        \mathcal{N}_{u,p,\min(p,q)}^s(\rd)\hookrightarrow\MF(\rd)\hookrightarrow\mathcal{N}_{u,p,\infty}^s(\rd).
    \end{align}
    The index $\infty$ on the right-hand side of \eqref{se-au} cannot be replaced by any finite number; 
    see \cite[Proposition~1.6]{saw08}.}

%\open{introduce $\MF$ and $\Fphi$ and their relation to each other, basic embedding results from $\rd$ (with references), notation $\MA$ and $\Aphi$}

Let us mention that
  \begin{equation}\label{E=Mf}
    \mathcal{E}^{0}_{\varphi,p,2} (\rd)=\Mf(\rd), \quad \varphi\in\Gp, \quad 1<p<\infty,
  \end{equation}
  {if} \eqref{intc} holds, cf. \cite{NNS16} with reference to \cite[Proposition~5.1]{SHG-15}, see also \cite[Proposition~3.18]{YZY-BJMA15} in a more general context.  
    This generalises the Mazzucato result \cite[Proposition~4.1]{Maz03} related to the case $\varphi(t)=t^\frac{d}{u}$ for $t>0$, which reads as $\mathcal{E}^0_{u,p,2}(\rd)=\Mu(\rd)$, $1<p\leq u<\infty$.
%In particular, this covers the well-known classical outcome that 
%\begin{equation}\label{E-Lp}
%\mathcal{E}^0_{p,p,2}(\rd)=L_p(\rd)=F^0_{p,2}(\rd),\quad 1<p<\infty,
%\end{equation}
%always understood in the sense of equivalent norms.
\bigskip~

Finally, let us introduce spaces on domains. By  $\Omega$ we will always denote  a bounded $C^{\infty}$ domain in $\rd$. All spaces $\MA(\Omega)$ and $\Aphi(\Omega)$ are defined by restriction. In special cases, like $\varphi(t)=t^{d/u}$, $0<p\leq u<\infty$, or $\varphi(t)=t^{d(\frac1p-\tau)}$, $0<p<\infty$, $0\leq \tau<\frac1p$, we have dealt with continuous and compact (even nuclear) embeddings of spaces of the above type in our recent papers \cite{GHS-21,GHS-23,HSS-morrey,hs13,hs14,hs20,HMSS-morrey,hs24}. When appropriate we shall refer to those results in detail below. As an example let us recall some  compactness result for spaces of type $\MAu(\Omega)$, see also \cite[Theorem~3.1]{GHS-21}, and for spaces $\At(\Omega)$ as obtained in \cite[Theorem~3.2]{GHS-21}. In view of the comparison with the spaces $\MA$ studied here, recall Remark~\ref{rem-coinc}, we present our earlier results for $u_i<\infty$, $p_i<\infty$, $0\leq \tau_i< \frac{1}{p_i}$, $i=1,2$, only.

\begin{corollary}[\cite{GHS-21}]  \label{comp-class}
  Let  $s_i\in \real$, $0<q_i\leq\infty$, $i=1,2$. 
  \begin{enumerate}[{\bfseries\upshape(i)}]
    \item
Assume  $0<p_i\leq u_i<\infty$, $i=1,2$. Then  the embedding
\begin{equation} \label{bd1comp}
	 \id_{\mathcal{A}}: \MAa(\Omega )\hookrightarrow \MAb(\Omega )
\end{equation}
is compact if, and only if,
the following condition holds:
\begin{equation}\label{bd3acomp}
\frac{s_1-s_2}{d} > \max\bigg\{0,\frac{1}{u_1} - \frac{1}{u_2},
\frac{1}{u_1} \Big(1- \frac{p_1}{p_2}\Big) \bigg\}.
\end{equation}
\item
Let $0\leq \tau_i< \frac{1}{p_i}$, $i=1,2$.    The embedding
\begin{equation} \label{tau-comp-u1}
	 \id_{\tau}: \Ata(\Omega )\hookrightarrow \Atb(\Omega )
\end{equation}
is compact if, and only if,
the following condition holds:
\begin{equation}\label{tau-comp-u2}
  \frac{s_1-s_2}{d} > \left(\frac{1}{p_1}-\tau_1 -\frac{1}{p_2}+\max\{\tau_2, \frac{p_1}{p_2}\tau_1\}\right)_+.
\end{equation}
\end{enumerate}
\end{corollary}

\remark{As already mentioned, this represents only part of our compactness results for spaces of type ${\At}(\Omega)$ and $\MAu(\Omega)$, just for comparison with the new ones obtained in this paper. Note that we also have continuity results for those `classical' Morrey smoothness spaces which can be found in \cite{hs13,hs14,GHS-23}. Moreover, in \cite{hs20,GHS-21} we could characterise the compactness of such function spaces further in terms of entropy (and approximation) numbers, a topic also studied by Hans Triebel before, cf. the monographs \cite{ET96,T97}. %\red{should we mention here some of the standard literature, including {Edmunds/Triebel \dots}?}
    %\magenta{\cite{ET96}, \cite{T93,T97} \cite{ht1994a,ht1994b}? others?}
    In \cite{hs24} we extended these investigations to questions of nuclearity of the corresponding embeddings. Finally, we also refer to   \cite{ht2021} where the clan concept, introduced by Hans Triebel (for the Morrey smoothness spaces, including $\MAu$ and $\At$), was discussed in detail.

\remark{\label{cond-varrho1}
  Note that in case of $p_1\geq p_2$, condition \eqref{tau-comp-u2} can be rewritten as
\begin{equation}\label{tau-comp-u2'}
  \frac{s_1-s_2}{d} > \left(\frac{1}{p_1}-\tau_1 -\frac{1}{p_2}+\tau_2\right)_+ \quad\text{if}\quad p_1\geq p_2,
\end{equation}
since $\frac{1}{p_1}-\tau_1 -\frac{1}{p_2}+ \frac{p_1}{p_2}\tau_1 = (\tau_1-\frac{1}{p_1})(\frac{p_1}{p_2}-1)\leq 0$.
}
}    
    
\section{Embeddings of sequence spaces}\label{sect-seq}
{Our aim is to study compact embeddings in function spaces of type $\MA(\Omega)$ and $\Aphi(\Omega)$. We follow the nowadays quite standard method to apply the wavelet transformation and thus transfer the problem to sequence spaces. So we deal with the appropriate sequence spaces in this section.}

Both, the  generalised Besov-Morrey spaces $\MB(\rd)$ and  the generalised  Besov-type spaces ${B}_{p,q}^{s,{\varphi}}(\rd)$
 admit  wavelet characterisations, cf.  \cite{hms22} and \cite{HLMS24}, respectively.  The underlying  sequence spaces  are defined below.
}

\begin{definition}\label{def-seq-spaces} 
Let $0<p<\infty$, $0<q\leq \infty$, $s\in \rr$, and $\varphi\in \Gp$. 
\begin{enumerate}[\bfseries\upshape  (i)]
\item	The  generalised Besov-Morrey sequence space $\n(\rd)$ is the set of all double-indexed sequences 
$\lambda:=\{\lambda_{j,m}\}_{j\in\no,m\in\zd}\subset \cc$ for which the quasi-norm 
	\begin{equation} \label{seq-besov}
		\| \lambda\mid \n(\rd)\|:= \biggl(\sum_{j=0}^{\infty}2^{jsq}   \Big\| \sum_{m\in\zd} \lambda_{j,m} \chi_{Q_{j,m}} \,\Big|\, \M(\rd) \Big\|^q \biggr)^{1/q} 
	\end{equation}
	is finite (with the usual modification if $q=\infty$).
        \item The generalised Besov-type sequence space ${b}_{p,q}^{s,{\varphi}}(\rd)$ is defined to be the set of all double-indexed sequences 
$\lambda:=\{\lambda_{j,m}\}_{j\in\no,m\in\zd}\subset \cc$ for which the quasi-norm 
        \begin{align*}
            \| \lambda \mid {b}_{p,q}^{s,{\varphi}}(\rd)\|:=
            \sup_{P\in\mathcal{Q}} {\frac{\varphi(\ell(P))}{|P|^\frac{1}{p}}} 
            \Bigg\{\sum_{j=j_P\vee 0}^\infty   \Bigg(\sum_{\substack{m\in\zd:\\ Q_{j,m}\subset P}} (2^{j (s-\frac{d}{p})}
            |\lambda_{jm}|)^p  \Bigg)^\frac{q}{p}\Bigg\}^\frac{1}{q}
        \end{align*}
        is finite (with the usual modification for $q=\infty$).
\end{enumerate}
\end{definition}
\smallskip

%\noindent{\em Convention.}  We adopt the same custom to write   $\aseq$ instead of $\n$ or $\e$, for convenience, when both %scales are meant simultaneously, 
%assuming always that there exist $C,\varepsilon>0$ such that  \eqref{cond-E} holds, when $q<\infty$ and $\aseq=\e$.

\begin{remark} \label{seq-besov-rem} When $\varphi(t)=t^{\frac{\nd}{u}}$ for $t>0$ and $0<p\leq u<\infty$, then 
	$$ 
	\n(\rd)={n}^s_{u,p,q}(\rd) % \quad \text{and} \quad \e(\zd)={e}^s_{u,p,q}(\rd)
	$$
	are the usual Besov-Morrey sequence spaces. % and Triebel-Lizorkin-Morrey sequence spaces.
	Moreover if $u=p$, then the space $\n(\rd)$ coincides with a classical Besov sequence space  $b^s_{p,q}(\rd)$ since $\M(\rd)=L_p(\rd)$ in that case. 
\end{remark}

Let $\Omega$ be a bounded domain in $\rd$. For any $c>0$ we put
\[\Omega_c=\{x\in \rd:\, \mathrm{dist}(x,\Omega)<c\}.  \]
We choose a dyadic cube $Q\subset \rd$ such that $\Omega_{2c}\subset Q$. We may assume that $Q=[0,2^{\nu_0}]^{d}$ for some $\nu_0\le 0$. 
We define the sequence space $\n(Q)$ as the set of all double-indexed sequences 
$\lambda:=\{\lambda_{j,m}\} \subset \cc$, with $j\in\no, m \in\zd$ such that $Q_{j,m} \subset Q$, for which the quasi-norm  
\begin{equation} \label{seq-besov-Om}
	\| \lambda\mid \n(Q)\|:= \Biggl(\sum_{j=0}^{\infty}2^{jsq}   \Big\| \sum_{m \in \zd: Q_{j,m}\subset Q} \lambda_{j,m} \chi_{Q_{j,m}} \,\Big|\, \M(\rd) \Big\|^q \Biggr)^{1/q} 
\end{equation}
is finite. Obviously, $\n(Q)$ can be interpreted as a subspace of $\n(\rd)$. \\

The proof of the next lemma is similar to the proof of Lemma~2.12 in \cite{hms22} {and thus omitted here}.

\begin{lemma} \label{equiv-norm}
	Let $0<p<\infty$, $0<q\leq \infty$, $s\in \rr$, and $\varphi\in \Gp$. Then 
	$$
	\n(Q)= \left\{\lambda=\{\lambda_{j,m}\}_{j\in \no, m\in \zd}:  \| \lambda\mid \n(Q)\|^*<\infty \right\},
	$$
	where
	$$
	\| \lambda\mid \n(Q)\|^*:=\Bigg(\sum_{j=0}^\infty 2^{jsq} 
	\mathop{\sup_{\nu: \nu \leq j}}_{k\in \zd}\!  \varphi(2^{-\nu})^q \,2^{(\nu-j)\frac{\nd}{p} q}\Bigg(\!\! \mathop{\sum_{m\in\zd:}}_{Q_{j,m}\subset Q_{\nu,k}\subset Q}\!\!|\lambda_{j,m}|^p\Bigg)^{\frac q p}\Bigg)^{1/q}
	$$
	with the usual modification if $q=\infty$. Furthermore, $ \| \cdot \mid \n(Q)\|^*$ is an equivalent quasi-norm in $ \n(Q)$.
\end{lemma}
\begin{remark}\label{rem-specialcase}
If  $\lim_{t\rightarrow  0}\varphi(t)=c>0$, then 
	\[ \mathop{\sup_{\nu: \nu \leq j}}_{k\in \zd}\!  \varphi(2^{-\nu}) \,2^{(\nu-j)\frac{\nd}{p}}\Bigg(\!\! \mathop{\sum_{m\in\zd:}}_{Q_{j,m}\subset Q_{\nu,k}\subset Q}\!\!|\lambda_{j,m}|^p\Bigg)^{\frac 1 p} \sim \mathop{\sup_{m\in\zd:}}_{Q_{j,m}%\subset Q_{\nu,k}
		\subset Q}\!\! |\lambda_{j,m}| \]	
uniformly in $j$. In consequence in that case we have the following  norm equivalence 
\begin{equation}
	\| \lambda\mid \n(Q)\|^* \sim  \Bigg(\sum_{j=0}^\infty 2^{jsq} 
\mathop{\sup_{m\in\zd:}}_{Q_{j,m}%\subset Q_{\nu,k}
	\subset Q}\!\! |\lambda_{j,m}|^q\Bigg)^{1/q}
\end{equation}
and for any $p$ the space $\n(Q)$ coincides with  $\ell_q(2^{js}\ell_\infty^{M_j})$,  where  $M_j=2^{jd}|Q|$.%$j(Q)= d^{-1} \log_2 |Q|$. 
\end{remark}

\begin{theorem}\label{th-cont}
Let $s_i\in\rr$, $0<p_i<\infty$, $0<q_i\leq \infty$, and $\varphi_i\in {\mathcal G}_{p_i}$, for $i=1,2$. 
We assume without loss of generality that $\varphi_1(1)=\varphi_2(1)=1$. 
Let $ \varrho=\min(1,\frac{p_1}{p_2})$ and $\displaystyle\alpha_j= \sup_{0\le\nu\le j}\frac{\varphi_2(2^{-\nu})}{\varphi_1(2^{-\nu})^\varrho}$, $j\in \mathbb{N}_0$. 

There is a continuous embedding 
\begin{equation} \label{embed1}
	\na(Q) \hookrightarrow \nb(Q) 
\end{equation}
if, and only if, 

\begin{align}
	%\label{cond0}
	%\sup_{\nu\le 0}\frac{\varphi_2(2^{-\nu})}{\varphi_1(2^{-\nu})^\varrho} & < %\infty , 
	%\intertext{and} 
	\label{cond2}
	\left\{   2^{j(s_2-s_1)} \alpha_j \frac{\varphi_1(2^{-j})^\varrho}{\varphi_1(2^{-j})}\right\}_{j\in \no} & \in \ell_{q^*}  \qquad \text{where}  \quad \frac{1}{q^*}=\left(\frac{1}{q_2}-\frac{1}{q_1}\right)_+  .
\end{align}

%%%%%%%%%%%%%%%%%%%%%%%%%%%%%%%%%%%%%%%%%%%%%%%%%%%%Let 
%$\nu\in\zz$ with  $j\geq \max(0,\nu)$. Suppose the following conditions hold:
%1) Let the sequence $\alpha_j$  be almost strictly increasing or $\alpha_j\rightarrow \alpha<\infty$ and there exist $c>0$ such that 
%\begin{equation}\label{thm1-c}
% c\le \inf_{j\in \mathbb{N}} \frac{\varphi_2(2^{-j})}{\varphi_1(2^{-j})^\varrho}
%\exists  C>0 \quad \forall j\in \mathbb{N}_0 \quad \varphi_1(2^{-j})^\varrho\le C \varphi_2(2^{-j}).
%\end{equation} 
% Then the embedding 
%\begin{align} %\label{cond2a}
% &\sup_{\nu\le j}\frac{\varphi_2(2^{-\nu})}{\varphi_1(2^{-\nu})^\varrho} \le C  \frac{\varphi_2(2^{-j})}{\varphi_1(2^{-j})^
	%\varrho},
% \intertext{and}
%\label{cond2}
%\left\{   2^{j(s_2-s_1)} \frac{\varphi_2(2^{-j})}{\varphi_1(2^{-j})}\right\}_j \in \ell_r,  \quad \text{for}  \quad 0< r\leq%\infty \quad \text{such that}\quad\frac{1}{r}=\big(\frac{1}{q_2}-\frac{1}{q_1}\big)_+  .
%\end{align} 
%2) Let  $\lim_{j\rightarrow \infty} \alpha_j= \alpha<\infty$ and let the condition \eqref{thm1-c} be not satisfied. 
%   If 
% $s_1>s_2$ or $s_1=s_2$  and  $q_1\le q_2$ then the embedding \eqref{embed1} holds. 
%In both cases  the embeddings  \eqref{embed1}  implies the conditions \eqref{cond0}. 
%%%%%%%%%%%%%%%%%%%%%%%%%%%%%%%%%%%%%%%%%%%%%%%%%%%%%%%%%%%%%%%%%%%%%%%%%%%%%%%%%%%%%%%%%%%%%%%%%%%%%
\end{theorem}

The proof of the above theorem is similar to the proof of Theorem~4.1 in \cite{hms22}. That theorem gives the sufficient and necessary conditions for continuous embeddings of the spaces $\n(\rr^d)$. The assumptions of the theorem consist of two parts: the local conditions -- involving the values of $\varphi_i$ near zero -- and the global one, depending on the behaviour of $\varphi_i$ at infinity. Now we do not need the global condition and the local one coincides with \eqref{cond2}.

\begin{corollary}\label{cor1} Under the same assumptions as in Theorem~\ref{th-cont}, if the embedding \eqref{embed1} holds, then $s_1\ge s_2$. 
\end{corollary}
\begin{proof}
	%Assume $s_1<s_2$. 
	The sequence $\{\alpha_j\}_{j\in \no}$ is a positive  increasing sequence, so  if $\{ 2^{j(s_2-s_1)} \alpha_j\}_{j\in \no}\in \ell_{\infty}$, then $s_1\ge s_2$. This proves the statement for $\varrho=1$. Similarly, if  $\varrho<1$, then    
	\begin{equation}\label{eq1}
	\alpha_j \frac{\varphi_1(2^{-j})^\varrho}{\varphi_1(2^{-j})}\ge \alpha_j
	\end{equation}  
	for any $j\in \no$. So we have $s_1\ge s_2$ also for the second case.  
\end{proof}
\begin{corollary}\label{cor2}
	In addition to the assumptions of Theorem~\ref{th-cont}, assume $s_1=s_2$ and  $\lim_{t\rightarrow  0}\varphi_1(t)=0$. The embedding \eqref{embed1} holds if, and only if, $p_1\ge p_2$,  $q_1\le q_2$ and there is a constant $C>0$ such that $\varphi_2(2^{-j})\le C \varphi_1(2^{-j})$, for any $j\in \no$.  
\end{corollary}
\begin{proof} 
If \eqref{embed1} holds, then $\big\{\alpha_j\frac{\varphi_1(2^{-j})^\varrho}{\varphi_1(2^{-j})}\big\}_{j\in \no}\in \ell_{q^*}$. If  $\lim_{t\rightarrow  0}\varphi_1(t)=0$ and $\varrho<1$, then the above sequence is not bounded since $\alpha_j\ge 1$. This implies  $\varrho=1$, i.e., $p_1\ge p_2$, and $\big\{\alpha_j\big\}_{j\in \no}\in \ell_{q^*}$. But $\{\alpha_j\}_{j\in \no}$ is not a null sequence so $q^*=\infty$. At the end 
		%Consequently,  
		\[\frac{\varphi_2(2^{-j})}{\varphi_1(2^{-j})} \le \alpha_j\le C. \]
	%Because of \eqref{eq1}, we have always $\{\alpha_j\}_{j\in \no}\in \ell_{q^*}$, that can hold only if $q^*=\infty$ since .  	
	If $q^*=\infty$, $\varrho=1$ and $\varphi_2(2^{-j})\le C \varphi_1(2^{-j})$, then we have immediately  
	%_1(2^{-j})^\varrho}\le C \frac{\varphi_1(2^{-j})}{\varphi_1(2^{-j})^\varrho}. \]
%In consequence,
$\big\{\alpha_j\frac{\varphi_1(2^{-j})^\varrho}{\varphi_1(2^{-j})}\big\}_{j\in \no}\in \ell_\infty$. 	
\end{proof}

 \ignore{The similar statement in the case $\lim_{t\rightarrow  0}\varphi_1(t)=c>0$ is covered by the next two lemmas.}
 In next two corollaries we consider the cases when at least one of the functions $\varphi_i$, $i=1,2$, satisfies the identity $\lim_{t\rightarrow  0}\varphi_i(t)=c>0$. 
\begin{corollary}\label{cor2a} 
	Let $0<p<\infty$, $\varphi\in {\mathcal G}_{p}$, and $s_i\in\rr$,  $0<q_i\leq \infty$,  for $i=1,2$. 
	We assume without loss of generality that $\varphi(1)=1$. 
	There is a continuous embedding 
	\begin{align} \label{embednb}
		n^{s_1}_{\varphi,p,q_1}(Q) \hookrightarrow \ell_{q_2}\left(2^{js_2}\ell_\infty^{{2^{jd}|Q|}}\right)
	\intertext{if, and only if,} \label{embednba} 
	\left\{2^{j(s_2-s_1)}\varphi( 2^{-j})^{-1}\right\}_{j\in \no}  \in \ell_{q^*}.
\end{align}
\end{corollary}
\begin{proof}
The corollary follows from Theorem~\ref{th-cont} and Remark~\ref{rem-specialcase}. If  $\varphi_2\in {\mathcal G}_{p_2}$ is such that 
	 $\ds\lim_{t\rightarrow  0}\varphi_2(t)=c>0$, then 
	 \[ \alpha_j\sim \varphi(2^{-j})^{-\varrho}\qquad \text{and}\qquad  n^{s_2}_{\varphi_2,p_2,q_2}(Q) =  \ell_{q_2}\left(2^{js_2}\ell_\infty^{{2^{jd}|Q|}}\right).   \]
\end{proof}
\begin{corollary}\label{cor2b} 
	Let $0<p<\infty$, $\varphi\in {\mathcal G}_{p}$, and $s_i\in\rr$,  $0<q_i\leq \infty$,  for $i=1,2$. 
	We assume without loss of generality that $\varphi(1)=1$. 
	There is a continuous embedding 
	\begin{align} \label{embedbn}
	\ell_{q_1}\left(2^{js_1}\ell_\infty^{{2^{jd}|Q|}}\right)	\hookrightarrow n^{s_2}_{\varphi,p,q_2}(Q) 
		%\intertext{if, and only if,} 
		%\left\{2^{j(s_2-s_1)}\varphi 2^{-j})^{-1}\right\}_{j\in \no}  \in %\ell_{q^*}
	\end{align}
if, and only if, $s_1>s_2$ or $s_1=s_2$ and $q_1\leq q_2$.% ($q^*=\infty$).
\end{corollary}
\begin{proof}
 The corollary can be proved in the similar way as the above one. Now  $\lim_{t\rightarrow  0}\varphi_1(t)=c>0$ and in consequence $	\alpha_j \frac{\varphi_1(2^{-j})^\varrho}{\varphi_1(2^{-j})}\sim c>0$. 
\end{proof}

We now deal with the compactness of the embedding \eqref{embed1}.

\begin{theorem} \label{th:comp}
	Let $s_i\in\rr$, $0<p_i<\infty$, $0<q_i\leq \infty$, and $\varphi_i\in {\mathcal G}_{p_i}$, for $i=1,2$. 
	We assume without loss of generality that $\varphi_1(1)=\varphi_2(1)=1$. 
	Let $ \varrho=\min(1,\frac{p_1}{p_2})$ and $\displaystyle\alpha_j= \sup_{0\le\nu\le j}\frac{\varphi_2(2^{-\nu})}{\varphi_1(2^{-\nu})^\varrho}$, $j\in \mathbb{N}_0$. 
	
	There is a compact embedding 
	\begin{equation} \label{embed1a}
		\na(Q) \hookrightarrow \nb(Q) 
	\end{equation}
	if, and only if, 
	
	\begin{align}
		\label{cond2a}
	   \left\{  2^{j(s_2-s_1)} \alpha_j \frac{\varphi_1(2^{-j})^\varrho}{\varphi_1(2^{-j})}\right\}_{j\in \no} & \in \ell_{q^*}  \qquad \text{where}  \quad \frac{1}{q^*}=\left(\frac{1}{q_2}-\frac{1}{q_1}\right)_+ 
        \end{align}
	and
        \begin{align}
	  2^{j(s_2-s_1)} \alpha_j \frac{\varphi_1(2^{-j})^\varrho}{\varphi_1(2^{-j})}\rightarrow 0\qquad \text{if}\qquad q_1\leq q_2. %q^*=\infty .
          \label{cond2b}
		%\qquad \text{where}  \quad %\frac{1}{q^*}=\left(\frac{1}{q_2}-\frac{1}{q_1}\right)_+ 
	\end{align}
\end{theorem} 

\begin{proof} 
	\emph{Step 1 (sufficiency)}.
	If $p_2\le p_1$, i.e., $\varrho=1$,  then  we have the following inequality 
	\begin{align}\label{s1}
		& \mathop{\sup_{\nu: \nu \leq j}}_{k\in \zd} \varphi_2(2^{-\nu}) 2^{(\nu-j)\frac{\nd}{p_2}}\biggl(\!\!\mathop{\sum_{m\in\zd:}}_{Q_{j,m}\subset Q_{\nu, k}}  
		\!\! | \lambda_{j,m}|^{p_2} \biggr)^{\frac{1}{p_2}}  \\
		& \qquad  \qquad \qquad  \qquad \leq  \alpha_j\, 
		\mathop{\sup_{\nu: \nu \leq j}}_{k\in \zd} \varphi_1(2^{-\nu})  2^{(\nu-j)\frac{\nd}{p_1}}\biggl(\!\! \mathop{\sum_{m\in\zd:}}_{Q_{j,m}\subset Q_{\nu, k}}  \!\! 
		| \lambda_{j,m}|^{p_1} \biggr)^{\frac{1}{p_1}}, \nonumber
	\end{align}
	cf. (4.4) in \cite{hms22}. On the other hand, if $p_2> p_1$, i.e., $\varrho<1$,  then 
	\begin{align}\label{s2}
		& \mathop{\sup_{\nu: \nu \leq j}}_{k\in \zd} \varphi_2(2^{-\nu}) 2^{(\nu-j)\frac{\nd}{p_2}}\biggl(\!\! \mathop{\sum_{m\in\zd:}}_{Q_{j,m}\subset Q_{\nu, k}}  \!\!
		| \lambda_{j,m}|^{p_2} \biggr)^{\frac{1}{p_2}}  \\
		& \qquad  \qquad \qquad  \qquad \leq  \alpha_j\, \frac{\varphi_1(2^{-j})^\varrho}{\varphi_1(2^{-j})}
		\mathop{\sup_{\nu: \nu \leq j}}_{k\in \zd} \varphi_1(2^{-\nu})  2^{(\nu-j)\frac{\nd}{p_1}}\biggl(\!\!\mathop{\sum_{m\in\zd:}}_{Q_{j,m}\subset Q_{\nu, k}}  \!\!
		| \lambda_{j,m}|^{p_1} \biggr)^{\frac{1}{p_1}},  \nonumber
	\end{align}
    cf. (4.5) in \cite{hms22}. In both cases we get, for a $j_0\in\no$, 
    \begin{align}
    	\Bigg(\sum_{j=j_0}^\infty 2^{j {s_2}q_2} 
    	\mathop{\sup_{\nu: \nu \leq j}}_{k\in \zd}\!  \varphi_2(2^{-\nu})^{q_2} \,2^{(\nu-j)\frac{\nd}{p_2} q_2}\Big(\!\! \mathop{\sum_{m\in\zd:}}_{Q_{j,m}\subset Q_{\nu,k}\subset Q}\!\!|\lambda_{j,m}|^p\Big)^{\frac{ q_2}{p_2}}\Bigg)^{1/q_2} \le  \\
           \Bigg(\sum_{j=j_0}^\infty   2^{j(s_2-s_1)q^*} \alpha_j^{q^*}  \varphi_1(2^{-j})^{(\varrho-1)q^*} \Bigg)^{1/q^*} \| \lambda \mid n^{s_1}_{\varphi_1,p_1,q_1}(Q)\|^\ast         	
    \end{align}
    with the usual modification if $q^*=\infty$. If conditions \eqref{cond2a}-\eqref{cond2b} hold, then for any $\varepsilon >0$ we find $j_0$ such that 
    \[ \Bigg(\sum_{j=j_0}^\infty   2^{j(s_2-s_1)q^*} \alpha_j^{q^*}  \varphi_1(2^{-j})^{(\varrho-1)q^*} \Bigg)^{1/q^*} < \varepsilon .\]
    Moreover, the space 
    	$$
    \widetilde{n}^{s_2}_{\varphi_2,p_2,q_2}(Q)= \{\lambda=\{\lambda_{j,m}\}_{j,m} \in n^{s_2}_{\varphi_2,p_2,q_2}(Q):  \lambda_{j,m}=0 \quad \text{if}\quad j\geq j_0 \}
    $$
    is finite-dimensional, so the unit ball of $n^{s_1}_{\varphi_1,p_1,q_1}(Q)$ is precompact in $n^{s_2}_{\varphi_2,p_2,q_2}(Q)$. \\

	\emph{Step 2 (necessity)}. If $q^*<\infty$, then the necessity of the condition \eqref{cond2a} follows from Theorem~\ref{th-cont}, since otherwise the embedding is not continuous. So it remains to prove the necessity of \eqref{cond2b} if $q^*=\infty$.  
	
    If $\big\{2^{j(s_2-s_1)} \alpha_j \frac{\varphi_1(2^{-j})^\varrho}{\varphi_1(2^{-j})}\big\} \in \ell_{\infty}\setminus c_0$, then there exists a subsequence of positive integers  $j_k$, $k\in \nat$, such that
    \begin{equation}\label{n1}
     2^{j_k(s_2-s_1)} \alpha_{j_k} \frac{\varphi_1(2^{-j_k})^\varrho}{\varphi_1(2^{-j_k})}\rightarrow \beta>0 \quad \text{if} \quad k\rightarrow \infty .
    \end{equation}
     We consider a few cases. Since the sequence $\{\alpha_{j_k}\}_{k \in \nat}$ is an increasing sequence of positive numbers, then it is convergent to a positive limit $\alpha$ or it is divergent to $\infty$.  \\
     
     \emph{Substep 2.1}. First, we assume that $\alpha_{j_k}\rightarrow \alpha$ as $k \rightarrow \infty$. Then \eqref{n1} is equivalent to 
     \begin{equation}\label{n2}
     	2^{j_k(s_2-s_1)} 	\frac{\varphi_1(2^{-j_k})^\varrho}{\varphi_1(2^{-j_k})}\rightarrow \beta_0>0 \quad \text{if} \quad k\rightarrow \infty .
     \end{equation}
     Note that the sequences $\{\varphi_1(2^{-j_k})\}_{k \in \nat}$ and $\{\varphi_2(2^{-j_k})\}_{k\in\nat}$ are decreasing. Assume first that $\displaystyle \lim_{k\rightarrow \infty}  \varphi_2(2^{-j_k})=\beta_2>0$. 
     %If $\displaystyle \lim_{k\rightarrow \infty}  \varphi_2(2^{-j_k})=\beta_2>0$, then we also have $\displaystyle\lim_{k\rightarrow \infty}  \varphi_1(2^{-j_k})=\beta_1>0$, since $\alpha_{j_k}\rightarrow \alpha$ as $k \rightarrow \infty$.
     Let us define a sequence 
     $\lambda^{(k)}= \{\lambda^{(k)}_{j,m}\}$ by
     \begin{equation}\label{n4}
     	\lambda^{(k)}_{j,m}= \begin{cases}
     	2^{-j_ks_1}\varphi_1(2^{-j_k})^{-1}& \quad j=j_k \quad\text{and}\quad m=0,\\
     	0 & \quad \text{otherwise}.
     \end{cases} 
 \end{equation} 
 Then $\|\lambda^{(k)}\mid n^{s_1}_{\varphi_1,p_1,q_1}(Q)\|^\ast=1$. Moreover,  {when $k_1\neq k_2$},
 \[
 \|\lambda^{(k_1)}-\lambda^{(k_2)}\mid n^{s_2}_{\varphi_2,p_2,q_2}(Q)\|^\ast \gtrsim 2^{j_{k_1}(s_2-s_1)}\varphi_1(2^{-j_{k_1}})^{\varrho-1} \beta_2 + 2^{j_{k_2}(s_2-s_1)}\varphi_1(2^{-j_{k_2}})^{\varrho-1} \beta_2.
 \]  
 Using \eqref{n2}, and $k_1\not= k_2$ are sufficiently large, then
\[
 \|\lambda^{(k_1)}-\lambda^{(k_2)}\mid n^{s_2}_{\varphi_2,p_2,q_2}(Q)\|^\ast \gtrsim \beta_2 \beta_0 >0,
 \]
 which allows us to conclude that the embedding is not compact. \\

 Consider now the case when $\displaystyle \lim_{k\rightarrow \infty}  \varphi_2(2^{-j_k})= 0$. 
 First we assume that $\alpha_{j_k}<\alpha$ for any $k\in \nat$. Then, for some subsequence $k_\ell$, we have $\alpha_{j_{k_\ell}}= \frac{\varphi_2(2^{-{j_{k_\ell}}})}{\varphi_1(2^{-{j_{k_\ell}}})^\varrho}$. For simplicity, we assume that the last identity is satisfied for all $j_k$. Now from \eqref{n1} we get 
 \begin{equation}\label{n6}
 	2^{j_ks_2}\varphi_2(2^{-j_k}) \sim 2^{j_k s_1}\varphi_1(2^{-j_k})
 \end{equation} 
for sufficiently large $k$. 
We put 
\begin{equation}\label{n5}	
\lambda^{(k)}_{j,m}= \begin{cases}
	2^{-j_ks_2}\varphi_2(2^{-j_k})^{-1}& \quad j=j_k \quad\text{and}\quad m=0,\\
	0 & \quad \text{otherwise}.
\end{cases}
\end{equation}  
Then, for $k$ sufficiently large, 
\[ \|\lambda^{(k)}\mid n^{s_1}_{\varphi_1,p_1,q_1}(Q)\|^\ast= 2^{j_k(s_1-s_2)}\frac{\varphi_1(2^{-j_k})}{\varphi_2(2^{-j_k})}\le C<\infty.
 %\frac{\alpha^{-1}}{2} 
\]
Furthermore, if $k_1\neq k_2$ are sufficiently large, 
\[ \|\lambda^{(k_1)}-\lambda^{(k_2)}\mid n^{s_2}_{\varphi_2,p_2,q_2}(Q)\|^\ast \ge 1.\]  
The same conclusion as above applies.  
%%%%%%%%%%%%%%%

%%%%%%%%%%%%%%%%%%%%%%%%%%%%
It remains to consider the case $\alpha=  \frac{\varphi_2(2^{-{j_0}})}{\varphi_1(2^{-{j_0}})^\varrho}$, for some $j_0\in \nat$. Then,
\begin{equation}\label{n8}
	\alpha= \frac{\varphi_2(2^{-{j_0}})}{\varphi_1(2^{-{j_0}})^\varrho} \ge  \frac{\varphi_2(2^{-{j_{k}}})}{\varphi_1(2^{-{j_{k}}})^\varrho}\qquad \text{if}\qquad j_k\ge j_0.
	\end{equation}
For simplicity, we assume that the last identity is satisfied for all $j_k$. Now from \eqref{n2} we get 
\begin{equation}\label{n7}
	2^{j_ks_2}\varphi_1(2^{-j_k})^\varrho \sim 2^{j_k s_1}\varphi_1(2^{-j_k})
\end{equation} 
for sufficiently large $k$. 
%%%%
Let $\varrho=1$.  It follows from \eqref{n7} that $s_1=s_2$. 
We consider the following sequence
\[
\lambda^{(k)}_{j,m}= \begin{cases}
	2^{-j_ks_2}\varphi_1(2^{-j_0})^{-1}& \quad j=j_k  \quad\text{and}\quad Q_{j_k,m}\subset Q_{j_0,0},\\
	0 & \quad \text{otherwise}.
\end{cases}\]  
Then 
\begin{align*}
	2^{j_ks_2}	\mathop{\sup_{\nu: \nu \leq {j_k}}}_{k\in \zd}\!  \varphi_2(2^{-\nu}) \,2^{(\nu-j_k)\frac{\nd}{p_2}}\Big(\!\! \mathop{\sum_{m\in\zd:}}_{Q_{j_k,m}\subset Q_{\nu,k}\subset Q_{j_0,0}}\!\!|\lambda^{(k)}_{j_k,m}|^{p_2}\Big)^{\frac {1}{ p_2}}= 
	\varphi_1(2^{-j_0})^{-1} \mathop{\sup_{\nu: j_0\le\nu \leq {j_k}}}_{k\in \zd}\!  \varphi_2(2^{-\nu}) \, %2^{(\nu-j_k)\frac{\nd}{p}}\Big(\!\! \mathop{\sum_{m\in\zd:}}_{Q_{j_k,m}\subset Q_{\nu,k}\subset Q_{j_0,0}}\!\!|\lambda^{(k)}_{j_k,m}|^p\Big)^{\frac 1 p}
	= \alpha %\varphi_2(2^{-j_0})\varphi_1(2^{-j_0})^{-1}
\end{align*}
and, similarly, 
\begin{align*}
	2^{j_ks_1}	\mathop{\sup_{\nu: \nu \leq {j_k}}}_{k\in \zd}\!  \varphi_1(2^{-\nu}) \,2^{(\nu-j_k)\frac{\nd}{p_1}}\Big(\!\! \mathop{\sum_{m\in\zd:}}_{Q_{j_k,m}\subset Q_{\nu,k}\subset Q_{j_0,0}}\!\!|\lambda^{(k)}_{j_k,m}|^{p_1}\Big)^{\frac {1}{ p_1}}\le
	\varphi_1(2^{-j_0})^{-1} \mathop{\sup_{\nu: j_0\le\nu \leq {j_k}}}_{k\in \zd}\!  \varphi_1(2^{-\nu}) \, %2^{(\nu-j_k)\frac{\nd}{p}}\Big(\!\! \mathop{\sum_{m\in\zd:}}_{Q_{j_k,m}\subset Q_{\nu,k}\subset Q_{j_0,0}}\!\!|\lambda^{(k)}_{j_k,m}|^p\Big)^{\frac 1 p}
	= 1. 
\end{align*}
Therefore, we get 
\[ \|\lambda^{(k_1)}-\lambda^{(k_2)}\mid n^{s_2}_{\varphi_2,p_2,q_2}(Q)\|^\ast\ge c>0\quad \text{if} \quad k_1\not= k_2 \]  
and 
\[ \|\lambda^{(k)}\mid n^{s_1}_{\varphi_1,p_1,q_1}(Q)\|^\ast = 1 %\frac{\varphi_1(2^{-j_k})}{\varphi_2(2^{-j_k})}\le 1 
\]
if $k$ is sufficiently large. 
%%%%%%

Let us now consider the case when $\varrho<1$, i.e., when $p_1<p_2$. We put $\widetilde{\varphi}_2(t)= t^{d/p_2}$. Then $\widetilde{\varphi}_2(t)\le {\varphi}_2(t)$ if $0<t\le 1$, since $\varphi_2\in \mathcal{G}_{p_2}$ and $\varphi_2(1)=1$. It follows from Theorem~\ref{th-cont} that the embedding 
\[ n^{s_2}_{\varphi_2,p_2,q_2}(Q)  \hookrightarrow  n^{s_2}_{\widetilde{\varphi}_2,p_2,q_2}(Q) \] 
is continuous. Consequently, if the embedding \eqref{embed1a} is compact, then the embedding 
\begin{equation} \label{n9}
	n^{s_1}_{\varphi_1,p_1,q_1}(Q)  \hookrightarrow  n^{s_2}_{\widetilde{\varphi}_2,p_2, q_2}(Q)
\end{equation}  
is compact, too. We prove that the last embedding is not compact if \eqref{n2} and \eqref{n8} hold. It follows from \eqref{n7} 
that the sequence $\{\varphi_1(2^{-j_k})\}_k$ is of polynomial growth, more precisely 
\begin{equation}
	\varphi_1(2^{-j_k}) \sim 2^{-j_k\frac{s_1-s_2}{1-\varrho}}.
\end{equation}    
Therefore, we can reduce this case to the problem of compactness of an embedding of  classical Besov-Morrey spaces on domains, cf. \cite{hs13}.  We have 
\begin{equation}
	\frac{s_1-s_2}{d} \le \frac{1}{p_1}-\frac{1}{p_2}
\end{equation}
since $\varphi_1\in \mathcal{G}_{p_1}$. 
We put  $\frac{s_1-s_2}{1-\varrho}=\frac{d}{u_1}$.  Then 
\[ \frac{p_1}{u_1}\bigg(\frac{1}{p_1}-\frac{1}{p_2}\bigg)= \frac{s_1-s_2}{d}\ge \max\left(0, \frac{1}{u_1}-\frac{1}{p_2}\right) \] 
Now we can deal in the same way as in Substep 2.3 of Theorem~4.1 in \cite{hs13}. Namely, we can construct a sequence $\lambda^{(k)}= \{\lambda_{j_k,m}^{(k)}\} \in 	n^{s_1}_{\varphi_1,p_1,q_1}(Q)$ bounded in the space $n^{s_1}_{\varphi_1,p_1,q_1}(Q)$ such that 
\[\| \lambda^{(k_1)}-\lambda^{(k_2)} \mid	n^{s_2}_{\widetilde\varphi_2,p_2,q_2}(Q)\|\ge c>0 \qquad \text{if}\qquad q_1\not= q_2. \] 
So the embedding \eqref{n9} is not compact, and thereafter neither the embedding \eqref{embed1a}. \\

\emph{Substep 2.2}.  Lastly, we consider the case when $\alpha_{j_k}\rightarrow \infty$ as $k \rightarrow \infty$. Passing to a subsequence, we may assume that 
$\alpha_{j_k}= \frac{\varphi_2(2^{-j_k})}{\varphi_1(2^{-j_k})^\varrho}$. 
Once more \eqref{n1} implies \eqref{n6} and we can use the sequence defined in \eqref{n5} to show that 
\[
\|\lambda^{(k)}\mid n^{s_1}_{\varphi_1,p_1,q_1}(Q)\|\le C<\infty \quad\text{and}\quad 
\|\lambda^{(k_1)}-\lambda^{(k_2)}\mid n^{s_2}_{\varphi_2,p_2,q_2}(Q)\| \ge c>0
\]   
if $k, k_1, k_2$, with $k_1\not= k_2$, are sufficiently large.  
\end{proof}

\remark{\label{n-n-class}
    If $\varphi_i(t)\sim t^{\nd/u_i}$, $0<p_i\leq u_i<\infty$, and $\mathcal{A}^{s_i}_{\varphi_i,p_i,q_i}=\mathcal{N}^{s_i}_{\varphi_i,p_i,q_i}$, $i=1,2$, then Theorem~\ref{th:comp} coincides with the sequence space version of Corollary~\ref{comp-class}(i), since $\alpha_j\sim 2^{j\nd (\frac{\varrho}{u_1}-\frac{1}{u_2})_+}$, $j\in\no$.
}

	\begin{corollary}
		If $s_1=s_2$, then the embedding \eqref{embed1} is never compact.  
	\end{corollary}
	\begin{proof}
		If $s_1=s_2$ and the embedding \eqref{embed1} holds, then $q^*=\infty$ and 
		$\alpha_j\frac{\varphi_1(2^{-j})^\varrho}{\varphi_1(2^{-j})}\sim \alpha_j$, cf. Corollary~\ref{cor2} and Corollary~\ref{cor2b}. But $\alpha_j$ is not a null sequence so \eqref{cond2b} is not satisfied.   
	\end{proof}

We have also the following counterparts of Corollaries~\ref{cor2a} and \ref{cor2b}.

	\begin{corollary}\label{nbbncomp}
	The embedding \eqref{embednb} is compact if,  and only if, \eqref{embednba} holds with $c_0$ instead of $\ell_{\infty}$ if {$q_1\leq q_2$}.  
	
	The embedding \eqref{embedbn} is compact if,  and only if, $s_1>s_2$.   	  
\end{corollary}
\bigskip~

Now we deal with the sequence spaces $\bp(Q)$, defined in analogy to \eqref{seq-besov-Om} based on $\bp(\rd)$, as the set of all double-indexed sequences 
$\lambda:=\{\lambda_{j,m}\} \subset \cc$, with $j\in\no, m \in\zd$ such that $Q_{j,m} \subset Q$,  for which the quasi-norm  
\begin{equation} \label{seq-bphi-Om}
	\| \lambda\mid \bp(Q)\|:= \sup_{P: |P|\leq 1} \frac{\varphi(\ell(P))}{|P|^{1/p}}  \Biggl(\sum_{j=j_P}^{\infty}2^{j(s-\frac{\nd}{p})q}   \Bigl(\sum_{m \in \zd: Q_{j,m}\subset P} |\lambda_{j,m}|^p\Bigr)^{\frac{q}{p}} \Biggr)^{1/q} 
\end{equation}
is finite, with the usual modification for $q=\infty$. Again, $\bp(Q)$ can be interpreted as a subspace of $\bp(\rd)$. First we study the continuity of embeddings and begin with the case $p_1\geq p_2$.

\begin{proposition}\label{P-bp-cont}
Let $s_i\in\rr$, $0<p_2\le p_1<\infty$, $0<q_i\leq \infty$, and $\varphi_i\in {\mathcal G}_{p_i}$, for $i=1,2$. 
We assume without loss of generality that $\varphi_1(1)=\varphi_2(1)=1$. 
%Let $ \varrho=\min(1,\frac{p_1}{p_2})$. 
% and $\displaystyle\alpha_j= \sup_{0\le\nu\le %j}\frac{\varphi_2(2^{-\nu})}{\varphi_1(2^{-\nu})^\varrho}$, $j\in \mathbb{N}_0$. 

There is a continuous embedding 
\begin{equation} \label{embed1-bp}
	\bpa(Q) \hookrightarrow \bpb(Q) 
\end{equation}
if, and only if, 
\begin{align}
	\label{cond2-bp}
\left\{2^{j(s_2-s_1)} \frac{\varphi_2(2^{-j})}{\varphi_1(2^{-j})}\right\}_{j\in\nat} \in \ell_\infty
\end{align}
and 
\begin{align}
	\label{cond3-bp}
s_1>s_2 \qquad \text{or}\qquad s_1=s_2 \quad \text{and}\quad q_1\leq q_2.
\end{align}
\end{proposition}

\begin{proof}  	
We adapt arguments from the proofs of \cite[Theorem~4.1]{hms22} and \cite{GHS-23}. 

  {\em Step 1}.   { We prove the sufficiency of the conditions \eqref{cond2-bp} and \eqref{cond3-bp}.} In view of $p_1\geq p_2$ we can conclude from H\"older's inequality that
  \[
    \left(\sum_{Q_{j,m}\subset P} |\lambda_{j,m}|^{p_2}\right)^{\frac{1}{p_2}} \leq 2^{\nd(j-j_P)(\frac{1}{p_2}-\frac{1}{p_1})}\left(\sum_{Q_{j,m}\subset P} |\lambda_{j,m}|^{p_1}\right)^{\frac{1}{p_1}} 
    \]
which leads to 
\begin{align*}
& \left(\sum_{j=j_P}^\infty 2^{j(s_2-\frac{\nd}{p_2})q_2} \left(\sum_{Q_{j,m}\subset P}
  |\lambda_{j,m}|^{p_2}\right)^{\frac{q_2}{p_2}} \right)^{\frac{1}{q_2}}    \leq \\  &\qquad\quad 2^{j_P(s_2-s_1 -\frac{d}{p_2}+\frac{d}{p_1})} 
  \left(\sum_{j=0}^\infty 2^{j(s_2-s_1)q^*} \right)^\frac{1}{q^*}\,   \left(\sum_{j=j_P}^\infty 2^{j(s_1-\frac{\nd}{p_1})q_1} \left(\sum_{Q_{j,m}\subset P}
    |\lambda_{j,m}|^{p_1}\right)^{\frac{q_1}{p_1}}\right)^{\frac{1}{q_1}} ,
\end{align*}  
where $\frac{1}{q^*}= (\frac{1}{q_2}-\frac{1}{q_1})_+$.  Here we applied \eqref{cond3-bp}. 
Using the $\Gp$-condition \eqref{Gp-def}, we note that
\[
  \frac{\varphi_2(\ell(P))}{|P|^{1/p_2}} = { 2^{j_P \nd(\frac{1}{p_2}-\frac{1}{p_1})}} \frac{\varphi_2(\ell(P))}{\varphi_1(\ell(P))} \frac{\varphi_1(\ell(P))}{|P|^{1/p_1}},
  \]
  such that
  \begin{align*}
%\lefteqn{
	\frac{\varphi_2(\ell(P))}{|P|^{1/p_2}} & \left(\sum_{j=j_P}^\infty  2^{j(s_2-\frac{\nd}{p_2})q_2} \left(\sum_{Q_{j,m}\subset P}
  |\lambda_{j,m}|^{p_2}\right)^{\frac{q_2}{p_2}}\right)^{\frac{1}{q_2}}\\
  & \leq {2^{j_P (s_2-s_1)}} \frac{\varphi_2(\ell(P))}{\varphi_1(\ell(P))} \frac{\varphi_1(\ell(P))}{|P|^{1/p_1}} \left(\sum_{j=j_P}^\infty  2^{j(s_1-\frac{\nd}{p_1})q_1} \left(\sum_{Q_{j,m}\subset P}
    |\lambda_{j,m}|^{p_1}\right)^{\frac{q_1}{p_1}}\right)^{\frac{1}{q_1}}\\
    &\leq   \|\lambda | \bpa(Q)\| \sup_{P: |P|\leq 1} {2^{j_P (s_2-s_1)}} \frac{\varphi_2(\ell(P))}{\varphi_1(\ell(P))} \leq c\ \|\lambda | \bpa(Q)\| 
\end{align*}  
in view of \eqref{cond2-bp}. This proves \eqref{embed1-bp}.% in case of $p_1\geq p_2$.
\\

{\em Step 2}.   We prove the necessity of the conditions \eqref{cond2-bp} and \eqref{cond3-bp}. Let us assume that the embedding \eqref{embed1-bp} is continuous. For any $j_0\in \nat$ we can define a sequence  
\[\lambda_{j,m}= 
\begin{cases}
	1 & \;\text{if}\quad j=j_0,\; m=0,\\
	0 & \; \text{otherwise}. 
\end{cases}
\]
Then 
\[ %\frac{\varphi_2(2^{-j_0})}{2^{-j_0\frac{d}{p_2}}} 2^{j_0(s_2-\frac{d}{p_2})} 
2^{j_0s_2}\varphi_2(2^{-j_0}) = \|\lambda|\bpb(Q)\|  	  \le C  \|\lambda|\bpa(Q)\| =  2^{j_0s_1}\varphi_1(2^{-j_0}) .\]
This proves that the condition \eqref{cond2-bp} is necessary.  

If $s_2> s_1$, then the sequence $\lambda_{j,m}= 2^{-jt}$, $Q_{j,m}\subset Q$, $s_2>t>s_1$, is an element of the space $\bpa(Q)$, but does not belong to $\bpb(Q)$.  So $s_2\le s_1$. If $s_1=s_2$ and $q_1>q_2$, then we can choose a positive sequence $\mu=\{\mu_j\}_j\in \ell_{q_1}\setminus \ell_{q_2}$ and put   
\[\lambda_{j,m}= 	2^{-js_1}\mu_j  \quad\text{if}\quad Q_{j,m}\subset Q . \]
 Then 
 \[\|\lambda|\bpa(Q)\|=\big(\sum_{j=0}^\infty\mu_j^{q_1}\big)^\frac{1}{q_1}<\infty,\quad
 \text{but}  \quad \|\lambda|\bpb(Q)\|=\big(\sum_{j=0}^\infty\mu_j^{q_2}\big)^\frac{1}{q_2}=\infty,
\]
so the condition \eqref{cond3-bp} is also necessary. 
\end{proof}

\begin{remark}\label{cont-special-B}
We compare our above result with the classical situation when $\varphi_i(t)\sim t^{\nd(\frac{1}{p_i}-\tau_i)}$, $0\leq\tau_i<\frac{1}{p_i}$, $i=1,2$, $p_1\geq p_2$,  and $\At=\Bt$.
Here the continuity of the embedding $\Bta(\Omega)\to\Btb(\Omega)$ can be characterised by \eqref{tau-comp-u2'}
including limiting cases, cf. \cite[Theorem~4.9]{GHS-23}. Then Proposition~\ref{P-bp-cont} coincides with our previous findings for the special case in this setting.
\end{remark}

We now concentrate on the counterpart of Proposition~\ref{P-bp-cont}  for $p_1<p_2$. Please note that, taking into account the sequence space version of the embeddings in \eqref{emb-B-N} and Theorem~\ref{th-cont}, 
 it turns out that  condition \eqref{cond2-bp} is necessary for the embedding \eqref{embed1-bp} also when $p_1<p_2$. In the same way,  with $ \varrho=\min(1,\frac{p_1}{p_2})$, the condition 
 $$
\left\{2^{j(s_2-s_1)} \varphi_1(2^{-j})^{\varrho-1}\right\}_{j\in\nat} \in \ell_\infty
$$ 
is necessary for the embedding {\eqref{embed1-bp}}. %\eqref{emb-B-N}. 
 This observation gives some motivation for the following definition.
  Depending on the functions $\varphi_1$ and $\varphi_2$, we introduce some critical smoothness indices  for later use in several propositions.
  
\begin{definition} \label{indices}
  Let $0<p_i<\infty$  for $i=1,2$ and $ \varrho=\min(1,\frac{p_1}{p_2})$.  Assume that $\varphi_i \in \Gpx{p_i}$ with $\varphi_i(1)=1$, $i=1,2$. Let $s_1\in\real$ be given.
\begin{enumerate}[{\bfseries\upshape(i)}]
	\item
%	If there exists some $c>0$ such that 
%	\begin{equation} \varphi_2(t)\leq c\ \varphi_1^\varrho(t), \quad t\in (0,1), 
%		\label{comp-1}
%	\end{equation} 
%	we define
We define 
	\begin{equation}\label{comp-1a}
		\sigma= \sigma(s_1)= \sup\{s\in\rr: \sup_{j\geq 0} 2^{j(s-s_1)} \varphi_1(2^{-j})^{\varrho-1} < \infty\}
	\end{equation}
        and
              \begin{equation}\label{sigma}
\sigma_{\infty}=\sigma_{\infty}(s_1)= \sup\{s\in\rr: \sup_{j\geq 0} 2^{j(s-s_1)} \varphi_1(2^{-j})^{-1} < \infty\}.
\end{equation}
      \item
Assume that there exists some $c>0$ such that 
	\begin{equation} \varphi_2(t)\geq c\ \varphi_1^\varrho(t), \quad t\in (0,1).
		\label{comp-2}
	\end{equation} 
Then we introduce   
	\begin{equation}\label{comp-2a}
		\overline{\sigma}=\overline{\sigma}(s_1) = \sup\{s\in\rr: \sup_{j\geq 0} 2^{j(s-\sigma(s_1))} \frac{\varphi_2(2^{-j})}{\varphi^{\varrho}_1(2^{-j})} < \infty\}.
	\end{equation}
    \end{enumerate}
\end{definition}

\begin{remark}\label{Rem-sigma}
  Let us briefly discuss the numbers $\sigma$, $\sigma_\infty$ and $\overline{\sigma}$ given by \eqref{comp-1a}, \eqref{sigma} and \eqref{comp-2a}. Note that the notation in \eqref{sigma} could be understood as the (not admitted) borderline case $p_2=\infty$ of \eqref{comp-1a}. {It is also obvious that the numbers $\sigma$, $\sigma_\infty$ and $\overline{\sigma}$ do not only depend on $s_1$, but also on $p_i$ and $\varphi_i$, $i=1,2$. But for convenience we decided to consider these numbers as bounds for the smoothness $s_2$ in relation to $s_1$ mainly, while regarding all other parameters as fixed.}
  
\begin{enumerate}[{\bfseries\upshape(i)}]
\item
 % Assume that \eqref{comp-1} is satisfied. 
If $\lim_{j\to\infty} \varphi_1(2^{-j}) =c>0$ or $\varrho=1$, then  $\sigma(s_1)=s_1$. 
If $\lim_{j\to\infty} \varphi_1(2^{-j}) =0$ and $\varrho<1$, then $\varphi_1(2^{-j})^{\varrho-1} \to\infty$ for $j\to\infty$. Moreover, $\varphi_1\in\Gpx{p_1} $ implies $\lim_{j\to\infty} 2^{j\nd/p_1} \varphi_1(2^{-j})>0$, that is, 
\[\lim_{j\to\infty} 2^{j (\varrho-1)\frac{\nd}{p_1}} \varphi_1(2^{-j})^{\varrho-1}<\infty.\]
Hence $\sigma(s_1)$ is always well-defined, with $s_1-\frac{\nd}{p_1}(1-\varrho)\leq \sigma(s_1)\leq s_1$. %Note that this is independent of \eqref{comp-1}.
The same  argument can be applied to $\sigma_\infty(s_1)$, leading finally to  $s_1-\frac{\nd}{p_1}\leq \sigma_\infty(s_1)\leq s_1$.
\item
Now assume that \eqref{comp-2} is satisfied. This implies 
$\varphi_1(t)\leq \varphi_1(t)^\varrho\leq c\ \varphi_2(t)$, $0<t<1$, and, by \eqref{comp-2a},  $\overline{\sigma}(s_1)\leq \sigma(s_1)\leq s_1$. For the estimate from below we claim, that $\overline{\sigma}(s_1)\geq s_1- \frac{\nd}{p_1}$. This can be seen as follows. Recall  $\varphi_2(2^{-j})\leq 1$, $j\in\no$, since  $\varphi_2\in \Gpx{p_2}$ { and $\varphi_2(1)=1$}. Let $\varrho=1$. Now $\sigma(s_1)=s_1$ and $\varphi_1\in \Gpx{p_1}$ that implies $\varphi_1(2^{-j}) \geq 2^{-j\nd/p_1}$, $j\in\no$, and this results in $\overline{\sigma}(s_1)\geq s_1-\nd/p_1$.  
	If $\varrho<1$, then  $\varphi^\varrho_1\in \Gpx{p_2}$, and thus $\varphi_1^\varrho(2^{-j}) \geq 2^{-j\nd/p_2}$, $j\in\no$, leading to $\overline{\sigma}(s_1)\geq \sigma(s_1)-\frac{\nd}{p_2} \ge  s_1-\frac{\nd}{p_1}(1-\varrho)-\frac{\nd}{p_2} = s_1-\frac{\nd}{p_1}$. So $\overline{\sigma}(s_1)$ is well-defined in all cases with the above estimates. 
	 
\end{enumerate}
\end{remark}

\begin{example}\label{ex-class}
For later comparison, let us exemplify the above numbers for the `classical' case $\varphi_i(t)\sim t^{\nd(\frac{1}{p_i}-\tau_i)}$, $0\leq\tau_i<\frac{1}{p_i}$, $i=1,2$. Then 
\[
\sigma(s_1)=\begin{cases} s_1, & \varrho=1, \\ s_1-\frac{\nd}{p_1}+\frac{\nd}{p_2} + \nd \tau_1(1-\varrho), &\varrho<1. \end{cases} 
\]
In case of $\sigma_\infty$ given by \eqref{sigma} we find in this setting
$  \sigma_\infty(s_1)= s_1-\frac{\nd}{p_1} + \nd \tau_1$, 
which refers to the above $\sigma(s_1)$ with $p_2\to \infty$, and thus $\tau_2\to 0$, $\varrho\to 0$. Finally,  condition \eqref{comp-2} means in this context that 
\[
\begin{cases} \frac{1}{p_1}-\tau_1-\frac{1}{p_2}+\tau_2 \geq 0, & \varrho=1, \\ \tau_2- \tau_1 \varrho \geq 0, & \varrho<1, \end{cases} 
\qquad\text{and thus}\quad 
%\overline{\sigma}(s_1)=\begin{cases} s_1-\nd\left(\frac{1}{p_1}-\tau_1-\frac{1}{p_2}+\tau_2\right), & \varrho=1, \\ s_1-\nd(\tau_2- \tau_1 \varrho), & \varrho<1, \end{cases}  
{\overline{\sigma}(s_1)=s_1-\nd\left(\frac{1}{p_1}-\tau_1-\frac{1}{p_2}+\tau_2\right)}
\]
in view of \eqref{comp-2a}.
\end{example}

%as well as Definition \ref{indices} below.
Recall that we characterised in Proposition~\ref{P-bp-cont} the continuity of the embedding $\bpa(Q)\hookrightarrow \bpb(Q)$ in case of $p_1\geq p_2$. Now we deal with the case $p_1<p_2$ and start with the limiting case {$q_2=\infty$} and $\varphi_2 \sim 1$.

{

\begin{proposition}\label{pinfinity}
Let  $0<p<\infty$ and $\varphi\in {\mathcal G}_{p}$. We assume without loss of generality that $\varphi(1)=1$.  Let $s_1\in\rr$, $0<q \leq \infty$,  and let  
	\begin{equation}\label{liminf}
	\sup_{j\geq 0} 2^{j(\sigma_\infty(s_1)-s_1)} \varphi(2^{-j})^{-1} < \infty	.
      \end{equation}
      \begin{enumerate}[\bfseries\upshape  (i)]
%	for some $\sigma\ge 0$. 
\item	There is a continuous embedding   
	\begin{equation}\label{pinfty}
		b^{s_1,\varphi}_{p,q}(Q)\hookrightarrow b^{s_2}_{\infty,\infty}(Q)
	\end{equation}
	 if, and only if, 
	$s_2\le \sigma_\infty(s_1)$.  
      \item The embedding \eqref{pinfty} is  compact   if, and only if, 
      		\begin{equation}\label{comp-nec}
      			s_2 < \sigma_\infty(s_1) \quad\text{or}\quad s_2 = \sigma_\infty(s_1) \quad\text{and}\quad \limsup_{j\rightarrow \infty} 2^{j(\sigma_\infty(s_1)-s_1)} \varphi(2^{-j})^{-1}=0.
      		\end{equation} 
%      \item If $s_2 < \sigma_\infty(s_1)$, then the embedding \eqref{pinfty} is  compact.\\
%        If,  in addition,
%        \begin{equation}\label{comp-nec}
%          \limsup_{j\rightarrow \infty} 2^{j(\sigma_\infty(s_1)-s_1)} \varphi_1(2^{-j})^{-1}>0,
%        \end{equation}
%        then the embedding \eqref{pinfty} is compact if, and only if,  %$s_2<\sigma_\infty(s_1)$.  
\end{enumerate} 
\end{proposition}

\begin{proof}
 {\em Step 1}.~ First we deal with (i) and show that $s_2\leq \sigma_\infty(s_1)$ implies the continuity of \eqref{pinfty}. It is sufficient to prove that the embedding holds for $q=\infty$ and $s_2= \sigma_\infty(s_1)$. 
 %If \eqref{liminf} holds, then $\inf_j 2^{j(s_1-\sigma_\infty(s_1))}\varphi(2^{-j})=c>0$. 
	Let $P=Q_{j_0,m_0}$.  We have,  {by \eqref{liminf},}
	\begin{align*}
		2^{j_0s_2}|\lambda_{j_0,m_0}| & =  2^{j_0\frac{d}{p}}\varphi(2^{-j_0}) 2^{j_0(s_1-\frac{d}{p})}|\lambda_{j_0,m_0}| 2^{j_0(s_2-s_1)} \varphi(2^{-j_0})^{-1}\\
		& \le c\; 2^{j_0\frac{d}{p}}\varphi(2^{-j_0})\, \sup_{j\ge j_0} 2^{j(s_1-\frac{d}{p})}\left( \sum_{Q_{j,m}\subset P}|\lambda_{j,m}|^p\right)^\frac{1}{p} \le c\; \|\lambda|b^{s_1,\varphi}_{p,\infty}(Q)\|,
	\end{align*}
        which proves \eqref{pinfty}, since  it is sufficient to take  the supremum over the cubes $Q_{j_0,m_0}$.

 Now we deal with the necessity of $s_2\leq \sigma_\infty(s_1)$ for the embedding \eqref{pinfty}. More precisely, $s_2> \sigma_\infty(s_1)$ means  
 	\[	\sup_{j\geq 0} 2^{j(s_2-s_1)} \varphi(2^{-j})^{-1} = \infty\] 
in view of \eqref{sigma} with $\varphi_1=\varphi$. We prove that then the embedding \eqref{pinfty} does not hold. We choose a strictly increasing sequence of positive integers $\{j_k\}_k$ such that
 	  \begin{equation}	\lim_{k\rightarrow \infty} 2^{j_k(s_2-s_1)} \varphi(2^{-j_k})^{-1} = \infty . \label{LS-5}
     \end{equation}
     The sequence $\{2^{j\frac{d}{p}} \varphi(2^{-j})\}_{j\in \nat}$ is increasing since $\varphi\in \mathcal{G}_p$. First we assume that  $\lim_{j\rightarrow \infty} 2^{j\frac{d}{p}} \varphi(2^{-j})=\infty$. In this case we can find a subsequence $\{j_{k_\ell}\}_{\ell\in \nat}$ of the sequence $\{j_k\}_{k\in \nat}$ such that 
     \begin{equation} 2^{j_{k_\ell}\frac{d}{p}} \varphi(2^{-j_{k_\ell}})\ge 2 \cdot 2^{j_{k_{\ell-1}}\frac{d}{p}} \varphi(2^{-j_{k_{\ell-1}}}) .
\label{LS-4}
     \end{equation}
     Let
     \[\lambda_{j,m}=
     \begin{cases}
     	2^{-j s_1} \varphi(2^{-j})^{-1} &\text{if}\quad j=j_{k_\ell}\; \text{ and} \; m=0,\\
     	0 & \text{otherwise.}
     \end{cases}\]
     Then, also in view of \eqref{LS-4}, 
     \begin{align}
     	\| \lambda |b^{s_1,\varphi}_{p,q}(Q)\| = &\sup_{\ell\in\nat} \frac{\varphi(2^{-j_{k_\ell}})}{2^{-j_{k_\ell}\frac{d}{p}}} \left(\sum_{\nu\ge \ell} 2^{j_{k_\nu}(s_1-\frac{d}{p})q} |\lambda_{j_{k_\nu},0}|^q     \right)^\frac{1}{q} \nonumber \\ 
     =   & \sup_{\ell\in\nat} \frac{\varphi(2^{-j_{k_\ell}})}{2^{-j_{k_\ell}\frac{d}{p}}} \left(\sum_{\nu\ge \ell} 2^{-j_{k_\nu}\frac{d}{p}q} \varphi(2^{-j_{k_\nu}})^{-q}     \right)^\frac{1}{q}  \nonumber \\ 
  \leq     & \sup_{\ell\in\nat} \frac{\varphi(2^{-j_{k_\ell}})}{2^{-j_{k_\ell}\frac{d}{p}}} \left(\sum_{\nu\ge \ell} 2^{(\ell-\nu)q}2^{-j_{k_\ell}\frac{d}{p}q} \varphi(2^{-j_{k_\ell}})^{-q}     \right)^\frac{1}{q} \le C<\infty.       \nonumber 
     \end{align}
     On the other hand, by \eqref{LS-5},
     \[\sup_{j,m} 2^{js_2} |\lambda_{j,m}| = \sup_{\ell\in\nat} 2^{j_{k_\ell}(s_2-s_1)} \varphi(2^{-j_{k_\ell}})^{-1} = \infty .
     \]
     So \eqref{pinfty} cannot hold. Finally we need to consider the case when $\lim_{j\rightarrow \infty} 2^{j\frac{d}{p}} \varphi(2^{-j}) = c<\infty$. But this implies $\varphi(2^{-j}) \sim 2^{-j\frac{d}{p}}$ and the space $b^{s_1,\varphi}_{p,q}(Q)$ coincides with the classical Besov sequence space $b^{s_1}_{p,q}(Q)$. In that case the statement is well known.     \\
        
{\em Step 2}.~ {
  We now deal with the compactness in (ii). First,  assume $s_2 < \sigma_\infty(s_1)$. Let $\|\lambda|b^{s_1,\varphi}_{p,q}(Q)\|\le 1$. For any  $\varepsilon>0$ we can choose $j_0\in \nat$ such that  $2^{j(s_2-\sigma_\infty(s_1))}\le \varepsilon$ 
  for any $j\ge j_0$. Then, in a similar way as above, 
  \[2^{js_2}|\lambda_{j,m}|\le c\,2^{j(s_2-\sigma_\infty(s_1))}\|\lambda |b^{s_1,\varphi}_{p,q}(Q)\|\le  c \,\varepsilon, \qquad j\ge j_0.\] 
    The subspace  
    $\{\lambda\in b^{s_1,\varphi}_{p,q}(Q): \lambda_{j,m}=0 \quad\text{if}\quad {j >j_0}\}$ 
    is finite-dimensional, therefore the image of the unit ball  of the space $b^{s_1,\varphi}_{p,q}(Q)$ is precompact  in $b^{s_2}_{\infty,\infty}(Q)$. \\
    If $s_2=\sigma_\infty{(s_1)}$  and $\limsup_{j\rightarrow \infty} 2^{j(\sigma_\infty{(s_1)}-s_1)} {\varphi}(2^{-j})^{-1}=0$, then  
    	\[\lim_{j\rightarrow \infty} 2^{j(\sigma_\infty{(s_1)}-s_1)} {\varphi}(2^{-j})^{-1}=\limsup_{j\rightarrow \infty} 2^{j(\sigma_\infty{(s_1)}-s_1)} {\varphi}(2^{-j})^{-1}=0,\]
    	since  the sequence is positive. For any  $\varepsilon>0$ there exists a number $j_0$ such that for any $j\ge j_0$ we have 
    	\begin{equation}\label{comr2}
    		2^{j\sigma_\infty{(s_1)}} \le \varepsilon 2^{js_1}\varphi(2^{-j}) .
    	\end{equation}
    	Let $\|\lambda|b^{s_1,\varphi}_{p,q}(Q)\|\le 1$. We represent $\lambda$ as a sum $\lambda=\lambda^{(1)}+\lambda^{(2)}$ where 
    	\[\lambda^{(1)}_{j,m} =
    	\begin{cases}
    		\lambda_{j,m} & \; \text{if}\; 0\le j\le j_0, \\
    		0 & \;\text{otherwise}. 
    	\end{cases}
    	\]  
    	It follows from \eqref{comr2} that 
    	\begin{align}
    		2^{j\sigma_\infty(s_1)} |\lambda_{j,m}|\le \varepsilon 2^{j\frac{d}{p}}\varphi(2^{-j}) 2^{j(s_1-\frac{d}{p})}|\lambda_{j,m}| \le \varepsilon \|\lambda|b^{s_1,\varphi}_{p,q}(Q)\| \leq \varepsilon  \quad \text{if}\quad 	j>j_0.
    	\end{align}    
    	On the other hand, the subspace of  sequences with $\lambda_{j,m}=0$ for $j> j_0$ is finite-dimensional. 
    	%So the subset of these sequences with norm less or equaled $1$ is precompact. 
    	This implies the compactness of the embedding \eqref{pinfty}. 
	  
	   {Now we prove the necessity of the condition \eqref{comp-nec}. In view of (i), it remains to show that if  $\limsup_{j\rightarrow \infty} 2^{j(\sigma_\infty(s_1)-s_1)} \varphi(2^{-j})^{-1}>0$, then the compactness of \eqref{pinfty}  implies $s_2<\sigma_\infty(s_1)$. }
%\red{If  $\limsup_{j\rightarrow \infty} 2^{j(\sigma_\infty-s_1)} \varphi_1(2^{-j})^{-1}=0$  then any continuous embedding  \eqref{pinfty} is compact, so it remains to show that if } $\limsup_{j\rightarrow \infty} 2^{j(\sigma_\infty(s_1)-s_1)} \varphi(2^{-j})^{-1}>0$, then the compactness of \eqref{pinfty} also implies $s_2<\sigma_\infty(s_1)$. 
This can be seen as follows. Assume $s_2=\sigma_\infty(s_1)$. Let $\{j_k\}_{k\in\nat}$ be an increasing sequence of integers such that 
	\[		\lim_{k\rightarrow \infty} 2^{j_k(s_2-s_1)} \varphi(2^{-j_k})^{-1}= c>0 , \qquad s_2=\sigma_\infty(s_1). \] 
	We take  sequences $\lambda^{(k)}$, $k\in \nat$, of elements of $b^{s_1,\varphi}_{p,q} (Q)$ defined in the following way, 
	\[\lambda_{j,m}^{(k)}=
	\begin{cases}
		2^{-j s_1} \varphi(2^{-j})^{-1} &\text{if}\quad j= k\; \text{ and} \; m=0,\\
		0 & \text{otherwise.}
	\end{cases}\]
	Then $ \|\lambda^{(k)}|b^{s_1,\varphi}_{p,q}(Q)\|=1$, but 
	\[ \|\lambda^{(k)}- \lambda^{(\ell)}|b^{s_2}_{\infty,\infty}(Q)\| =  \max \{2^{k(s_2- s_1)} \varphi(2^{-k})^{-1}, 2^{\ell(s_2- s_1)} \varphi(2^{-\ell})^{-1}\}\ge \frac{c}{2}, \]
	for $k,\ell$ sufficiently large. This contradicts the assumed compactness of \eqref{pinfty}.}
  \end{proof}

{
\begin{example} 
Consider the function $\varphi$ given by \eqref{exp1}, from Examples~\ref{ex-phi}(vi), assuming that the parameter $a$ is chosen sufficiently large so that $\varphi \in \mathcal{G}_p$. In this case we have $\sigma_\infty(s_1)=s_1$ and $\limsup_{j\rightarrow \infty} 2^{j(\sigma_\infty(s_1)-s_1)} \varphi(2^{-j})^{-1}=0$. 
%	Note that the above additional assumption \eqref{comp-nec} is not always satisfied.  If we %take  the function $\varphi$ from Examples~\ref{ex-phi}(v) we have $\sigma_\infty(s_1)=s_1$ and %$\liminf_{j\rightarrow \infty} 2^{j(\sigma_\infty(s_1)-s_1)} \varphi(2^{-j})^{-1}=0$.  \\

In the case of the function $\psi$ in  \eqref{ex-psi}, from Examples~\ref{ex-phi}(vii), which belongs to $\mathcal{G}_p$ when $p \le d$,
%\[\psi(t)=\begin{cases}
%	(e\,t)^\frac{d}{p}\ln(t^{-1})&\; \text{if} \quad 0<t<e^{-1},\\
%	1& \;\text{if}\quad t\ge e^{-1}.  
%\end{cases}
%\]
%belongs to $\mathcal{G}_p$ \magenta{if $p \le d$}. 
 it holds  $\sigma_\infty{(s_1)}=s_1-\frac{d}{p}$ and 
 \[ \sup_{j\ge 0} 2^{j(\sigma_\infty{(s_1)}-s_1)} \psi(2^{-j})^{-1}<\infty,\quad\text{but}\quad \limsup_{j\rightarrow \infty} 2^{j(\sigma_\infty{(s_1)}-s_1)} \psi(2^{-j})^{-1}=0. 
\] 
\end{example}

\begin{remark}
If $ \sup_{j\geq 0} 2^{j(\sigma_\infty(s_1)-s_1)} \varphi(2^{-j})^{-1}=\infty$, then, in a similar way, we can prove that the embedding \eqref{pinfty} is continuous and compact if $s_2<\sigma_\infty(s_1)$.  
\end{remark}
}

Now we approach the situation when $0<p_1<p_2<\infty$. 
Let us state first, for convenience, the following Gagliardo-Nirenberg type inequality which might be well-known, but we have not found an immediate reference. The result follows by H\"older's inequality. 

 \begin{lemma}\label{G-N}
 Let $0<q_1,q_2\le \infty$, $-\infty<s_2<s_1<\infty$ and $0<\theta<1$. 
\begin{enumerate}[\bfseries\upshape  (i)]
\item	Let $0<p_1,p_2<\infty$. We assume that 
	\[ s=s_1(1-\theta)+ \theta s_2, \quad \frac{1}{p}=\frac{1-\theta}{p_1}+\frac{\theta}{p_2},\quad \frac{1}{q}=\frac{1-\theta}{q_1}+\frac{\theta}{q_2}.\]  
	Let $\varphi_1\in {\mathcal G}_{p_1}$, $\varphi_2\in {\mathcal G}_{p_2}$ and 
	\[ \varphi(t)=\varphi_1^{1-\theta}(t)\varphi_2^\theta(t), \qquad t>0.\]
	Then $\varphi\in {\mathcal G}_{p}$ and there exists a positive constant $C>0$ such that
	the inequality 
	\begin{equation}\label{GN}
		\|\lambda|b^{s,\varphi}_{p,q}(Q)\| \le C \|\lambda|b^{s_1,\varphi_1}_{p_1,q_1}(Q)\|^{1-\theta}\, \|\lambda|b^{s_2,\varphi_2}_{p_2,q_2}(Q)\|^\theta  
	\end{equation}
	holds for all sequences $\lambda=\{\lambda_{j,m}\}_{j,m}\subset \cc$.
\item Let $0<p_1<\infty$. We assume that 
	\[ s=s_1(1-\theta)+ \theta s_2, \quad \frac{1}{p}=\frac{1-\theta}{p_1},\quad \frac{1}{q}=\frac{1-\theta}{q_1}+\frac{\theta}{q_2}.\]  
	Let $\varphi_1\in {\mathcal G}_{p_1}$ and 
	\[ \varphi(t)=\varphi_1^{1-\theta}(t), \qquad t>0.\]
	Then $\varphi\in {\mathcal G}_{p}$ and there exists a positive constant $C>0$ such that
	the inequality 
	\begin{equation}\label{GNinfty}
		\|\lambda|b^{s,\varphi}_{p,q}(Q)\\| \le C \|\lambda|b^{s_1,\varphi_1}_{p_1,q_1}(Q)\|^{1-\theta}\, \|\lambda|b^{s_2}_{\infty,q_2}(Q)\|^\theta 
	\end{equation}
	holds for all sequences $\lambda=\{\lambda_{j,m}\}_{j,m}\subset \cc$.
	\end{enumerate}
\end{lemma}
%\begin{proof}
%It follows by H\"older's  inequality. 
%\end{proof}

We study the embedding $\bpa(Q)\hookrightarrow \bpb(Q)$ for $0<p_1<p_2<\infty$ and deal first with the case converse to \eqref{comp-2}, that is, $\varphi_2(t)\leq C\ \varphi_1^\varrho(t)$, $t\in (0,1)$.

\begin{proposition}\label{P-bp-cont2}
Let $0<p_1<  p_2<\infty$,   $s_i\in\rr$, $0<q_i\leq \infty$,  and $\varphi_i\in {\mathcal G}_{p_i}$, for $i=1,2$. Let $\varrho=\frac{p_1}{p_2}$ and assume without loss of generality that $\varphi_1(1)=\varphi_2(1)=1$. Assume  further that 
	\begin{equation} 	\label{comp-1}
\varphi_2(t)\le C \varphi^\varrho_1(t), \quad 0<t\le 1,
	\end{equation}  
 and  
	\begin{equation} \label{emb-lim1}
		\sup_{j\geq 0} 2^{j(\sigma(s_1)-s_1)} \varphi_1(2^{-j})^{\varrho-1} < \infty .
	\end{equation}   
	Then there is a continuous embedding 
		\begin{equation}  \label{emb-lim2}
			\bpa(Q) \hookrightarrow b^{\sigma(s_1), \varphi_2}_{p_2,q_2}(Q) 
		\end{equation}
		if $q_1\le \varrho q_2$. 
	\end{proposition}
	
\begin{proof}
	It is sufficient to consider the case $\varphi_2(t)=\varphi_1^\varrho(t)$, $0<t\le 1$,  { and $q_1=\varrho q_2$}. It follows from  \eqref{emb-lim1} that 
	\begin{equation} \label{emb-lim3}
		\sup_{j\geq 0} 2^{j\frac{\sigma(s_1)-s_1}{1-\varrho}} \varphi_1(2^{-j})^{-1} < \infty .
	\end{equation}   
	We take $s_3\in \rr$ such that $\frac{s_3-\sigma(s_1)}{\varrho}=\frac{\sigma(s_1)-s_1}{1-\varrho}$. Then 
     \[ {\sigma(s_1)}=\varrho s_1 + (1-\varrho)s_3\qquad\text{and}\qquad   \frac{1}{p_2}=\frac{\varrho}{p_1} .\]
     As a consequence of the second part of Lemma~\ref{G-N}, we get 
     	\begin{equation}\label{GNinfty2}
     	\|\lambda|b^{\sigma(s_1),\varphi_2}_{p_2,q_2}(Q)\| \le C \|\lambda|b^{s_1,\varphi_1}_{p_1,q_1}(Q)\|^{\varrho}\, \|\lambda|b^{s_3}_{\infty,\infty}(Q)\|^{1-\varrho},
     \end{equation}
  {with $q_1=\varrho q_2$}. Moreover, the inequality \eqref{emb-lim3} gives us
     	\begin{equation} %\label{emb-lim3}
     	\sup_{j\geq 0} 2^{j(s_3-\sigma(s_1))} \varphi_1(2^{-j})^{-\varrho} < \infty .
     \end{equation}   
      So,  the Proposition~\ref{pinfinity}(i) implies 
     	\begin{equation}\label{pinfty2}
     	b^{\sigma(s_1),\varphi_2}_{p_2,q_2}(Q)\hookrightarrow b^{s_3}_{\infty,\infty}(Q)
     \end{equation}   
     Now \eqref{emb-lim2} follows by combining  \eqref{GNinfty2} with \eqref{pinfty2}. 
\end{proof}

\begin{remark}
  If $0<p_1<p_2<\infty$, \eqref{comp-1} shall hold,  %$\varphi_2(t)\le C \varphi^\varrho_1(t)$
  and $s_1-s_2= \frac{d}{p_1}-\frac{d}{p_2}$,	then there is always a continuous embedding  			$\bpa(Q) \hookrightarrow b^{s_2, \varphi_2}_{p_2,q_2}(Q) $ at least if $q_1\le \varrho q_2$,  since 
  \[2^{j(s_2-s_1)}\varphi_1^{\varrho-1}(2^{-j})\le 2^{j(s_2-s_1+\frac{d}{p_1}-\frac{d}{p_2})} \]
  and thus $\sigma(s_1)\geq s_2$.
   But the embedding may hold with no restriction for $q's$  and be compact if 
  $s_2= s_1-\frac{d}{p_1}+\frac{d}{p_2}<\sigma(s_1)$.    			
\end{remark}

\begin{example}  {
%		Let $a \ge e$. If  
%	\begin{equation}\label{exp1}
%		\varphi(t)= \begin{cases}
%			(\ln a)\,  (\ln t^{-1})^{-1} & \text{if} \; 0<t\le a^{-1},\\
%			1  & \text{if} \; t \ge a^{-1}
%		\end{cases}
%	\end{equation}
%	then $\varphi\in \mathcal{G}_r$, where $r=d\ln a$.  
	Let $p_1<p_2$. Consider once more the  function $\varphi$ given by \eqref{exp1} in Examples~\ref{ex-phi}(vi), assuming that the parameter $a$ is chosen sufficiently large so that $\varphi \in \mathcal{G}_{p_1}$. }
	We have  
	\begin{align*}
		\sup_{j\in\no} 2^{j(s-s_1)} \varphi(2^{-j})^{\varrho-1}= (\ln a)^{\varrho-1} \sup_{j\in\no} 2^{j(s-s_1)} j^{1-\varrho} < \infty
	\end{align*}
for any $s<s_1$, so $\sigma(s_1)=s_1$. But, plainly, the condition \eqref{emb-lim1} is not satisfied. 
\end{example}

Finally we consider the case $0<p_1<p_2<\infty$ with  \eqref{comp-2}, that is, $\varphi_2(t)\geq c\ \varphi_1^\varrho(t)$, $t\in (0,1)$.

\begin{corollary}\label{P-bp-cont3}
Let $0<p_1<  p_2<\infty$,   $s_i\in\rr$, $0<q_i\leq \infty$,  and $\varphi_i\in {\mathcal G}_{p_i}$, for $i=1,2$. Let $\varrho=\frac{p_1}{p_2}$ and assume without loss of generality that $\varphi_1(1)=\varphi_2(1)=1$.  Let 
$\varphi_2(t)\ge c \varphi^\varrho_1(t)$ for $0<t\le 1$. We assume that the conditions \eqref{emb-lim1}  and  
	\begin{align} \label{emb-lim21}
%		\sup_{j\geq 0} 2^{j(\sigma(s_1)-s_1)} \varphi_1(2^{-j})^{\varrho-1} < \infty 
%		\intertext{and}
%			\label{cond2-bp}
		\sup_{j\geq 0}  2^{j(\overline{\sigma}(s_1)-\sigma(s_1))} \frac{\varphi_2(2^{-j})}{\varphi^\varrho_1(2^{-j})}< \infty 
	\end{align}
        are satisfied. 	Then there is a continuous embedding 
	\begin{equation}  %\label{emb-lim2}
		\bpa(Q) \hookrightarrow b^{\overline{\sigma}(s_1), \varphi_2}_{p_2,q_2}(Q) 
	\end{equation}
	if $q_1\le \varrho q_2$.  
      \end{corollary}
      
\begin{proof}
  To prove the proposition we use the following factorisation
  \begin{equation}
 \bpa(Q) \hookrightarrow b^{\sigma(s_1), \varphi^\varrho_1}_{p_2,q_2}(Q) \hookrightarrow b^{\overline{\sigma}(s_1), \varphi_2}_{p_2,q_2}(Q). 
\end{equation}	
The continuity of the first embedding follows from Proposition~\ref{P-bp-cont2}  whereas the second embedding follows from \eqref{emb-lim21} and Proposition~\ref{P-bp-cont}.   
\end{proof}
}

  \begin{remark}
    Note that if $\varphi_1, \varphi_2$ satisfy \eqref{emb-lim1}  and  
 \eqref{emb-lim21}, then we arrive at \eqref{cond2-bp} with $s_2=\overline{\sigma}(s_1)$. So one might be tempted to extend Proposition~\ref{P-bp-cont} to the case $\varrho<1$ (with possibly $q_1\leq \varrho q_2$). However, this argument is not yet covered by our splitting technique in the above proof where we need both conditions \eqref{emb-lim1}  and   \eqref{emb-lim21} to be satisfied separately.    
  \end{remark}

  \begin{example}
  We return to the setting in Example~\ref{ex-class}. Assume $0<p_1<p_2<\infty$, $0\leq \tau_i<\frac{1}{p_i}$, $\varphi_i(t)\sim t^{d(\frac{1}{p_i}-\tau_i)}$, $i=1,2$. Then 
  $\varphi_2(t)\geq c \varphi^{{\varrho}}_1(t)$, $t\in (0,1)$,  means $\tau_2\geq \varrho \tau_1$, conditions \eqref{emb-lim1}  and \eqref{emb-lim21} are always satisfied, such that we arrive at the limiting embedding
  \[
    b^{s_1,\tau_1}_{p_1,q_1}(Q) \hookrightarrow b^{s_1-d(\frac{1}{p_1}-\tau_1-\frac{1}{p_2}+\tau_2), \tau_2}_{p_2,q_1 p_2/p_1}(Q) .
    \]
This corresponds to the sequence space version of \cite[Theorem~4.9]{GHS-23}.
\end{example}

We now study the compactness of \eqref{embed1-bp}, that is,
\begin{equation*} %\label{embed1-bp}
	\bpa(Q) \hookrightarrow \bpb(Q) 
\end{equation*}
and may assume for convenience, that $Q$ is the unit cube $Q_{0,0}$.  
%Again we deal first with the case \eqref{comp-1}, i.e., $\varphi_2(t)\le C \varphi^\varrho_1(t)$, $0<t\le 1$. 

\begin{proposition}\label{Lemma-LS_1+2}
Let $s_i\in\rr$, $0<p_i<\infty$, $0<q_i\leq \infty$, and $\varphi_i\in {\mathcal G}_{p_i}$, for $i=1,2$, and $ \varrho=\min(1,\frac{p_1}{p_2})$. {We assume without loss of generality that $\varphi_1(1)=\varphi_2(1)=1$.}
\begin{enumerate}[\bfseries\upshape  (i)]
 \item Assume that \eqref{comp-1} is satisfied, i.e., there exists some $c>0$ such that 
$$
\varphi_2(t)\leq c\ \varphi_1^\varrho(t), \quad t\in (0,1).
$$
Then the embedding \eqref{embed1-bp} is compact, if $s_2<\sigma(s_1)$, given by \eqref{comp-1a}.
\item
  Assume that \eqref{comp-2} is satisfied, i.e., there exists some $c>0$ such that 
$$ \varphi_2(t)\geq c\ \varphi_1^\varrho(t), \quad t\in (0,1). $$
{Then  the embedding \eqref{embed1-bp} is compact, if $s_2<\overline{\sigma}(s_1)$, given by \eqref{comp-2a}.}
\end{enumerate}
\end{proposition}

\begin{proof}
  {\em Step 1}. We first prove (i). \\
  {\em Substep 1.1}~ If $\varrho=1$,  the statement is obvious since $\varphi_2(t)\leq C \varphi_1(t)$,   so the condition \eqref{cond2-bp} is satisfied with the same  smoothness in the source and the target space.  By Proposition~\ref{P-bp-cont} we thus conclude that
	\[
	\bpa(Q) \hra b^{s_1,\varphi_2}_{p_2,q_1}(Q)\hra \bpb(Q), 
	\]
	and the second embedding is compact when $s_2<s_1$  as assumed in this case. This follows by parallel arguments as presented in Step~1 of the proof of Theorem~\ref{th:comp}.

{\em Substep 1.2}. Now assume $\varrho<1$, that is, $p_1<p_2$. Then we may assume $s_1=0$ and $s_2<0$ by some lift argument, since $\bpa(Q)\hra \bpb(Q)$ if, and only if, $b^{0,\varphi_1}_{p_1,q_1}(Q) \hra b^{s_2-s_1,\varphi_2}_{p_2,q_2}(Q)$, including the compactness of the embeddings. So we are left to show the compactness of the latter embedding. Let $\lambda \in b^{0,\varphi_1}_{p_1,q_1}(Q)$ be such that $\|\lambda | b^{0,\varphi_1}_{p_1,q_1}(Q)\|=1$, that is,
\[
\sup_{P\subset Q} \frac{\varphi_1(2^{-j_P})}{|P|^{1/p_1}} \left(\sum_{j=j_P}^\infty 2^{-j\frac{\nd}{p_1}q_1} \left(\sum_{Q_{j,m}\subset P} |\lambda_{j,m}|^{p_1}\right)^{\frac{q_1}{p_1}}\right)^{{\frac{1}{q_1}}}=1.
\]
Then for all $j\geq j_P$,
\begin{equation}\label{LS-2}
 \varphi_1(2^{-j_P}) 2^{(j_P-j)\frac{\nd}{p_1}} \left(\sum_{Q_{j,m}\subset P} |\lambda_{j,m}|^{p_1}\right)^{\frac{1}{p_1}}\leq 1,
\end{equation}
leading to
\begin{equation} \label{LS-2a}
\sum_{Q_{j,m}\subset P} |\lambda_{j,m}|^{p_1} \leq   \varphi_1(2^{-j_P})^{-p_1} 2^{-(j_P-j)\nd}.
\end{equation}
Taking $P=Q_{j,m}$ we get for all $(j,m)\in \no\times \zd$ with $Q_{j,m}\subset Q$ by \eqref{LS-2} that
\begin{equation}\label{LS-3}
\varphi_1(2^{-j}) |\lambda_{j,m}| \leq 1,
\end{equation}
and thus, by monotonicity and $p_1<p_2$,
\[
  \varphi_1(2^{-j})^{p_2}  \sum_{m: Q_{j,m}\subset P} |\lambda_{j,m}|^{p_2} \leq
\varphi_1(2^{-j})^{p_1}  \sum_{m: Q_{j,m}\subset P} |\lambda_{j,m}|^{p_1}.
\]
Consequently, by \eqref{LS-2a},
\begin{align*}
  \varphi_1(2^{-j})^{p_2}  \sum_{m: Q_{j,m}\subset P} |\lambda_{j,m}|^{p_2} & \leq \ \varphi_1(2^{-j})^{p_1}   \varphi_1(2^{-j_P})^{-p_1} 2^{-(j_P-j)\nd},  
\end{align*}
thus
\[ \left(\sum_{m: Q_{j,m}\subset P} |\lambda_{j,m}|^{p_2}\right)^{\frac{1}{p_2}}  \leq \ \varphi_1(2^{-j})^{\varrho-1}   \varphi_1(2^{-j_P})^{-\varrho} 2^{-(j_P-j)\frac{\nd}{p_2}}.
  \]
  By assumption, $s_2<\sigma(s_1)$ and we can choose $s_3$ such that $s_2<s_3<\sigma<(s_1)\leq s_1=0$, and by \eqref{comp-1a} (with $s_1=0$),
  \[ \sup_{j\geq 0} 2^{{j} s_3 } \varphi_1(2^{-j})^{\varrho-1} < \infty. \]
  We arrive at
  \begin{align*}
    2^{j(s_2-\frac{\nd}{p_2})} \left(\sum_{Q_{j,m}\subset P} |\lambda_{j,m}|^{p_2}\right)^{\frac{1}{p_2}} &\leq 2^{j(s_2-s_3)} 2^{j(s_3-\frac{\nd}{p_2})} \varphi_1(2^{-j})^{\varrho-1} \varphi_1(2^{-j_P})^{-\varrho} 2^{(j-j_P)\frac{\nd}{p_2}}\\
                                                                                 &\leq \varphi_1(2^{-j_P})^{-\varrho} 2^{- j_P\frac{\nd}{p_2}} 2^{j s_3} \varphi_1(2^{-j})^{\varrho-1} 2^{j(s_2-s_3)},
  \end{align*}
  such that
  \begin{align*}
    \left(\sum_{j=j_P}^\infty 2^{j(s_2-\frac{\nd}{p_2})q_2} \left(\sum_{Q_{j,m}\subset P} |\lambda_{j,m}|^{p_2}\right)^{\frac{q_2}{p_2}}\right)^{\frac{1}{q_2}} & \leq
\varphi_1(2^{-j_P})^{-\varrho} 2^{- j_P\frac{\nd}{p_2}} \sup_{j\geq 0} 2^{j s_3} \varphi_1(2^{-j})^{\varrho-1} \left(\sum_{j=j_P}^\infty 2^{j(s_2-s_3)q_2}\right)^{\frac{1}{q_2}}\\                                            &\leq c\ \varphi_1(2^{-j_P})^{-\varrho} 2^{-j_P\frac{\nd}{p_2}} \\
    & \leq c' \ \varphi_2(2^{-j_P})^{-1} |P|^{\frac{1}{p_2}}
  \end{align*}
  in view of \eqref{sigma}. Finally, taking the supremum over all $P\subset Q$ {after} multiplying with $\varphi_2(2^{-j_P}) |P|^{-\frac{1}{p_2}} $,  leads to
  \[
    \|\lambda| \bpb(Q)\| \leq \ c = c \|\lambda | b^{0,\varphi_1}_{p_1,q_1}(Q)\| <\infty \]
  and ensures the continuity of $b^{0,\varphi_1}_{p_1,q_1}(Q) \hra \bpb(Q)$ for any $s_2<\sigma(0)\leq 0$. As for the compactness, we use the factorisation
  \[ b^{0,\varphi_1}_{p_1,q_1}(Q) \hra b^{s_3,\varphi_2}_{p_2,q_2}(Q) \hra \bpb(Q),\]
  where the first embedding is continuous by the above observation and the second one compact, where we benefit in both cases from our choice of $s_3$ with  $s_2<s_3<\sigma(0) \leq 0$.\\

  {\em Step 2}. We now care for (ii). \\
  {\em Substep 2.1}~
  We deal  {first with the case  $p_1\geq p_2$}. The continuity follows by Proposition~\ref{P-bp-cont} for all $s_2<\overline{\sigma}(s_1)$. As for the compactness, we may thus choose some $s_3$ with $s_2<s_3<\overline{\sigma}(s_1)$ and observe that
  \[
\bpa(Q) \hra b^{s_3,\varphi_2}_{p_2,q_2}(Q) \hra \bpb(Q),
  \]
where the first embedding is continuous and the second one compact, {by Proposition~\ref{P-bp-cont} and (i), respectively.}\\  

{\em Substep 2.2}.  Assume now $p_1<p_2$. We have $\varphi_1(t)\le \varphi_1(t)^\varrho \le {c} \varphi_2(t)$, $0<t\le 1$. %We may assume that $s_1=0$.  
	We factorise the embedding via 
\[
  b^{s_1,\varphi_1}_{p_1,q_1}(Q) \hra b^{\tau, \varphi_1^\varrho}_{p_2,q_2}(Q) \hra \bpb(Q),
  \]
  where $\tau \in (s_2,{s_1})$ has to be chosen appropriately. For the first embedding we can apply (i) and obtain compactness when $\tau<\sigma(s_1)$.
  If $s_2< {\overline{\sigma}(s_1)}$, then we can choose $\tau$ such that 
  \[ \sigma(s_1) + s_2 - {\overline{\sigma}(s_1)}< \tau <\sigma(s_1), \]
   which results in %the condition that
  \[ s_2-\tau <  {\overline{\sigma}(s_1)} - \sigma(s_1)<{ \overline{\sigma}(s_1) }-\tau .\]
 Thus for any $s< {\overline{\sigma}(s_1)}$ such that $s_2-\tau < s - \sigma(s_1) $  we have %For the latter embedding we apply Step~1 to conclude that it is compact when %$s_2<\overline{\sigma}(\tau)$, that is, when
   \[
    \sup_{j\geq 0} 2^{j(s_2-\tau)} \frac{\varphi_2(2^{-j})}{\varphi_1(2^{-j})^{\varrho}}\le \sup_{j\geq 0} 2^{j(s-\sigma(s_1))} \frac{\varphi_2(2^{-j})}{\varphi_1(2^{-j})^{\varrho}} <\infty.
   \]
   This implies the continuity of the second embedding by Proposition~\ref{P-bp-cont}.
\end{proof}

  \begin{remark}\label{rem-class-1}
    Let us briefly return to our Example~\ref{ex-class} and compare the above result with our (function space) result in Corollary~\ref{comp-class}(ii) for $\At=\Bt$. \\
    In case of (i) and $\varrho=1$, condition \eqref{comp-1} reads as $\frac{1}{p_1}-\tau_1-\frac{1}{p_2}+\tau_2\leq 0$ and we get compactness for $s_2<\sigma(s_1)=s_1$, which fits well together with the condition \eqref{tau-comp-u2'}. This is even known to be sharp (for the function space result). 
    If $\varrho<1$, then condition \eqref{comp-1} means $\tau_2\leq \tau_1 \varrho$ and we get compactness for $s_1-s_2>\nd (\frac{1}{p_1}-\frac{1}{p_2}-\tau_1+\tau_1\varrho)$ which again coincides with \eqref{tau-comp-u2} in that setting.\\ 
    In case of (ii), in Example~\ref{ex-class} we find  {$\overline{\sigma}(s_1)=s_1-\nd (\frac{1}{p_1}-\tau_1-\frac{1}{p_2}+\tau_2)$, }  hence (ii) ensures compactness when 
 $ s_2<s_1-\nd(\frac{1}{p_1}-\tau_1-\frac{1}{p_2}+\tau_2)$. 
{When $\varrho=1$ this} coincides with \eqref{tau-comp-u2'}, {as condition  \eqref{comp-2} means $\frac{1}{p_1}-\tau_1-\frac{1}{p_2}+\tau_2 \geq 0$.} In case of $\varrho<1$, {the above condition} coincides with \eqref{tau-comp-u2} since our assumption \eqref{comp-2} reads as $\tau_2\geq \frac{p_1}{p_2}\tau_1$. So also this case fits well together with \eqref{tau-comp-u2} again.\\
    So our new result in Proposition~\ref{Lemma-LS_1+2} suits perfectly in the example case, though we have only the sufficiency (in the sequence space setting) at the moment.    
  \end{remark}

\ignore{Now we care for the situation converse to \eqref{comp-1}. %, split in two cases $\varrho=1$ or $\varrho<1$.

\begin{proposition}\label{Lemma-LS_2}
Let $s_i\in\rr$, $0<p_i<\infty$, $0<q_i\leq \infty$, and $\varphi_i\in {\mathcal G}_{p_i}$, for $i=1,2$. {We assume without loss of generality that $\varphi_1(1)=\varphi_2(1)=1$.}
Assume that \eqref{comp-2} is satisfied.  
{Then  the embedding \eqref{embed1-bp} is compact, if $s_2<\overline{\sigma}(s_1)$, given by \eqref{comp-2a}.}
%\begin{enumerate}[{\bfseries\upshape(i)}]
%\item
%  If $p_2\leq p_1$, then  the embedding \eqref{embed1-bp} is compact, if $s_2<\overline{\sigma}(s_1)$.
%\item
%  If $p_1<p_2$,  then  the embedding \eqref{embed1-bp} is compact, if
%  \violet{$s_2 <  \overline{\sigma}(\sigma(s_1))$}. 
%\end{enumerate}
\end{proposition}

\begin{proof}
   \end{proof}
}

\

  %We follow Leszek's notes, p.4-6.

%%%%%%%%%%%%%%%%%%%%%%%%%%%%%%%%%%%%%%%%%%%%%%%%%%%%

\section{Embeddings of generalised Besov-Morrey spaces on domains}\label{sect-MB}

Recall that by  $\Omega$ we will always denote  a bounded {$C^{\infty}$}  domain in $\rd$. All spaces $\MA(\Omega)$ and $\Aphi(\Omega)$ are defined by restriction.

%\open{In the extension operator in \cite{IN19} it is assumed  to be a $C^{\infty}$ domain, so I think for the continuous embedding we need this assumption, or am I wrong? (SM) I think you are right, we should assume $C^\infty$ from the beginning. (DDH)}
%We first consider the  generalised Besov-Morrey spaces on $\Omega$.%  defined by restriction.

\begin{theorem}\label{th-cont-BM}
Let $s_i\in\rr$, $0<p_i<\infty$, $0<q_i\leq \infty$, and $\varphi_i\in {\mathcal G}_{p_i}$, for $i=1,2$. 
We assume without loss of generality that $\varphi_1(1)=\varphi_2(1)=1$. 
Let $ \varrho=\min(1,\frac{p_1}{p_2})$ and $\displaystyle\alpha_j= \sup_{0\le\nu\le j}\frac{\varphi_2(2^{-\nu})}{\varphi_1(2^{-\nu})^\varrho}$, $j\in \mathbb{N}_0$. 

There is a continuous embedding 
\begin{equation} \label{embed1O}
	\MBa (\Omega) \hookrightarrow \MBb(\Omega)
\end{equation}
if, and only if, 
\begin{align}
	\label{cond2O}
	\left\{   2^{j(s_2-s_1)} \alpha_j \frac{\varphi_1(2^{-j})^\varrho}{\varphi_1(2^{-j})}\right\}_{j\in \no} & \in \ell_{q^*}  \qquad \text{where}  \quad \frac{1}{q^*}=\left(\frac{1}{q_2}-\frac{1}{q_1}\right)_+  .
\end{align}
Moreover,  \eqref{embed1O} is compact if, and only if,  \eqref{cond2O} holds with $\ell_{\infty}$ replaced by $c_0$ if $q^*=\infty$, which means that
  \begin{align*}
		 2^{j(s_2-s_1)} \alpha_j \frac{\varphi_1(2^{-j})^\varrho}{\varphi_1(2^{-j})}\rightarrow 0\qquad \text{if}\qquad q_1\leq q_2. %q^*=\infty . 
	\end{align*}
\end{theorem}	

\begin{proof}
The proof follows the argument presented in Substep 1.1 in the proof of Theorem~3.1 in \cite{hs13},  now based  on Theorems~\ref{th-cont} and \ref{th:comp},  taking into account the wavelet characterisation of the spaces  $\MB (\rd)$,  cf. \cite[Theorem~3.1]{hms22},  and the existence of a common extension operator for both spaces $\MBa (\Omega)$ and $\MBb(\Omega)$, cf. \cite[Theorem~5.4]{IN19}.
\end{proof}	

{
The above theorem gives us the  precise statement also for  the classical case $\mathcal{N}^{s_1}_{u_1,p_1,q_1} (\Omega) \hookrightarrow \mathcal{N}^{s_2}_{u_2,p_2,q_2}  (\Omega)$. In particular, the continuity part of the following corollary improves  Theorem~3.1  in \cite{hs13}, while the compactness criterion coincides with Theorem~4.1  proved there and recalled in Corollary~\ref{comp-class}(i) for convenience. %in \cite{hs13}. 

\begin{corollary}\label{comp-class_new}
Let $s_i\in\rr$, $0<p_i\le u_i<\infty$,   $0<q_i\leq \infty$ 
	%, and $\varphi_i\in {\mathcal G}_{p_i}$, 
	for $i=1,2$ and    $ \varrho=\min(1,\frac{p_1}{p_2})$.
	% and $\displaystyle\alpha_j= \sup_{0\le\nu\le j}\frac{\varphi_2(2^{-\nu})}{\varphi_1(2^{-\nu})^\varrho}$, $j\in \mathbb{N}_0$. 
	
	There is a continuous embedding 
	\begin{equation} \label{embed1O'}
		\mathcal{N}^{s_1}_{u_1,p_1,q_1} (\Omega) \hookrightarrow \mathcal{N}^{s_2}_{u_2,p_2,q_2}  (\Omega)
	\end{equation}
	if, and only if,
          \begin{align}
\text{\upshape\bfseries (i)}\ \frac{u_1}{u_2}\ge \varrho & \quad \text{and}  \nonumber\\
\label{cond2Ocl}
		&\frac{s_1-s_2}{d}   > \frac{p_1}{u_1}\left(\frac{1}{p_1}-\frac{1}{p_2}\right)_+  \qquad \text{or}\qquad
		  \frac{s_1-s_2}{d}  = \frac{p_1}{u_1}\left(\frac{1}{p_1}-\frac{1}{p_2}\right)_+\quad\text{and}\quad q_1\le q_2,
                 \intertext{or}\nonumber
\text{\bfseries\upshape (ii)}\ \frac{u_1}{u_2}< \varrho & \quad  \text{and}\nonumber\\
		 &\frac{s_1-s_2}{d}  > \frac{1}{u_1}-\frac{1}{u_2} \qquad \text{or} \label{cond2Ocl2} \qquad
		 \frac{s_1-s_2}{d}  = \frac{1}{u_1}-\frac{1}{u_2} \quad\text{and}\quad q_1\le q_2. 
	\end{align}%
	Moreover,  \eqref{embed1O'} is compact if, and only if, strict inequalities for the difference of smoothnesses $\frac{s_1-s_2}{d}$ hold. %}
	%the conditions \eqref{cond2Ocl} or \eqref{cond2Ocl2} hold.  
\end{corollary}
\ignore{
	\begin{corollary}\label{comp-class_new}
		Let $s_i\in\rr$, $0<p_i\le u_i<\infty$,   $0<q_i\leq \infty$ 
		%, and $\varphi_i\in {\mathcal G}_{p_i}$, 
		for $i=1,2$ and    $ \varrho=\min(1,\frac{p_1}{p_2})$.
		% and $\displaystyle\alpha_j= \sup_{0\le\nu\le j}\frac{\varphi_2(2^{-\nu})}{\varphi_1(2^{-\nu})^\varrho}$, $j\in \mathbb{N}_0$. 
		
		There is a continuous embedding 
		\begin{equation} \label{embed1O'}
			\mathcal{N}^{s_1}_{u_1,p_1,q_1} (\Omega) \hookrightarrow \mathcal{N}^{s_2}_{u_2,p_2,q_2}  (\Omega)
		\end{equation}
		if, and only if, 
		\begin{align}
			\frac{u_1}{u_2}\ge \varrho\quad \text{and} \quad &\frac{s_1-s_2}{d}   > \frac{p_1}{u_1}\left(\frac{1}{p_1}-\frac{1}{p_2}\right)_+  \quad \text{or} \label{cond2Ocl}\\
			& \frac{s_1-s_2}{d}  = \frac{p_1}{u_1}\left(\frac{1}{p_1}-\frac{1}{p_2}\right)_+\quad\text{and}\quad q_1\le q_2; \nonumber 
			\\
			\intertext{or}
			\frac{u_1}{u_2}< \varrho\quad \text{and} \quad &\frac{s_1-s_2}{d}  > \frac{1}{u_1}-\frac{1}{u_2} \quad \text{or} \label{cond2Ocl2} \\
			&\frac{s_1-s_2}{d}  = \frac{1}{u_1}-\frac{1}{u_2} \quad\text{and}\quad q_1\le q_2. \nonumber
		\end{align}
		Moreover,  \eqref{embed1O'} is compact if, and only if,  the conditions \eqref{cond2Ocl} or \eqref{cond2Ocl2} hold.  
	\end{corollary}
}

\begin{proof}
  The corollary follows immediately from Theorem~\ref{th-cont-BM}.  If $\frac{u_1}{u_2}\ge \varrho$, then $\alpha_j=1$. If $\frac{u_1}{u_2}< \varrho$, then $\alpha_j=2^{jd(\frac{{\varrho}}{u_1}-\frac{1}{u_2})}$, recall also Remark~\ref{n-n-class}.  This covers the above two cases.
\end{proof}
}

\begin{remark}
As the previous theorem, also  Corollaries~\ref{cor1} and \ref{cor2} can be rewritten in terms of the embedding {\eqref{embed1O}}. 
 The compactness assertion in the last theorem holds true for   any bounded domain $\Omega$ in $\rd$, following a similar reasoning as the one explained, for instance, in the proof of Theorem~3.7 in \cite{HLMS2023}.
\end{remark}

{%\color{blue}
In the next corollary we consider special cases, where the conditions for the continuity/compactness of the embeddings take a simplified form.

\begin{corollary} \label{cor-liminfsup}
Let $s_i\in\rr$, $0<p_i<\infty$, $0<q_i\leq \infty$, and $\varphi_i\in {\mathcal G}_{p_i}$, for $i=1,2$. 
We assume without loss of generality that $\varphi_1(1)=\varphi_2(1)=1$. 
Let $ \varrho=\min(1,\frac{p_1}{p_2})$ and $\frac{1}{q^*}=\left(\frac{1}{q_2}-\frac{1}{q_1}\right)_+$. % and $\displaystyle\alpha_j= \sup_{0\le\nu\le j}\frac{\varphi_2(2^{-\nu})}{\varphi_1(2^{-\nu})^\varrho}$, $j\in \mathbb{N}_0$. 
\begin{enumerate}[{\bfseries\upshape(i)}]
\item  If
\begin{equation} \label{cond1-limsup}
\sup\left\{ \frac{\varphi_2(2^{-j})}{\varphi_1(2^{-j})^{\varrho}}: j\in \no \right\} <\infty,
%\limsup_{j\rightarrow \infty} %\frac{\varphi_2(2^{-j})}{\varphi_1(2^{-j})^{\varrho}} <\infty,
\end{equation}
then the embedding  {\eqref{embed1O}} holds if, and only if,
\begin{equation}\label{cond2-limsup0}
\left\{ 2^{j(s_2-s_1)} \varphi_1(2^{-j})^{\varrho-1}\right\}_{j\in \no} \in \ell_{q^*}.  
\end{equation}
\item Assume  $p_1\ge p_2$, i.e., $\varrho=1$,  and 
\begin{equation} \label{cond2-limsup}
	\sup\left\{ \frac{\varphi_2(2^{-j})}{\varphi_1(2^{-j})}: j\in \no \right\} =\infty.
%\limsup_{j\rightarrow \infty} \frac{\varphi_2(2^{-j})}{\varphi_1(2^{-j})} %=\infty.
\end{equation} 
%In case %of $p_1\geq p_2$ and $q_1\leq q_2$,  
The embedding  { \eqref{embed1O}} holds if, and only if,
\begin{equation} \label{cond3-limsup} 
%s_1\geq s_2 \qquad \text{and} \qquad   
\left\{ 2^{j(s_2-s_1)} \frac{\varphi_2(2^{-j})}{\varphi_1(2^{-j})}  \right\}_{j\in \no} \in \ell_{q^*}. 
\end{equation}
\item  Let  $p_1 < p_2$, i.e., $\varrho< 1$, and 
\begin{equation} \label{cond2a-limsup}
	\sup\left\{ \frac{\varphi_2(2^{-j})}{\varphi_1(2^{-j})^\varrho}: j\in \no \right\} =\infty.
	%\limsup_{j\rightarrow \infty} \frac{\varphi_2(2^{-j})}{\varphi_1(2^{-j})} %=\infty.
\end{equation} 
The embedding  { \eqref{embed1O}} holds if  the condition \eqref{cond3-limsup} is satisfied and $s_1-s_2>{d}\left(\frac{1}{p_1}-\frac{1}{p_2}\right)$ or  $s_1-s_2={d}\left(\frac{1}{p_1}-\frac{1}{p_2}\right)$ and $q_1\leq q_2$. %$q^*=\infty$. 

If we have the embedding   { \eqref{embed1O}}, then the condition  \eqref{cond3-limsup} holds.
%
%In case %of $p_1\geq p_2$ and 
%$q_1\leq q_2$,  the embedding   \eqref{embed1O} holds if, and only if,
%\begin{equation} \label{cond3a-limsup} 
%	s_1-s_2\geq \sigma_1+\sigma_2 \qquad \text{and} \qquad   
%	\left\{ 2^{j(s_2-s_1)} \frac{\varphi_2(2^{-j})}{\varphi_1(2^{-j})}  %\right\}_{j\in \no} \in \ell_{\infty}. 
%\end{equation}} 
\item In each case  (i) and (ii),  { \eqref{embed1O}} is compact if, and only if,   the conditions \eqref{cond2-limsup0} and \eqref{cond3-limsup} hold respectively   (with $c_0$ instead of  $\ell_{q^*}$ when $q^*=\infty$).  In case of (iii), {\eqref{embed1O}} is compact if the conditions \eqref{cond2-limsup0} and \eqref{cond3-limsup} hold (with $c_0$ instead of  $\ell_{\infty}$ when {$q_1\leq q_2$}). %$q^*=\infty$).
\end{enumerate}
\end{corollary}	

\begin{proof}
Part (i) is an immediate consequence of Theorem~\ref{th-cont-BM} as \eqref{cond1-limsup} implies that 
$$
\alpha_j= \sup_{0\le\nu\le j}\frac{\varphi_2(2^{-\nu})}{\varphi_1(2^{-\nu})^\varrho}\sim 1, \quad j\in\no.
$$ 

Now we deal with (ii). The assumption \eqref{cond2-limsup}  guarantees the existence of a strictly increasing sequence of  natural numbers $(j_k)_{k\in \no}$ such that the sequence
$\left\{ \frac{\varphi_2(2^{-j_k})}{\varphi_1(2^{-j_k})} \right\}_{k\in\no}$ is strictly increasing, tends to $\infty$ and 
$$
\alpha_{j}=  \sup_{0\le\nu\le j} \frac{\varphi_2(2^{-\nu})}{\varphi_1(2^{-\nu})}=  \frac{\varphi_2(2^{-j_k})}{\varphi_1(2^{-j_k})} \quad\text{if}\quad j_k\le j < j_{k+1},\quad k\in\no.
$$
First we assume that the condition \eqref{cond3-limsup} holds. This condition and \eqref{cond2-limsup} imply $s_1>s_2$. Let $q^*<\infty$. Then
\begin{align}
	\sum_{j=0}^\infty 2^{j(s_2-s_1)q^*}\alpha_j^{q^*} &= \sum_{k=0}^\infty \sum_{j=j_k}^{j_{k+1}-1} 2^{(j-j_k)(s_2-s_1)q^*}2^{j_k(s_2-s_1)q^*}\frac{\varphi_2(2^{-j_k})^{q^*}}{\varphi_1(2^{-j_k})^{q^*}}  \nonumber\\
	& =\sum_{k=0}^\infty 2^{j_k(s_2-s_1)q^*} \frac{\varphi_2(2^{-j_k})^{q^*}}{\varphi_1(2^{-j_k})^{q^*}} \sum_{j=j_k}^{j_{k+1}-1}  2^{(j-j_k)(s_2-s_1)q^*} < \infty \nonumber
\end{align}
since $s_1>s_2$, cf. \eqref{cond3-limsup}. Now the statement follows from  Theorem~\ref{th-cont-BM}. The proof in the case $q^*=\infty$ is similar. 

Now assume that  the embedding   {\eqref{embed1O}} holds with $p_1\ge p_2$ and \eqref{cond2-limsup}.  We prove the necessity of \eqref{cond3-limsup}. 
We have always
$$\alpha_j= \sup_{0\le\nu\le j} \frac{\varphi_2(2^{-\nu})}{\varphi_1(2^{-\nu})}\geq  \frac{\varphi_2(2^{-j})}{\varphi_1(2^{-j})}, \quad j\in \no.$$
Therefore the condition  \eqref{cond2O} implies {\eqref{cond3-limsup}}. 

To prove (iii) we deal in a similar way as above. Now we take the  sequence $\{j_k\}_{k\in\no}$ such that the sequence
$\left\{ \frac{\varphi_2(2^{-j_k})}{\varphi_1(2^{-j_k})^\varrho}  \right\}_{k\in\no}$ is strictly increasing, tends to $\infty$ and 
\begin{equation}\label{embp1p2}
\alpha_{j}= % \sup_{0\le\nu\le j} \frac{\varphi_2(2^{-\nu})}{\varphi_1(2^{-\nu})}= 
 \frac{\varphi_2(2^{-j_k})}{\varphi_1(2^{-j_k})^\varrho} \quad\text{if}\quad j_k\le j < j_{k+1},\quad k\in\no.
\end{equation}
Moreover, 
\begin{equation}\label{embp1p2-2}
	\frac{\varphi_1(2^{-j_k})}{\varphi_1(2^{-j})} \le 2^{(j-j_k)\frac{d}{p_1}}
 \qquad  \text{if}\quad j_k\le j < j_{k+1}, 
\end{equation} 	 
since $\varphi_1\in \mathcal{G}_{p_1}$.  Let $q^*<\infty$. Then
\begin{align}
	\sum_{j=0}^\infty 2^{j(s_2-s_1)q^*}\alpha_j^{q^*} \varphi_1(2^{-j})^{(\varrho-1)q^*} & = \sum_{k=0}^\infty \sum_{j=j_k}^{j_{k+1}-1} 2^{(j-j_k)(s_2-s_1)q^*}2^{j_k(s_2-s_1)q^*}\frac{\varphi_2(2^{-j_k})^{q^*}}{\varphi_1(2^{-j_k})^{\varrho q^*}}\varphi_1(2^{-j})^{(\varrho-1)q^*} \nonumber   \\
	& = \sum_{k=0}^\infty 2^{j_k(s_2-s_1)q^*}  \frac{\varphi_2(2^{-j_k})^{q^*}}{\varphi_1(2^{-j_k})^{q^*}} \sum_{j=j_k}^{j_{k+1}-1}  2^{(j-j_k)(s_2-s_1)q^*} %\frac{\varphi_1(2^{-j_k)})}{\varphi_1(2^{-j)})}
   \left(\frac{\varphi_1(2^{-j)})}{\varphi_1(2^{-j_k)})}\right)^{(\varrho-1)q^*} \nonumber \\
  &\le  \sum_{k=0}^\infty 2^{j_k(s_2-s_1)q^*}  \frac{\varphi_2(2^{-j_k})^{q^*}}{\varphi_1(2^{-j_k})^{q^*}}   \sum_{j=j_k}^{j_{k+1}-1}  2^{(j-j_k)(s_2-s_1-\frac{{d}}{p_2}+\frac{{d}}{p_1})q^*} %\frac{\varphi_1(2^{-j_k)})}{\varphi_1(2^{-j)})}
	 %\left(\frac{\varphi_1(2^{-j)})}{\varphi_1(2^{-j_k)})}\right)^{(\varrho-1)q^*}  
	 	 < \infty,  \nonumber
\end{align}
where the last but one inequality follows from \eqref{embp1p2-2} and $\varrho<1$ and the last one from \eqref{cond3-limsup} and $s_1-s_2>{d}\left(\frac{1}{p_1}-\frac{1}{p_2}\right)$. Now the statement follows from  Theorem~\ref{th-cont-BM}. The proof in the case $q^*=\infty$ is similar. 
\end{proof}

In the following corollaries we consider several special cases of Theorem~\ref{th-cont-BM}, where the source or target space is of Besov,  Lebesgue or Besov-Morrey type.

%%%%%%%%%%%%%%%%%%%%%%%%%%	
\begin{corollary}  \label{cor-emb-in-N}
	Let $s_i\in\rr$, $0<p_i<\infty$,  $0<q_i\leq \infty$, for $i=1,2$, and define
	$$ \varrho=\min\left(1,\frac{p_1}{p_2}\right) \qquad  \text{and} \qquad  \frac{1}{q^*}=\left(\frac{1}{q_2}-\frac{1}{q_1}\right)_+.$$ 
	\begin{enumerate}[{\bfseries\upshape(i)}]
		\item  Let  $\varphi\in {\mathcal G}_{p_1}$ and assume that for some $u_2\ge p_2$ it holds
		\begin{equation}\label{cond-emb-in-N}
			\liminf_{j\rightarrow \infty} \varphi(2^{-j})  2^{j\frac{d}{u_2\varrho}} >0.
		\end{equation}
		Then there is a continuous embedding 
		\begin{equation}\label{emb-in-N}
			\mathcal{N}^{s_1}_{\varphi,p_1,q_1}(\Omega)  \hookrightarrow \mathcal{N}^{s_2}_{u_2,p_2,q_2}(\Omega)
		\end{equation}
		if, and only if,
		\begin{align}\label{emb-in-Na}
			\left\{   2^{j(s_2-s_1)} \varphi(2^{-j})^{\varrho-1}\right\}_{j\in \no} & \in \ell_{q^*} . 
		\end{align}
	The embedding is compact if, and only if,  \eqref{emb-in-Na} holds (with $c_0$ instead of $\ell_{\infty}$ if {$q_1\leq q_2$}). 
	%	\item  Let  $\varphi\in {\mathcal G}_{p_1}$ and let  \red{$u_2\ge p_2$}  be such that  $ \varphi(t)t^{-\frac{d}{u_2\varrho}}$ is increasing on $(0,1)$. Then  \eqref{emb-in-N} holds
	%	if, and only if,
	%	\begin{align*}
	%		\left\{   2^{j(s_2-s_1-\frac{d}{u_2})} \varphi(2^{-j})^{-1}\right\}_{j\in \no} & \in \ell_{q^*} . 
	%	\end{align*}
	%	
	\item  Let  $\varphi\in {\mathcal G}_{p_1}$, $p_1\ge p_2$, and assume that for some $u_2\ge p_2$ it holds
	\begin{equation}\label{cond-emb-in-Nb}
		\liminf_{j\rightarrow \infty} \varphi(2^{-j})  2^{j\frac{d}{u_2}} = 0.
	\end{equation}
	Then there is a continuous embedding \eqref{emb-in-N} 	if, and only if,
	\begin{align}\label{emb-in-Nc}
		\left\{   2^{j(s_2-s_1 -\frac{d}{u_2})} \varphi(2^{-j})^{-1}\right\}_{j\in \no} & \in \ell_{q^*} . 
	\end{align}
	The embedding is compact if, and only if,  \eqref{emb-in-Nc} holds (with $c_0$ instead of $\ell_{\infty}$ if {$q_1\leq q_2$}). 
   	\item  Let  $\varphi\in {\mathcal G}_{p_1}$, $p_1 < p_2$, and assume that for some $u_2\ge p_2$ it holds
   	\begin{equation}\label{cond-emb-in-Nb'}
   		\liminf_{j\rightarrow \infty} \varphi(2^{-j})  2^{j\frac{d}{u_2\varrho}} = 0.
   	\end{equation}
   	Then there is a continuous embedding \eqref{emb-in-N} 	if \eqref{emb-in-Nc} holds and  $s_1-s_2>{d}\left(\frac{1}{p_1}-\frac{1}{p_2}\right)$ or  $s_1-s_2={d}\left(\frac{1}{p_1}-\frac{1}{p_2}\right)$ and $q_1\leq q_2$. %$q^*=\infty$. 
   	The embedding is compact if  \eqref{emb-in-Nc} holds (with $c_0$ instead of $\ell_{\infty}$ if {$q_1\leq q_2$}).       
   
		\item  Let  $\varphi\in {\mathcal G}_{p_2}$ and  assume that for some $u_1\geq p_1>0$ it holds
		\begin{equation}\label{cond-emb-from-N}
			\limsup_{j\rightarrow \infty} \varphi(2^{-j})  2^{j\frac{d \varrho}{u_1}}<\infty.
		\end{equation}
		Then there is a continuous embedding 
		\begin{equation}\label{emb-from-N}
			\mathcal{N}^{s_1}_{u_1,p_1,q_1}  (\Omega)  \hookrightarrow \mathcal{N}^{s_2}_{\varphi,p_2,q_2}(\Omega)
		\end{equation}
		if, and only if,
		%$$
		%\frac{s_1-s_2}{d} > \frac{p_1}{u_1} \left(\frac{1}{p_1}-\frac{1}{p_2}\right)_+ \qquad \text{or} \qquad  \frac{s_1-s_2}{d} > \frac{p_1}{u_1} \left(\frac{1}{p_1}-\frac{1}{p_2}\right)_+ \qquad \text{and} \qquad   q_1\leq q_2.
		%$$
		\begin{align*}
			\left\{   2^{j(s_2-s_1+\frac{d}{u_1}(1-\varrho))}\right\}_{j\in \no}  \in \ell_{q^*} . 
		\end{align*}
	%	\item  Let  $\varphi\in {\mathcal G}_{p_2}$ and let   $u_1>0$ be such %that  $ \varphi(t)t^{-\frac{d \varrho}{u_1}}$ is decreasing on $(0,1)$. %Then  \eqref{emb-from-N} holds
	%	if, and only if,
	%	\begin{align*}
	%		\left\{   2^{j(s_2-s_1+\frac{d}{u_1})} \varphi(2^{-j})\right\}_{j\in \no} & \in \ell_{q^*} . 
	%	\end{align*}
	\end{enumerate}
\end{corollary}	

\begin{proof} 
	Both (i) and (iv) are immediate consequences of  Corollary~\ref{cor-liminfsup}.
	%Both (i) and (iii)  follow immediately from  Theorem \ref{th-cont-BM}, since condition  \eqref{cond-emb-in-N} and \eqref{cond-emb-from-N}  lead, respectively,  to 
	%$$
	% \sup_{0\le\nu\le j}\frac{2^{-\nu \frac{d}{u_2}}}{\varphi(2^{-\nu})^{\varrho}} \sim 1  \qquad \text{and} \qquad   \sup_{0\le\nu\le j}\frac{\varphi(2^{-\nu})}{2^{-\nu \frac{d \varrho}{u_1}}} \sim 1  , \quad j\in \mathbb{N}_0.
	%$$
	As for (ii) and (iii), they follow directly from Theorem~\ref{th-cont-BM}, since by the assumptions made in (ii) and (iii) we have, respectively, 
	$$
	\sup_{0\le\nu\le j}\frac{2^{-\nu \frac{d}{u_2}}}{\varphi(2^{-\nu})^{\varrho}} = \frac{2^{-j \frac{d}{u_2}}}{\varphi(2^{-j})^{\varrho}} 
	\qquad \text{and} \qquad
	\sup_{0\le\nu\le j}\frac{\varphi(2^{-\nu})}{2^{-\nu \frac{d \varrho}{u_1}}} = \varphi(2^{-j}) 2^{j \frac{d \varrho}{u_1}} , \quad j\in \mathbb{N}_0.
	$$
\end{proof}

\begin{remark}  \label{rema} Assume that $0<p<\infty$ and $\varphi\in {\mathcal G}_{p}$. If there exists some $u\geq p$ so that 
	\begin{equation*}
		0< \liminf_{j\rightarrow \infty} \varphi(2^{-j})  2^{j\frac{d}{u}}  \leq  \limsup_{j\rightarrow \infty} \varphi(2^{-j})  2^{j\frac{d}{u}}<\infty,
	\end{equation*}
	then it follows from the last corollary that $\mathcal{N}^{s}_{\varphi,p,q}(\Omega)= \mathcal{N}^{s}_{u,p,q}(\Omega)$,
	in the sense of equivalent quasi-norms.  
	As an example,  $\mathcal{N}^{s}_{\varphi_{u,v},p,q}(\Omega)= \mathcal{N}^{s}_{u,p,q}(\Omega)$, for $\varphi_{u,v}$ defined by
	$$
	\varphi_{u,v}(t)=\begin{cases}
		t^{\frac{d}{u}} & \text{if} \quad 0<t<\varepsilon, \\
		c t^{\frac{d}{v}} & \text{if} \quad t \geq \varepsilon,
	\end{cases}
	$$
	where $u,v\geq p$, $\varepsilon>0$ and $c= \varepsilon^{\frac{d}{u}-\frac{d}{v}}$.
\end{remark}

\begin{remark}
Assume that $s\in\rr$, $0<p<\infty$, $0<q\leq \infty$ and $\varphi \in {\mathcal G}_{p}$. If  $\ds\liminf_{j\rightarrow \infty} \varphi(2^{-j})  2^{j\frac{d}{u}} >0$ for some $u\geq p$, then  
	${\mathcal N}^{s}_{\varphi,p,q} (\Omega)  \hookrightarrow {\mathcal N} ^{s}_{u,p,q}(\Omega)$; but the inclusion is strict if  $\ds\limsup_{j\rightarrow \infty} \varphi(2^{-j})  2^{j\frac{d}{u}} =\infty$.
\end{remark}
%%%%%%%%%%%%%%%%%%%%

\begin{corollary} \label{cor-N-in-B}
Let $s_i\in\rr$, {$0<p_1<\infty$, $0<p_2\le\infty$,}  $0<q_i\leq \infty$, for $i=1,2$,  and $\varphi\in {\mathcal G}_{p_1}$. 
There is a continuous embedding 
		\begin{equation}\label{emb-in-B}
				{\mathcal N}^{s_1}_{\varphi,p_1,q_1} (\Omega)  \hookrightarrow B^{s_2}_{p_2,q_2}(\Omega)
		\end{equation}
		if, and only if,
		\begin{align}
			\left\{   2^{j(s_2-s_1)} \varphi(2^{-j})^{\varrho-1}\right\}_{j\in \no} & \in \ell_{q^*} ,  \label{cond-emb-in-B}
	\end{align}
where  $ \varrho=\min(1,\frac{p_1}{p_2})$  {(in particular $\varrho= 0$ if $p_2=\infty$)} and $\frac{1}{q^*}=\left(\frac{1}{q_2}-\frac{1}{q_1}\right)_+$.  Moreover, \eqref{emb-in-B} is compact if, and only if,  \eqref{cond-emb-in-B} holds  
with $ \ell_{\infty}$ replaced by $c_0$ if {$q_1\leq q_2$}.
\end{corollary}	

\begin{proof}
Let us assume first that $p_1\geq p_2$. Then $\varrho=1$ and $\varphi\in {\mathcal G}_{p_1}\subset {\mathcal G}_{p_2} $.  As a consequence, we have that
$$
\alpha_j= \sup_{0\le\nu\le j}\frac{2^{-\nu \frac{d}{p_2}}}{\varphi(2^{-\nu})}=1, \quad j\in \mathbb{N}_0,
$$
and hence the desired  equivalence follows by Theorem~\ref{th-cont-BM}.

Now consider the case $p_1<p_2$. Then $\varrho=\frac{p_1}{p_2}$ and  since $\varphi\in {\mathcal G}_{p_1}$, we have
$$
\alpha_j= \sup_{0\le\nu\le j}\frac{2^{-\nu \frac{d}{p_2}}}{\varphi(2^{-\nu})^{\frac{p_1}{p_2}}}=  \sup_{0\le\nu\le j} \left[\frac{2^{-\nu \frac{d}{p_1}}}{\varphi(2^{-\nu})}\right]^{\frac{p_1}{p_2}}=1  , \quad j\in \mathbb{N}_0.
$$
Once more  Theorem~\ref{th-cont-BM} entails the desired equivalence. {If $p_2=\infty$, then the statement follows from Corollary~\ref{cor2a} and Corollary~\ref{nbbncomp}.}
\end{proof}

\begin{remark}	
\begin{enumerate}[{\bfseries\upshape(i)}]
\item If $p_1\geq p_2$, it turns out that  ${\mathcal N}^{s_1}_{\varphi,p_1,q_1} (\Omega)  \hookrightarrow B^{s_2}_{p_2,q_2}(\Omega)$ if, and only if, $s_1>s_2$ or   $s_1=s_2$ and $q_1\leq q_2$. Moreover, the embedding is compact if, and only if,  $s_1>s_2$.  
In particular, ${\mathcal N}^{s}_{\varphi,p,q} (\Omega)  \hookrightarrow B^{s}_{p,q}(\Omega)$ for any $s\in\rr$, $0<p<\infty$, $0<q\leq \infty$ and $\varphi \in {\mathcal G}_{p}$.
\item  In the special case of $\varphi(t)=t^{-d/u_1}$ with $0<p_1\leq u_1<\infty$, hence ${\mathcal N}^{s_1}_{\varphi,p_1,q_1} (\Omega)= {\mathcal N}^{s_1}_{u_1,p_1,q_1} (\Omega)$,    we recover Corollary~3.5 of \cite{hs13}, with a slight improvement in one of the limiting cases. % the case of $p_1 <p_2$, $s_1-s_2=d \frac{p_1}{u_1}\left(\frac{1}{p_1}-\frac{1}{p_2}\right)$
\end{enumerate}
\end{remark}

\begin{corollary}\label{cor-B-in-N}
		Let  $0<p_1\le \infty$, $0<p_2< \infty$, $s_i\in\rr$,  $0<q_i\leq \infty$, for $i=1,2$,  and $\varphi\in {\mathcal G}_{p_2}$. 
		There is a continuous embedding 
		\begin{equation}\label{emb-in-B'}
		B^{s_1}_{p_1,q_1}(\Omega)\hookrightarrow 	{\mathcal N}^{s_2}_{\varphi,p_2,q_2} (\Omega)  
		\end{equation}
		if $s_1-\frac{d}{p_1}>s_2$ or $s_1-\frac{d}{p_1}=s_2$ and {$q_1\leq q_2$.}%$q^*=\infty$.  
		
The embedding 	\eqref{emb-in-B'} is compact if  $s_1-\frac{d}{p_1}>s_2$. 
 \end{corollary}
\begin{proof}
	The statement follows from the following  factorisation
\[B^{s_1}_{p_1,q_1}(\Omega)\hookrightarrow B^{s_2}_{\infty,q_2}(\Omega)\hookrightarrow	{\mathcal N}^{s_2}_{\varphi,p_2,q_2} (\Omega),\]
Corollary~\ref{cor2b}, Corollary~\ref{nbbncomp} and the properties of embeddings between classical Besov spaces. 
\end{proof}

\begin{corollary} \label{cor-Lr}
Let $s\in\rr$, $0<p<\infty$, $0<q\leq \infty$,  $\varphi\in {\mathcal G}_{p}$, and $1\leq r \leq \infty$. Let  $ \varrho=\min(1,\frac{p}{r})$ and $t_1,t_2$ be defined by
$$  \frac{1}{t_1}=\left(\frac{1}{\min(r,2)}-\frac{1}{q}\right)_+ \qquad \text{and} \qquad   \frac{1}{t_2}=\left(\frac{1}{\max(r,2)}-\frac{1}{q}\right)_+ ,$$
where in the latter case we assume $r>1$ or put $t_2=\infty$ if $r=1$.
Then there is a continuous embedding 
\begin{equation}\label{emb-in-L}
			\mathcal{N}^{s}_{\varphi,p,q}(\Omega)    \hookrightarrow L_{r}(\Omega)
\end{equation}
if
\begin{align}  \label{cond-emb-in-L}
\left\{  2^{-sj} \varphi(2^{-j})^{\varrho-1} 
\right\}_{j\in \no}  \in \ell_{t_1} .
\end{align}
Conversely,  \eqref{emb-in-L}  implies
\begin{align*}
	\left\{   2^{-sj} \varphi(2^{-j})^{\varrho-1}\right\}_{j\in \no} & \in \ell_{t_2}.
\end{align*}
 \end{corollary}	
		
\begin{proof} For $r\geq 1$, if condition \eqref{cond-emb-in-L} holds, using 
 Corollary~\ref{cor-N-in-B}, we have the following  chain of embeddings
$$
\mathcal{N}^{s}_{\varphi,p,q}(\Omega)    \hookrightarrow B^0_{r,\min(r,2)}(\Omega) \hookrightarrow L_{r}(\Omega).
$$
On the other hand,  the continuity of the embedding \eqref{emb-in-L}  leads to
$$
\mathcal{N}^{s}_{\varphi,p,q}(\Omega)    \hookrightarrow L_{r}(\Omega) \hookrightarrow  B^0_{r,\max(r,2)}(\Omega) \quad \text{if} \quad  r>1,
$$
and, in case of $r=1$, to 
$$
\mathcal{N}^{s}_{\varphi,p,q}(\Omega)    \hookrightarrow L_{1}(\Omega) \hookrightarrow  B^0_{1,\infty}(\Omega).
$$
Then  Corollary~\ref{cor-N-in-B} implies $\left\{   2^{-sj} \varphi(2^{-j})^{\varrho-1}\right\}_{j\in \no} \in \ell_{t_2}$ .
\end{proof}	
	
	\begin{corollary}
	Let $s\in\rr$, $0<p<\infty$, $0<q\leq \infty$, $r=\max(1,p)$ and   $\varphi\in {\mathcal G}_{p}$.  Then there is a continuous embedding 
	\begin{equation}\label{emb-in-L0}
		\mathcal{N}^{0}_{\varphi,p,q}(\Omega)    \hookrightarrow L_{r}(\Omega)
	\end{equation}
	if $q\le\min(r,2)$. Conversely,  \eqref{emb-in-L0}  implies  $q\le\max(r,2)$ if $p>1$, and no further assumption on $q$ %$q\le\infty$ 
        if $p\le 1$.
	\end{corollary}
\begin{remark}
 In case $\varphi\in {\mathcal G}_{p}$ satisfies the additional condition that there exists some $\varepsilon>0$ so that the function $\varphi(t) t^{-\varepsilon}$ is increasing, then the generalised Triebel-Lizorkin-Morrey spaces  $\mathcal{E}^{s}_{\varphi,p,q}(\rd)$ are defined.  Moreover,  $\mathcal{E}^{0}_{\varphi,p,2}(\rd)=\mathcal{M}_{\varphi,p}(\rd)$ for $1<p<\infty$. By considering the restriction to $\Omega$, we can improve a bit the sufficient condition in  Corollary~\ref{cor-Lr} when $1<p<\infty$ and $p\geq r$. Indeed, in such a case $\varrho=1$ and necessarily $s\geq 0$. Then, if  $0<q\leq \min \left(\max(p,r),2\right)=\min(p,2)$,  we have the chain of embeddings
$$
\mathcal{N}^{s}_{\varphi,p,q}(\Omega)  \hookrightarrow \mathcal{N}^{0}_{\varphi,p,\min(p,2)}(\Omega)   \hookrightarrow \mathcal{E}^{0}_{\varphi,p,2}(\Omega) =\mathcal{M}_{\varphi,p}(\Omega)    \hookrightarrow  L_p(\Omega)  \hookrightarrow L_r(\Omega).
$$ 
This means that  we could  replace $\min(r,2)$ by   $\min \left(\max(p,r),2\right)$ in the definition of $t_1$ regarding the sufficient condition for the embedding \eqref{emb-in-L}.
\end{remark}

\section{Embeddings of further generalised Morrey smoothness spaces on domains}\label{sect-Bphi}

{We now consider spaces of type $\Aphi(\Omega)$ and $\MF(\Omega)$, and start by presenting  new embedding results for spaces of type $\Bphi(\Omega)$,  building on our previous findings for sequence spaces. Relying on Proposition~\ref{P-bp-cont}, we derive the following.} 

  %\begin{corollary}
  	\begin{theorem}\label{P-Bp-cont}
Let $s_i\in\rr$, $0<p_2\le p_1<\infty$, $0<q_i\leq \infty$, and $\varphi_i\in {\mathcal G}_{p_i}$, for $i=1,2$. 
We assume without loss of generality that $\varphi_1(1)=\varphi_2(1)=1$. 
%Let $ \varrho=\min(1,\frac{p_1}{p_2})$. 
% and $\displaystyle\alpha_j= \sup_{0\le\nu\le %j}\frac{\varphi_2(2^{-\nu})}{\varphi_1(2^{-\nu})^\varrho}$, $j\in \mathbb{N}_0$. 

There is a continuous embedding 
\begin{equation} \label{embed1-bpa}  % I changed label {embed1-bp} to {embed1-bpa} as it was repeated to a previous one
\Bphia(\Omega) \hookrightarrow \Bphib(\Omega)  
\end{equation}
if, and only if, \eqref{cond2-bp} and \eqref{cond3-bp} are satisfied.
\end{theorem}
%\end{corollary}

{In the same manner, Proposition~\ref{P-bp-cont2} and  {Corollary}~\ref{P-bp-cont3} yield corresponding corollaries on continuous embeddings.  
Propositions~\ref{Lemma-LS_1+2} and \ref{pinfinity} allow us to derive the following two corollaries regarding compact embeddings.}

Let the numbers $\sigma(s_1)$, $\sigma_{\infty}(s_1)$  and $\overline{\sigma}(s_1)$ be given by \eqref{comp-1a}, \eqref{sigma}  and \eqref{comp-2a}, respectively. %Recall our conditions \eqref{comp-1} and \eqref{comp-2}.

%\begin{corollary}
		\begin{theorem}\label{cor-5.1}
Let $s_i\in\rr$, $0<p_i<\infty$, $0<q_i\leq \infty$, and $\varphi_i\in {\mathcal G}_{p_i}$, for $i=1,2$, and $ \varrho=\min(1,\frac{p_1}{p_2})$. We assume without loss of generality that $\varphi_1(1)=\varphi_2(1)=1$.
\begin{enumerate}[{\bfseries \upshape (i)}]
\item
  Assume that \eqref{comp-1} is satisfied.  
  Then the embedding
\begin{equation} \label{B-B-comp}
	\Bphia(\Omega) \hookrightarrow \Bphib(\Omega) 
\end{equation}
is compact, if $s_2<\sigma(s_1)$.
\item
{Assume that \eqref{comp-2} holds. Then  the embedding \eqref{B-B-comp} is compact, if  $s_2<\overline{\sigma}(s_1)$.}
%
%  \begin{align*}
%  s_2 < \begin{cases} s_1 & \text{if}\quad p_1\ge p_2, \\
%     \overline{\sigma}(\sigma(s_1)) & \text{if}\quad p_1<p_2 . \end{cases} %\qquad\text{and \eqref{comp-2} holds.
%	\end{align*}
%}
 \end{enumerate}
	\end{theorem}
%\end{corollary}

\remark{We refer to our Remark~\ref{rem-class-1} for the comparison with the classical case recalled in Corollary~\ref{comp-class}.
Recall that this means, in particular, that the assumptions that guarantee the compactness of embedding \eqref{B-B-comp} cannot be improved in general. 

Please note that the breaking point in the assumptions of {Theorem}~\ref{B-B-comp}, i.e., the case $\varphi_2(t)\sim \varphi_1(t)^\varrho$, $0<t\le 1$, corresponds to the concept of so-called {\em clans} of function spaces introduced and studied (for the Morrey smoothness spaces, including $\MAu$ and $\At$) by Triebel and Haroske  in \cite{ht2021}.    
}

{
% \begin{corollary}
 		\begin{theorem}\label{Pinfinity}
Let  $0<p<\infty$ and $\varphi\in {\mathcal G}_{p}$. We assume without loss of generality that $\varphi(1)=1$.  Let $s_1\in\rr$, $0<q \leq \infty$,  and  assume  that \eqref{liminf} is satisfied.

	There is a continuous embedding   
	\begin{equation}\label{pinfty3}
		B^{s_1,\varphi}_{p,q}(\Omega)\hookrightarrow B^{s_2}_{\infty,\infty}(\Omega)
	\end{equation}
	 if, and only if,   $s_2\le \sigma_\infty(s_1)$.  
	
	Moreover, if $s_2 < \sigma_\infty(s_1)$, then the embedding \eqref{pinfty3} is  compact.
if, and only if, {\eqref{comp-nec} is satisfied.}
%	\eqref{cond2-bp} and \eqref{cond3-bp} are satisfied.
	\end{theorem}
%\end{corollary}
}

\remark{\label{rembmo}
It follows directly from the definitions that if $\lim_{t\rightarrow 0} \varphi(t)=c>0$, $\varphi\in \mathcal{G}_p$, then 
	\[ B^{s,\varphi}_{p,q}(\Omega)\, = \,B^{s,1/p}_{p,q}(\Omega)\qquad \text{(in the sense of equivalence of norms)}\]
for any possible parameters $s,p,q$. On the other hand, it is well known  that $B^{0,1/2}_{2,2}(\Omega)= \rm{bmo}(\Omega)$, cf.~\cite{s011}. In consequence 
\[ B^{0,\varphi}_{2,2}(\Omega)\, = \,\rm{bmo}(\Omega)\qquad \text{(in the sense of equivalence of norms)}\]
if $\lim_{t\rightarrow 0} \varphi(t)=c>0$. 
}
      
\begin{corollary}
	Let  $0<p<\infty$, $0<q\le \infty$, $\varphi\in \mathcal{G}_p$, and $\lim_{t\rightarrow 0} \varphi(t)=0$. Then 
	\begin{enumerate}[{\bfseries\upshape (i)}]
\item the embedding \[ B^{s,\varphi}_{p,q}(\Omega)\hookrightarrow \rm{bmo}(\Omega)\] is compact if $s>\frac{d}{p}$ and $p<2$, 
\item the embedding \[ \mathrm{bmo}(\Omega)\hookrightarrow B^{s,\varphi}_{p,q}(\Omega)\] is compact if $s<0$.
	\end{enumerate} 
\end{corollary}
\begin{proof} Let $p<2$. 
	We take $\varepsilon>0$ such that    $s-\varepsilon>\frac{d}{p}$. Then$(s-\varepsilon- \frac{d}{2})(1-\varrho)^{{-1} }> \frac{d}{p}$, where $\varrho=\frac{p}{2}$. But $\varphi\in \mathcal{G}_p$, so 
	\[ {\sup_{j\in\no} }  \ 2^{j(\frac{d}{2}-s+\varepsilon)}\varphi(2^{-j})^{\varrho-1}\le C<\infty. \] 
	 This means that $\sigma(s)\ge \frac{d}{2}+\varepsilon$. Moreover $\varphi^\varrho\in \mathcal{G}_2$, so 
	 \[{\sup_{j\in\no}} \ 2^{j(\varepsilon-\sigma(s))}\varphi(2^{-j})^{-\varrho}\le C<\infty. \] 
We take $\varphi_2\in \mathcal{G}_2$ such that  $\lim_{t\rightarrow 0} \varphi_2(t)=c>0$. Then {$\overline{\sigma}(s)\geq \varepsilon>0$} and in consequence Corollary~\ref{cor-5.1}(ii) and  Remark~\ref{rembmo} imply the compactness of the embedding 
\[  B^{s,\varphi}_{p,q}(\Omega)\hookrightarrow B^{0,\varphi_2}_{2,2}(\Omega) = \rm{bmo}(\Omega). \]

If $p\ge 2$,   then $\varrho=1$ and 
\[{\sup_{j\in\no}}\   2^{j(\varepsilon-s)}\varphi(2^{-j})^{-1}\le C<\infty\]
if  $s-\varepsilon>\frac{d}{p}$. This means that { $\sigma(s)=s$, $\overline{\sigma}(s)\geq \varepsilon>0$ and the assertion follows once more by}  {Theorem}~\ref{cor-5.1}(ii).
 
Similarly the second point follows from {Theorem}~\ref{cor-5.1}(i). If we take  $\varphi_1\in \mathcal{G}_2$ such that  $\lim_{t\rightarrow 0} \varphi_1(t)=c>0$, then ${\sigma({0})}=-s$ for any $s<0$. This completes  the proof of the corollary. 	    
\end{proof}

\ignore{now extension to $\Fphia(\Omega)\hookrightarrow \Fphib(\Omega)$ and $\MFa(\Omega)\hookrightarrow \MFb(\Omega)$ by elementary embeddings and coincidences, only for the compactness case, I suggest; then comparison with Corollary~\ref{comp-class}}

 {The results we obtained on the continuity and compactness of embeddings within the generalised Besov-Morrey scale $\MB(\Omega)$ and within the scale of spaces  $\Bphi(\Omega)$, together with the existence of embeddings linking these Besov-type spaces to their  Triebel-Lizorkin-type counterparts, enable us to derive analogous results for the latter class of spaces. Recall that, due to \eqref{E=F},
\begin{equation} \label{E=FO}
\Fphi(\Omega)=\MF(\Omega).
\end{equation}
 For compactness we obtain the following consequence of Theorem~\ref{th-cont-BM}.

\begin{corollary}\label{comp-GTLM}
Let $s_i\in\rr$, $0<p_i<\infty$, $0<q_i\leq \infty$, and $\varphi_i\in {\mathcal G}_{p_i}$, for $i=1,2$,    assuming  that $\varphi_i$ satisfies \eqref{intc} when $q_i<\infty$.
We assume without loss of generality that $\varphi_1(1)=\varphi_2(1)=1$. 
Let $ \varrho=\min(1,\frac{p_1}{p_2})$ and $\displaystyle\alpha_j= \sup_{0\le\nu\le j}\frac{\varphi_2(2^{-\nu})}{\varphi_1(2^{-\nu})^\varrho}$, $j\in \mathbb{N}_0$. Then the embedding
\begin{equation} \label{EinE}
	{\mathcal E}^{s_1}_{\varphi_1,p_1,q_1} (\Omega) \hookrightarrow {\mathcal E}^{s_2}_{\varphi_2,p_2,q_2}(\Omega)
\end{equation}
 is compact if 
\begin{align}\label{cond-TLM}
	\left\{   2^{j(s_2-s_1)} \alpha_j \varphi_1(2^{-j})^{\varrho -1}\right\}_{j\in \no} & \in \ell_{\min(p_2,q_2)}.   
\end{align}
Likewise $\Fphia(\Omega)\hookrightarrow \Fphib(\Omega)$ is compact, if \eqref{cond-TLM} is satisfied.
 \end{corollary}	
 
 \remark{\label{Rmk5.8}
 Since the conditions of Theorem~\ref{cor-5.1} do not depend on $q$'s, and in view of the embeddings \eqref{elem-gen}, that theorem can be reformulated in exactly the same way for the spaces $\Fphi(\Omega)$. However, 
 the conditions established there for compactness are weaker than the ones obtained above in Corollary~\ref{comp-GTLM}. Indeed, if  \eqref{comp-1} is satisfied, then $\alpha_j \sim 1$, so that  $s_2<\sigma(s_1)$ implies \eqref{cond-TLM}. 
 On the other hand, if \eqref{comp-2}  holds, by choosing $\varepsilon>0$ and $s_3$ so that $s_2<s_3-\varepsilon <\overline{\sigma}(s_1)$, we obtain
% $$ 2^{j(s_2-s_1)} \alpha_j \varphi_1(2^{-j})^{\varrho -1} \le c 2^{j(s_2-s_1)} 2^{j(\sigma(s_1)-s_3)}  \varphi_1(2^{-j})^{\varrho-1}  \le  2^{j(s_2-s_3+\varepsilon)},  $$
 $\alpha_j\leq c\ 2^{j(\sigma(s_1)+\varepsilon-s_3)}$. Thus
   $$ 2^{j(s_2-s_1)} \alpha_j \varphi_1(2^{-j})^{\varrho -1} \le c 2^{j(s_2+ \sigma(s_1)-s_3+\varepsilon -s_1)}  \varphi_1(2^{-j})^{\varrho-1}  \le c',
   $$
since $s_2+ \sigma(s_1)-s_3+\varepsilon<\sigma(s_1)$ and \eqref{comp-1a}.
This once more implies  \eqref{cond-TLM}. \\
Note that the condition 
$$
\left\{   2^{j(s_2-s_1)} \alpha_j \varphi_1(2^{-j})^{\varrho -1}\right\}_{j\in \no}  \in \ell_{\infty}   
$$
is necessary for the embedding \eqref{EinE}. This follows from the chain of embeddings
$$
   \mathcal{N}^{s_1}_{\varphi_1, p_1,\min(p_1,q_1)}(\Omega) \hra \mathcal{E}^{s_1}_{\varphi_1,p_1,q_1}(\Omega)\hra \mathcal{E}^{s_2}_{\varphi_2,p_2,q_2}(\Omega) \hra \mathcal{N}^{s_2}_{\varphi_2, p_2,\infty}(\Omega).
$$
}

\remark{Analogously to the observation made in  Remark~\ref{n-n-class}, if $\varphi_i(t)\sim t^{\nd/u_i}$, $0<p_i\leq u_i<\infty$, and $\mathcal{A}^{s_i}_{\varphi_i,p_i,q_i}=\mathcal{E}^{s_i}_{\varphi_i,p_i,q_i}$, $i=1,2$, 
then Corollary~\ref{comp-GTLM}  reduces to Corollary~\ref{comp-class}(i), since $\alpha_j\sim 2^{j\nd (\frac{\varrho}{u_1}-\frac{1}{u_2})_+}$, $j\in\no$.  
Moreover,  using \eqref{E=FO} and setting $\frac{1}{u_i}=\frac{1}{p_i}-\tau_i$, $i=1,2$,  one recovers Corollary~\ref{comp-class}(ii).
}}

\bigskip

Finally we deal with generalised Morrey spaces $\M(\Omega)$. Here we use ideas from \cite{HSS-morrey} together with \eqref{E=Mf}. First we recall the situation on $\rd$.    In \cite[Section~12.1.2, Cor.~30, p.32]{FHS-MS-2} it is shown that for $0<p_i<\infty$, $\varphi_i\in\Gpx{p_i}$, $i=1,2$, with  $\lim_{t\to 0} \varphi_1(t)=0$, 
   \begin{equation} \label{emb-Mphi}
     {\mathcal M}_{\varphi_1, p_1}(\rd) \hookrightarrow {\mathcal M}_{\varphi_2, p_2}(\rd) \end{equation}
   {if, and only if,}
   \begin{equation}\label{emb-Mphi-a}
p_2\leq p_1 \quad\text{and 
there exists some $C>0$ such that for all $t>0$, $\varphi_1(t)\geq C \varphi_2(t)$}.
\end{equation}     
There is a forerunner in  \cite[Theorem~3.3]{GHLM17} for parameters $1\leq p_2\leq p_1<\infty$, without the additional assumption on $\varphi_1$.

\ignore{      there exists some $C>0$ such that for all $t>0$, $\varphi_1(t)\geq C \varphi_2(t)$. 
   The argument can be immediately extended to $0<p_2\leq p_1<\infty$, cf. \cite[Corollary~30]{FHS-MS-2}.
   
 \open{ \magenta{According to  \cite[Cor.~30]{FHS-MS-2} : for $0< p_i<\infty$, $\varphi_i\in \mathcal{G}_{p_i}$, $i=1,2$, assuming  $\lim_{t\to 0} \varphi_1(t)=0$; then $ {\mathcal M}_{\varphi_1, p_1}(\rd) \hookrightarrow {\mathcal M}_{\varphi_2, p_2}(\rd)$ if, and only if,  $p_2\leq p_1$ and 
there exists some $C>0$ such that for all $t>0$, $\varphi_1(t)\geq C \varphi_2(t)$.\\
Shall we change? If so, we shall also adapt Corollary~\ref{M-M}(i)}
 }   }
   
   \begin{remark}
     In case of $\varphi_i(t) =t^{\nd/u_i}$, $0<p_i\leq u_i<\infty$, $i=1,2$, it is well-known that
\begin{equation}\label{emb-Mu}
\Me(\rd) \hookrightarrow \Mz(\rd)\qquad\text{if, and only if,}\qquad
p_2\leq p_1\leq u_1=u_2, 
\end{equation}
cf. \cite{piccinini-1,rosenthal}. This can be observed from \eqref{emb-Mphi}, \eqref{emb-Mphi-a}  { since} 
\[
\varphi_1(t)\geq C \varphi_2(t) \quad\text{{means}}\quad t^{\frac{\nd}{u_1}} \geq C t^{\frac{\nd}{u_2}},\quad t>0,
\]
which results in $u_1=u_2$.
\end{remark}

Let us assume that  $\Omega\subset\rd$ is a bounded {$C^{\infty}$} domain, that is, in particular an $A$-type domain as dealt with in \cite{HSS-morrey}. Let $\M(\Omega)$ be defined by restriction, that is,
for $0<p<\infty$, $\varphi\in\Gp$,
\[
\M(\Omega) = \{f\in L_p^{\mathrm{loc}}(\Omega): \exists\ g\in \M(\rd): g\big|_{\Omega} = f\}
\]
quasi-normed by
\[
\|f| \M(\Omega)\| = \inf\{ \|g|\M(\rd)\|: g\in \M(\rd): g\big|_{\Omega} = f\}.
\] 
Then one can prove parallel to \cite[Theorem~2.3]{HSS-morrey} the following result.

\begin{lemma}\label{char-intr}
  Let $0<p<\infty$, $\varphi\in\Gp$, $\Omega\subset\rd$ a bounded {$C^{\infty}$}  domain. Then for all $f\in\M(\Omega)$,
  \begin{equation}\label{intr-norm}
  \|f|\M(\Omega)\| \sim \sup_{x\in\Omega, j\in\no} \frac{\varphi(2^{-j})}{|B(x,2^{-j})|^{1/p}} \left(\int_{\Omega\cap B(x, 2^{-j})} |f(y)|^p \dint y\right)^{1/p}.
  \end{equation}
\end{lemma}

\begin{remark}\label{rem-hoelder-MO}
Note that we thus always have $L_\infty(\Omega)\hra \M(\Omega) \hra L_p(\Omega)$. Moreover, $\lim_{t\to 0} \varphi(t)>0$ implies $\M(\Omega) = L_\infty(\Omega)$ due to Lebesgue's differentiation theorem. Hence, if $\lim_{t\to 0} \varphi_1(t)>0$, then
  \[
  \Ma(\Omega) = L_\infty(\Omega)\hookrightarrow \Mb(\Omega)
  \]
for arbitrary $0<p_i<\infty$, $i=1,2$.
\end{remark}

\begin{corollary}\label{M-M}
  Let $\Omega\subset\rd$ be a bounded {$C^{\infty}$}  domain. 
Let $0<p_i<\infty$ and $\varphi_i\in {\mathcal G}_{p_i}$, for $i=1,2$. 
We assume without loss of generality that $\varphi_1(1)=\varphi_2(1)=1$. Let $\lim_{t\to 0} \varphi_1(t)=0$.
\begin{enumerate}[{\bfseries\upshape(i)}]
\item
  Assume that $p_2\leq p_1$ and there exists some ${C>0}$ such that $\varphi_1(t)\geq C \varphi_2(t)$, $0<t\leq 1$. Then
  \begin{equation}\label{MM-Om}
  \Ma(\Omega) \hookrightarrow \Mb(\Omega).
  \end{equation}
\item
  If $1<p_i<\infty$, $i=1,2$, and $\Ma(\Omega) \hookrightarrow \Mb(\Omega)$, then $p_2\leq p_1$ and $\varphi_1(t)\geq C \varphi_2(t)$, $0<t\leq 1$.
\item
  The embedding  $\Ma(\Omega) \hookrightarrow \Mb(\Omega)$ is never compact.
  \end{enumerate}
  \end{corollary}

  \begin{proof}
    The sufficiency part (i) is a direct consequence of the corresponding result \eqref{emb-Mphi},  \eqref{emb-Mphi-a} on $\rd$ {(or  Lemma~\ref{char-intr}, using just H\"older's inequality and the assumptions on $\varphi_i$, $i=1,2$)}.

    \ignore{
      \magenta{\bf For (i) isn't it just by restriction? (SM)} In principle, yes, with extensions of $\varphi_i$ to $\real_+$, but I wanted to include the cases $p_i<1$, at least in part (i), and it seems that so far we only have the generalised Morrey space embedding for $p_i\geq 1$ in \cite{GHLM17}. (DDH) 
  }

%  \open{We may revise the argument a little once the result for the embedding $\MFa(\Omega)\hookrightarrow \MFb(\Omega)$ is formulated above.}
  
    As for (ii), we benefit from the counterpart of \eqref{E=Mf} for spaces on domains $\Omega$ (defined by restriction). {Then, as in Remark~\ref{Rmk5.8},  
   $$ 
   \left\{ \alpha_j \varphi_1(2^{-j})^{\varrho -1}\right\}_{j\in \no}  \in \ell_{\infty}.       
    $$}%
 % together with Theorem~\ref{th-cont-BM} and elementary embeddings.    
%This can be seen as follows. By elementary embeddings,
%    \[
%      \mathcal{N}^0_{\varphi_1, p_1,\min(p_1,2)}(\Omega) \hra \mathcal{E}^0_{\varphi_1,p_1,2}(\Omega)=\Ma(\Omega)\hra \Mb(\Omega)=\mathcal{E}^0_{\varphi_2,p_2,2}(\Omega) \hra \mathcal{N}^0_{\varphi_2, p_2,\infty}(\Omega),
%      \]
%      which results in 
%      \[
%        \sup_{j\in\no} \alpha_j \varphi_1(2^{-j})^{\varrho-1} < \infty,
%        \]
%in view of \eqref{cond2O}, note that $s_1=s_2=0$ and $q^\ast=\infty$. 
Since $\alpha_j\geq 1$, $j\in\no$, and $\varphi_1(t)\to 0$ for $t\to 0$, this implies $\varrho=1$, that is, $p_1\geq p_2$. Thus $\{\alpha_j\}_{j\in\no} \in\ell_\infty$ which leads to $\varphi_1(t)\geq C \ \varphi_2(t)$, $0<t\leq 1$.
Finally, concerning (iii), note that the compactness of \eqref{MM-Om} would imply the compactness of $L_\infty(\Omega)\hra L_{p_2}(\Omega)$ which is not true.
  \end{proof}

  \remark{\label{emb-p>1}
    If we restrict ourselves to $p_i>1$, $i=1,2$, then Corollary~\ref{M-M} can be written as
    \[
  \Ma(\Omega) \hookrightarrow \Mb(\Omega) \quad\text{if, and only if,}\quad p_1\geq p_2\quad\text{and}\ \varphi_1(t)\geq c\ \varphi_2(t), \ 0<t\leq 1,
  \]
that is, the complete counterpart of \cite[Section~12.1.2, Cor. 30, p.32]{FHS-MS-2}, as recalled in \eqref{emb-Mphi}, \eqref{emb-Mphi-a}. 
  We suppose this characterisation to be true for $p_i\in (0,1]$, $i=1,2$, too, but then our above argument (in (ii)) does not work any longer.\\
  If $\varphi_i(t)\sim t^{\nd/u_i}$, $0<p_i\leq u_i<\infty$, $i=1,2$, then we dealt in detail with such embeddings as above in our paper \cite{HSS-morrey} where a special focus was laid on the quality of the underlying domain in connection with the parameters $u_i$, $i=1,2$.
  }
  
  \begin{corollary}
    Let $\Omega\subset\rd$ be a bounded {$C^{\infty}$}  domain. 
Let $1<p,r<\infty$ and $\varphi\in \Gp$, with $\varphi(1)=1$. Assume that $\lim_{t\to 0} \varphi(t)=0$.
\begin{enumerate}[{\bfseries\upshape(i)}]
\item
Then 
  \begin{equation}
  \M(\Omega) \hookrightarrow L_r(\Omega) \quad\text{if, and only if,}\quad p\geq r.
  \end{equation}
\item
Then 
  \begin{equation}
L_r(\Omega)   \hookrightarrow \M(\Omega) \quad\text{if, and only if,}\quad p\leq r \quad \text{and}\quad \varphi(t)\leq c\ t^{\frac{d}{r}}, \ 0<t\leq 1.
  \end{equation}
\item
Neither the embedding  $\M(\Omega) \hookrightarrow L_r(\Omega)$  nor $L_r(\Omega)   \hookrightarrow \M(\Omega)$ is ever compact in this situation.
  \end{enumerate}
  \end{corollary}

This follows immediately from Corollary~\ref{M-M} and Remark~\ref{emb-p>1}, together with $\M(\Omega)=L_r(\Omega)$ for $p=r$ and $\varphi(t)\sim t^{d/r}$, $0<t\leq 1$. Note that in part (i) the first condition $p\geq r$ already implies the second one, $ \varphi(t)\geq c\ t^{\frac{d}{r}}$, $0<t\leq 1$, since $\varphi\in\Gp$.  So the above result can also be seen as some slightly stronger result than the embeddings mentioned in Remark~\ref{rem-hoelder-MO} above. In \cite{HSS-morrey} we dealt with the spaces $\M(\Omega)=\Mu(\Omega)$, $\varphi(t)\sim t^{d/u}$, $0<p\leq u<\infty$, only, where the counterpart reads as the optimal chain of embeddings $L_u(\Omega)\hookrightarrow \Mu(\Omega) \hookrightarrow L_p(\Omega)$. For the corresponding situation on $\rd$, cf. \cite[Remark~2.5]{hms24} and \cite[Remarks~5.16 and 5.21]{HLMS24}.

  \begin{example}
    Recall that the choice $\varphi(t)=t^{-\sigma} \chi_{(0,1)}(t)$, where  $-\frac{\nd}{p}\leq \sigma <0$, implies that $\M(\rd)$ coincides with the local Morrey spaces $\mathcal{L}^{\sigma}_p(\rd)$ introduced by Triebel, cf. \cite[Section~1.3.4]{Tri13}. For convenience, let us assume that $1<p_i<\infty$, $-\frac{\nd}{p_i}< \sigma_i <0$, $i=1,2$, $|\Omega|=1$, and the spaces $\mathcal{L}^{\sigma}_p(\Omega)$ be defined by restriction. Then Corollary~\ref{M-M}  leads to
    \[
      \mathcal{L}^{\sigma_1}_{p_1}(\Omega) \hookrightarrow \mathcal{L}^{\sigma_2}_{p_2}(\Omega) \quad\text{if, and only if,}\quad p_1\geq p_2\quad\text{and}\quad \sigma_1\geq \sigma_2 .
      \]
  \end{example}

%%%%%%%%%%%%%%%%%%%%%%%%%%%%%%%%%%%%
%%%%%%%%%%%%%%%%%%%%%%%%%%%%%%%%%%%%
\bibliographystyle{alpha}	

\begin{thebibliography}{NTS16}
	
	\bibitem{AGNS}
	A.~Akbulut, V.S.~Guliyev, T.~Noi, and Y.~Sawano, 
	Generalized Morrey spaces-revisited, 
	{\em Z. Anal. Anwend. } 36 (2017), 17--35. 
	
	%\bibitem{Al}
	%A.~Almeida,
	%\newblock Wavelet bases in generalized Besov spaces,
	%\newblock {\em J. Math. Anal. Appl.}  304 (2005) 198--211.


\bibitem{ElBaraka1}
A.~El Baraka.
\newblock An embedding theorem for {C}ampanato spaces.
\newblock {\em Electron. J. Differential Equations}, pages No. 66, 17 pp.
  (electronic), 2002.

\bibitem{ElBaraka2}
A.~El Baraka.
\newblock Function spaces of {BMO} and {C}ampanato type.
\newblock In {\em Proceedings of the 2002 {F}ez {C}onference on {P}artial
  {D}ifferential {E}quations}, volume~9 of {\em Electron. J. Differ. Equ.
  Conf.}, pages 109--115 (electronic). Southwest Texas State Univ., San Marcos,
  TX, 2002.

\bibitem{ElBaraka3}
A.~El Baraka.
\newblock Littlewood-{P}aley characterization for {C}ampanato spaces.
\newblock {\em J. Funct. Spaces Appl.}, 4(2):193--220, 2006.
        
%	\bibitem{bauer}
%	M.~Bauer,
%	Verallgemeinerte Morrey-R\"aume. Grundlagen, Eigenschaften und Beispiele (German),
%	Master's thesis, Friedrich Schiller University Jena, Germany, 2021.
	
%	\bibitem{BS}
%	C.~Bennett and R.~Sharpley,
%	Interpolation of operators,
%	Academic Press, Boston, 1988.
	
%	\bibitem{BGT}
%	N.H. Bingham, C.M., Goldie, J.L. Teugels, Regular Variation. Encyclopedia of Mathematics and its
% Applications. Cambridge University Press, Cambridge, 1989.

  \bibitem{CaH-Tri}
A.M. Caetano and D.D. Haroske.
\newblock An interview with {Hans Triebel}.
\newblock CIM Bulletin, No. 23, p. 13-17, December 2007.
\newblock Universidade de Coimbra, Portugal. {\em Reprinted in:} EMS
  Newsletter, December 2008, p. 37-40, EMS Publishing House, Z\"urich; {\em
  Chinese transl. in:} Mathematical Advance in Translation, 2010, vol. 29, no.
  1, pp. 45-49.
	
	%\bibitem{Dau2}
	%I.~Daubechies,
	%\newblock{\em Ten Lectures on Wavelets},
	%\newblock  {CBMS-NSF Regional Conference Series in Applied Mathematics}, vol.61, 
	%\newblock{SIAM, Philadelphia, 1992. }
	
	%\bibitem{CF06}
	%A.M. Caetano and W.~Farkas,
	%Local growth envelopes of {B}esov spaces of generalized smoothness, 
	%{\em Z. Anal. Anwendungen} 25 (2006), 265--298.
	
	
%	\bibitem{CH15}
%	A.M. Caetano and D.D.~Haroske,
%	{\em Embeddings of Besov spaces on fractal h-sets},
%	{\em Banach J. Math. Anal.} 9 (2015), no. 4, 259--295.
	
	
	%\bibitem{CM04}
	%A.M.~Caetano and S.D.~Moura,
	% {\em Local growth envelopes of spaces of generalized smoothness: the
		%  critical case},
	% {\em Math. Ineq. \& Appl.} 7 (2004), no.~4, 573--606. 
	
\bibitem{ET96}
D.E. Edmunds and H. Triebel, H,
{\em Function spaces, entropy numbers, differential operators},
Cambridge Tracts in Mathematics 120, Cambridge University Press, Cambridge
1996.	
	
	% \bibitem{FaLe}
	%W. Farkas and H.-G.~Leopold,
	%\newblock Characterisation of function spaces of generalised smoothness,
	%\newblock {\em Ann. Mat. Pura Appl.} 185(1) (2006) 1--62.


	\bibitem{FHS} 
	G.~Di Fazio, D.I.~Hakim,  and Y.~Sawano, 
	Elliptic equations with discontinuous	coefficients in generalized Morrey spaces,
	{\em Eur. J. Math.}  3 (2017),  728--762.
	
	
	\bibitem{FP}
	L.C.F.~Ferreira and M.~Postigo, 
	Global well-posedness and asymptotic behavior in Besov-Morrey spaces for chemotaxis-Navier-Stokes fluids,
	{\em  J. Math. Phys.}  60 (2019), no. 6, 061502, 19 pp. 


        
	%\bibitem{Gu}
	%V.S.~Guliyev, 
	%\newblock Integral operators on function spaces on the homogeneous groups
	%and on domains in Rn (Russian). 
	%\newblock Doctor degree dissertation, Mat. Inst. Steklov,  Moscow (1994)
\bibitem{GHS-21}
H.F. Gon\c{c}alves, D.D. Haroske, and L.~Skrzypczak.
\newblock {Compact embeddings of Besov-type and Triebel-Lizorkin-type spaces on bounded domains}.
%\newblock {\em submitted}; arXiv:2001.02046, 2020.
\newblock {\em Rev. Mat. Complut.}, 34:761--795, 2021.

        \bibitem{GHS-23}
H.F. Gon\c{c}alves, D.D. Haroske, and L.~Skrzypczak.
\newblock {Limiting embeddings of Besov-type and Triebel-Lizorkin-type spaces
  on bounded domains, and extension operator}.
\newblock {\em Ann. Mat. Pura Appl. (4)}, 202:2481--2516, 2023.
	
%	\bibitem{Graf}
%	L.~Grafakos, 
%	Classical Fourier Analysis, Second edition.
%	Springer, 2008.
	
	
	\bibitem{GHLM17}
	H.~Gunawan, D.I.~Hakim, K.M.~Limanta, and A.A.~Masta,
	Inclusion properties of generalized Morrey spaces,
	{\em Math. Nachr.} 290 (2017), no. 2-3, 332--340. 
	
	
%	\bibitem{hardy}
%	G.H. Hardy, J.E. Littlewood, and G.~P{\'o}lya, \emph{Inequalities}, 2nd ed.,
%	Cambridge Univ. Press, Cambridge, 1952.
	
	
%	\bibitem{HaHabil}
%	D.D.~Haroske,
%	Limiting Embeddings, Entropy Numbers and Envelopes in Function Spaces, 
%	Habi\-li\-ta\-tions\-schrift, Friedrich-Schiller-Universit\"at Jena, Germany, 2002.
	
	
%	\bibitem{Ha-crc}
%	D.D.~Haroske,
%	Envelopes and Sharp Embeddings of Function Spaces, volume 437  of Chapman \& Hall/CRC Research Notes in Mathematics,
%	Chapman \& Hall/CRC, Boca Raton, FL, 2007.
	
	\bibitem{HLMS2023}
	D.D. Haroske, H.-G. Leopold, S.D. Moura, and L.~Skrzypczak,
	Nuclear and Compact Embeddings in Function Spaces of Generalised Smoothness,
	{\em Anal. Math.} 49 (2023), 1007--1039.
	
	\bibitem{HL23} 
	D.D. Haroske and Z. Liu, 
	Generalized Besov-type and Triebel-Lizorkin-type spaces, 
	{\em Studia Math.}  273 (2023), no. 2, 161--199. 
	
	\bibitem{HLMS24}
D.D. Haroske, Z. Liu, S.D. Moura, and L. Skrzypczak,
Embeddings of generalised Morrey smoothness spaces,
 Acta. Math. Sin.-English Ser. 41 (2025), 413--456 .
	
%	\bibitem{HaSM3}
%	D.D.~Haroske and S.D.~Moura,
%	Some specific unboundedness property in Smoothness Morrey Spaces, 
%	The non-existence of growth envelopes in the subcritical case, 
%	{\em Acta Math. Sin. (Engl. Ser.)}, 32 (2016), 137--152.
	
	\bibitem{HMSS-morrey}
	D.D. Haroske, S.D. Moura, C.~Schneider, and L.~Skrzypczak,
	{Unboundedness properties of Smoothness Morrey spaces of regular
		distributions on domains},
	{\em Sci. China Math.}, 60  (2017), 2349--2376.
	
%	\bibitem{hms16}
%	D.D. Haroske, S.D. Moura, and L.~Skrzypczak,
%	{Smoothness Morrey Spaces of regular distributions, and some unboundedness property},
%	{\em Nonlinear Anal.} 139 (2016), 218--244.
	
%	\bibitem{hms20}
%	D.D. Haroske, S.D. Moura, and L.~Skrzypczak,
%	{Some Embeddings of Morrey Spaces with Critical Smoothness},
%	{\em  J. Fourier Anal. Appl.} 26 (2020), 50.

 
	\bibitem{hms22}
	D.D. Haroske, S. Moura,  and L. Skrzypczak,
	Wavelet decomposition and  embeddings of generalised Besov-Morrey  spaces,
	{\em Nonlinear Anal.} 214 (2022), 112590.
	
	\bibitem{hms24}
	D.D. Haroske, S. Moura,  and L. Skrzypczak,
	On a bridge connecting Lebesgue and Morrey spaces in view of their growth properties,
	{\em Anal. Appl.}    22 (2024), no. 4,  751--790.
	
	
	\bibitem{HSS-morrey}
	D.D. Haroske, C.~Schneider, and L.~Skrzypczak,
	{Morrey spaces on domains: Different approaches and growth  envelopes},
	{\em J. Geom. Anal.}  28 (2018),  817--841.
	
	%\bibitem{HST}
	%	D.D.~Haroske, P.~Skandera, H. Triebel, 
	%	\newblock{An approach to wavelet isomorphisms of function spaces via atomic representations,}
	%	\newblock{{\em J. Fourier Anal. Appl. }24 (2018), 830-871}. 
	
%	\bibitem{hs12}
%	D.D. Haroske and L. Skrzypczak,
%	Continuous embeddings of Besov-Morrey function spaces,
%	{\em Acta Math. Sin. (Engl. Ser.) } 28 (2012), 1307--1328.
	
	\bibitem{hs13}
	D.D. Haroske and  L. Skrzypczak,  
	Embeddings of Besov-Morrey spaces on	bounded domains,
	{\em Studia Math.} 218 (2013), 119--144.
	
	
	\bibitem{hs14}
	D.D.~Haroske and L. Skrzypczak, 
	\newblock On Sobolev and Franke-Jawerth embeddings of smoothness Morrey spaces,
	\newblock {\em Rev. Mat. Complut.}  27 (2014), no. 2,  541--573. 
	
	\bibitem{hs20}
D.D. Haroske and L.~Skrzypczak.
\newblock {Entropy numbers of compact embeddings of Smoothness Morrey spaces on
  bounded domains}.
\newblock {\em J. Approx. Theory}, 256:105424, 2020.

\bibitem{hs24}
D.D. Haroske and L.~Skrzypczak.
\newblock {Nuclear embeddings of Morrey sequence spaces and smoothness Morrey
  spaces}.
\newblock {\em Bull. Malays. Math. Sci. Soc.}, 47:111, 2024.

%\bibitem{ht1994a}
%D.D. Haroske and  H. Triebel,  
%Entropy numbers in weighted function spaces and eigenvalue
%distributions of some degenerate pseudodifferential operators. {I}
%\newblock {\em  Math. Nachr.} 167 (1994),  131--156.


%\bibitem{ht1994b}
%D.D. Haroske and  H. Triebel,  
%Entropy numbers in weighted function spaces and eigenvalue
%distributions of some degenerate pseudodifferential operators. {II}
%\newblock {\em  Math. Nachr.} 168 (1994),  109--137.



\bibitem{ht2021}
D.D. Haroske and  H. Triebel,  
Morrey-smoothness spaces: A new approach,
{\em Sci. China Math.}  66 (2023), 1301--1358.
	
	
	%	\bibitem{HW96}
	%	\newblock{E.~Hern{\'a}ndez, G.~Weiss,} 
	%	\newblock{\em A First Course on Wavelets},
	%	\newblock{CRC Press, Boca Raton,  1996.}
	
	\bibitem{IN19}  
	M.~Izuki and  T.~Noi,
	Generalized Besov-Morrey spaces and generalized Triebel-Lizorkin-Morrey spaces on domains, 
	{\em Math. Nachr.}  292 (2019), 2212--2251.
	
	
%	\bibitem{JSS2019}  
%	J. JimÃ©nez-Garrido, J. Sanz, and G. Schindl,  
%	Indices of O-regular variation for weight functions and weight sequences, 
%	{\em Rev. R. Acad. Cienc. Exactas Fis. Nat. Ser. A Mat.}  113 (2019), no.4, 3659--3697.
	
	
	
	\bibitem{KMR}
	V.~Kokilashvili, A.~Meskhi, and  H.~Rafeiro, 
	Estimates for nondivergence elliptic equations with VMO coefficients in generalized grand Morrey spaces, 
	{\em Complex Var. Elliptic Equ.} 59 (2014), no. 8, 1169--1184.
	
	
	\bibitem{KY}
	H.~Kozono and M.~Yamazaki,
	Semilinear heat equations and the {N}avier-{S}tokes equation with distributions in new function spaces as initial data,
	{\em Comm. Partial Differential Equations} 19 (1994), 959--1014.
	
	
	\bibitem{KNS}
	K.~Kurata, S.~Nishigaki, and S.~Sugano, 
	Boundedness of integral operators on generalized Morrey spaces and its application to Schr\"odinger operators, 
	{\em Proc. Amer. Math. Soc.} 128 (2000), 1125--1134.
	
	%
	%\bibitem{Landau}
	%E.~Landau, 
	%\"Uber einen Konvergenzsatz, {\em  G\"ottingen Nachr.} (1907), 25-27.  
	%
	
	%	\bibitem{LSUYY13} 
	%	Y.~Liang, Y. Sawano, T. Ullrich, D.~Yang,  W.~Yuan, 
	%	A new framework for generalized Besov-type and Triebel-Lizorkin-type spaces. {\em Dissertationes Math.} 489 (2013), 114 pp.
	%
	
	\bibitem{Maz03}
	A.L.~Mazzucato,
	Besov-{M}orrey spaces: function space theory and applications to non-linear {PDE},
	{\em Trans. Amer. Math. Soc. } 355 (2003), 1297--1364. 
	
	
	\bibitem{mi}
	T. Mizuhara, 
	Boundedness of some classical operators on generalized Morrey spaces,
	Harmonic analysis (Sendai, 1990), 183-189, ICM-90 Satell. Conf. Proc., pringer, Tokyo, 1991.
	
	\bibitem{Nak00}
	E.~Nakai, 
	\newblock A characterization of pointwise multipliers on the Morrey spaces,
	\newblock {\em Sci. Math.} 3 (2000),  no. 3, 445--454.
	
	\bibitem{Nak94}
	E.~Nakai, 
	Hardy-Littlewood maximal operator, singular integral operators and the Riesz potentials on generalized Morrey spaces,
	{\em Math. Nachr.} 166 (1994),  95--103.
	
	
	\bibitem{NNS16}
	S.~Nakamura, T.~Noi, and Y. Sawano,
	Generalized Morrey spaces and trace operator,
	{\em Sci. China Math.} 59 (2016), no. 2, 281--336.
	
	
	
	\bibitem{piccinini-1}
	L.C. Piccinini,
	{Inclusioni tra spazi di Morrey},
	{\em Boll. Un. Mat. It. (4)}, 2 (1969), 95--99.
	
	%\bibitem{pietsch}
	%  A.~Pietsch, 
	% History of Banach Spaces and Linear Operators, Birkh\"auser Verlag 2007 . 	
	
	\bibitem{rosenthal}
	M.~Rosenthal,
	{Morrey-R\"aume aus der Sicht der harmonischen Analysis},
	Master's thesis, Friedrich-Schiller-Universit\"at Jena, Germany,	2009.
	
	%	\bibitem{MR-1}
	%	M.~Rosenthal,
	%	\newblock {Local means, wavelet bases, representations, and isomorphisms in Besov-Morrey and Triebel-Lizorkin-Morrey spaces.}
	%	\newblock {\em Math. Nachr.} 286 (2013), no. 1, 59--87.
	
	
%	\bibitem{RS}
%	T.~Runst and  W.~Sickel,
%	Sobolev Spaces of fractional order, Nemytskij Operators, and Nonlinear Partial Differential Equations,
%	Walter de Gruyter, Berlin, New York, 1996.
	
	
\bibitem{saw08} Y. Sawano, Wavelet characterizations of Besov-Morrey and Triebel-Lizorkin-Morrey spaces, Funct. Approx. Comment. Math. 38 (2008), 93--107.
	
	
	\bibitem{Saw18}
	Y. Sawano, 
	A thought on generalized Morrey spaces,
	{\em   J. Indones. Math. Soc.} 25 (2019), no. 3, 210--281.
	
	
	\bibitem{FHS-MS-1}
	Y.~Sawano, G.~Di~Fazio, and D.I. Hakim,
	{\em {M}orrey spaces. {Introduction and Applications to Integral
			Operators and PDE's}. Vol. {I}}.
	Monographs and Research Notes in Mathematics. Chapman \& Hall CRC
	Press, Boca Raton, FL, 2020.
	
	
	\bibitem{FHS-MS-2}
	Y.~Sawano, G.~Di~Fazio, and D.I. Hakim,
	{\em {M}orrey spaces. {Introduction and Applications to Integral
			Operators and PDE's}. Vol. {II}}.
	Monographs and Research Notes in Mathematics. Chapman \& Hall CRC
	Press, Boca Raton, FL, 2020.
	

\bibitem{SHG-15} Y. Sawano, D.I. Hakim and H. Gunawan, Non-smooth atomic decomposition for generalized
              {O}rlicz-{M}orrey spaces, Math. Nachr. {\bfseries 288}, no. 14-15, 1741--1775 (2015)
  
        
%	\bibitem{SeeT}
%	A.~Seeger and W. Trebels, 
%	\newblock Low regularity classes and entropy numbers
%	\newblock {\em Arch. Math.} 92 (2009), 147--157.
	
	
	\bibitem{s011}
	W.~Sickel, 
	\newblock Smoothness spaces related to Morrey spaces -- a survey. I, 
	\newblock {\em Eurasian Math. J.} 3 (2012), 110--149.
	
	
	\bibitem{s011a}
	W.~Sickel, 
	\newblock Smoothness spaces related to Morrey spaces -- a survey. II, 
	\newblock {\em Eurasian Math.} J. 4 (2013), 82--124.
	
	
%	\bibitem{ST}
%	W.~Sickel, H.~Triebel,
%	\newblock  H\"older inequalities and sharp embeddings in function spaces of {$B^s_{p,q}$} and {$F^s_{p,q}$} type, 
%	\newblock {\em  Z. Anal. Anwendungen } 14 (1995), 105--140.
	
	\bibitem{TX}
L.~Tang and J.~Xu.
\newblock Some properties of {M}orrey type {B}esov-{T}riebel spaces.
\newblock {\em Math. Nachr.}, 278(7-8):904--917, 2005.

		
	\bibitem{T83}
	H.~Triebel,
	Theory of function spaces,
	Birkh\"auser, Basel, 1983.


\bibitem{T92} H. Triebel,
Theory of function spaces II,
Birkh\"auser, Basel, 1992.


%\bibitem{T93}
%H.~Triebel.
%\newblock Approximation numbers and entropy numbers of embeddings of fractional
%  {B}esov--{S}obolev spaces in {O}rlicz spaces.
%\newblock {\em Proc. London Math. Soc. (3)}, 66:589--618, 1993.

%\bibitem{T-func}
%	H.~Triebel,
%	The Structure of Functions,
%	Birkh\"auser, Basel, 2001.

\bibitem{T97}
H. Triebel,
\newblock \textit{Fractals and Spectra: Related to Fourier Analysis and Function Spaces}.
\newblock Birkh\"auser, Basel, 1997.
	
\bibitem{T06} H. Triebel, Theory of Function Spaces, III,
  Birkh\"auser, Basel, 2006.

  \bibitem{Tri08} 
	H.~Triebel,
	\newblock{\em  Function Spaces and Wavelets on Domains}, EMS Tracts in Mathematics 7, 
	\newblock European Mathematical Society, 2008.
	
	\bibitem{Tri11}
	H.~Triebel, 
	\newblock Morrey-Campanato spaces and their smooth relatives,
	\newblock  unpublished notes, 2011.

\bibitem{Tri13} 
	H.~Triebel,
	\newblock{\em Local Function Spaces, Heat and Navier-Stokes Equations}, EMS Tracts in Mathematics 20, 
	\newblock European Mathematical Society, 2013.
	%
	
      \bibitem{Tri14}
	H.~Triebel,
	\newblock{\em Hybrid function spaces, heat and Navier-Stokes equations}, EMS Tracts in Mathematics 24, 
	\newblock European Mathematical Society, 2014.

         \bibitem{Tri17}
	H.~Triebel,
	\newblock{\em PDE Models for Chemotaxis and Hydrodynamics in Supercritical Function Spaces}, EMS Series of Lectures in Mathematics, EMS Publishing House, Z\"urich, 2017.
        
 
\bibitem{T20} H. Triebel, Theory of function spaces, IV. Birkh\"{a}user, Basel, 2020.

        
	
	\bibitem{WNTZ}
	N.~Wei, P.C.~Niu, S.~Tang, M.~Zhu, 
	\newblock Estimates in generalized Morrey spaces for nondivergence degenerate elliptic operators with discontinuous coefficients,
	\newblock {\em Rev. R. Acad. Cienc. Exactas Fis. Nat. Ser. A Math. RACSAM} 106 (2012),  1--33.
	
	%	\bibitem {Woj} 
	%	{P.~Wojtaszczyk,}
	%	\newblock {\em  A Mathematical Introduction to Wavelets,}
	%	\newblock{Cambridge Univ. Press, Cambridge, 1997}.
	


\bibitem{YZY-BJMA15}
D. Yang, C. Zhuo, and W. Yuan,
{Triebel-{L}izorkin type spaces with variable exponents}, Banach
  J. Math. Anal. \textbf{9}, no.~4, 146--202 (2015)


        
	\bibitem{YFS}
	M.~Yang, Z.~Fu, and  J.~Sun,
	\newblock Existence and large time behaviour to coupled chemotaxis-fluid equations in Besov-Morrey spaces,
	\newblock {\em J. Differential Equations} 266 (2019), no. 9,  5867--5894.
	
	
	\bibitem{YSY10}
	W.~Yuan, W.~Sickel, and D.~Yang,
	\newblock {\em Morrey and Campanato meet Besov, Lizorkin and Triebel},
	\newblock Lecture Notes in Mathematics 2005, Springer-Verlag, Berlin, 2010.
	
	
	\bibitem{ZJSZ}
	L.~Zhang, Y.~Jiang, Y.~Sheng, J.~Zhou, 
	\newblock Parabolic equations with VMO coefficients in generalized Morrey spaces, 
	\newblock {\em Acta Math. Sin.} 26 (2010), no. 1, 117--130.
	
\end{thebibliography}
\def\cprime{$'$}

%%%%%%%%%%%%%%%%%%%%%%%%%%%%%%%%%%%%%%%%%%%%%%%%%%%%%%%%%%%%%%%%%%%%%%%%%%
\bigskip~

{\small
	\noindent%\begin{minipage}[t]{0.3\textwidth}
	Dorothee D. Haroske\\
	Institute of Mathematics \\
	Friedrich Schiller University Jena\\
	07737 Jena\\
	Germany\\
	%\vfill
	{\tt dorothee.haroske@uni-jena.de}\\[4ex]
	\noindent%\begin{minipage}[t]{0.3\textwidth}
	Susana D. Moura\\
	University of Coimbra\\
	CMUC, Department of Mathematics\\
	EC Santa Cruz\\
	3001-501 Coimbra\\
	Portugal\\
	{\tt smpsd@mat.uc.pt}\\[4ex]
	% 
	%\end{minipage}\hfill\begin{minipage}[t]{0.3\textwidth}
	Leszek Skrzypczak\\
	Faculty of Mathematics \& Computer Science\\
	Adam  Mickiewicz University\\
	ul. Uniwersytetu Pozna\'nskiego 4\\
	61-614 Pozna\'n\\
	Poland\\
	{\tt lskrzyp@amu.edu.pl}
}

\end{document}